\documentclass{amsart}
\usepackage{amssymb,latexsym,bbm}
\usepackage{graphicx}
\usepackage{color}
\usepackage{subfigure}
\usepackage{multirow}
\usepackage{mathptmx}
\usepackage{url,bm}
\usepackage{enumerate}
\usepackage{float}
\usepackage[ruled,vlined,linesnumbered,noresetcount]{algorithm2e}

\makeatletter
\newcommand{\AlgoResetCount}{\renewcommand{\@ResetCounterIfNeeded}{\setcounter{AlgoLine}{0}}}
\newcommand{\AlgoNoResetCount}{\renewcommand{\@ResetCounterIfNeeded}{}}
\newcounter{AlgoSavedLineCount}

\SetKwInOut{Parameter}{Parameters}
\makeatother

\newtheorem{theorem}{Theorem}
\newtheorem{assumption}{Assumption}
\newtheorem{definition}{Definition}
\newtheorem{lemma}{Lemma}
\newtheorem{corollary}{Corollary}
\newtheorem{proposition}{Proposition}
\newtheorem{remark}{Remark}

\title[Heat kernel reconstruction]{Spectral convergence of graph Laplacian and Heat kernel reconstruction in $L^\infty$ from random samples}

\author{David B Dunson}
\address{David B Dunson\\
Department of Statistical Science\\
Duke University}
\email{dunson@duke.edu}

\author{Hau-Tieng~Wu}
\address{Hau-Tieng Wu\\
Departments of Mathematics and Department of Statistical Science\\
Duke University}
\email{hauwu@math.duke.edu}

\author{Nan~WU}
\address{Nan Wu\\
Department of Mathematics\\
Duke University}
\email{nan.wu@duke.edu}

\begin{document}


\maketitle

\begin{abstract}	
In the manifold setting, we provide a series of spectral convergence results quantifying how the eigenvectors and eigenvalues of the graph Laplacian converge to the eigenfunctions and eigenvalues of the Laplace-Beltrami operator in the $L^\infty$ sense. Based on these results, convergence of the proposed heat kernel approximation algorithm, as well as the convergence rate, to the exact heat kernel is guaranteed. To our knowledge, this is the first work exploring the spectral convergence in the $L^\infty$ sense and providing a numerical heat kernel reconstruction from the point cloud with theoretical guarantees.
\end{abstract}   

\section{Introduction}
The graph Laplacian (GL) is a tool developed in the {spectral} graph theory community \cite{Fan:1996}, and it has been widely applied in the machine learning society to design various algorithms. It has a natural interpretation as the infinitesimal generator of a random walk on a point cloud, or the associated graph. 

Given data points sampled from a manifold, the GL is defined based on the affinity matrix constructed from determining how similar each pair of points are. It has been well {known} that under the manifold setup, the GL asymptotically converges to the Laplace-Beltrami operator. In the past decades, in addition to a rich methodological literature on the GL and its application, there is a lot of theoretical literature describing how the GL converges to the Laplace-Beltrami operator. {We briefly review these results.
The pointwise convergence of the GL to the Laplace-Beltrami operator has been discussed in several places, like \cite{coifman2006diffusion,singer2006graph,hein2006uniform}, where the convergence rate is provided. The pointwise convergence result shows that when the GL is applied to the discretization of a test function over the data points, it converges to the value of the Laplacian of the test function at each point. In \cite{coifman2006diffusion}, the pointwise convergence is carried out when the sampling is non-uniform from a compact manifold with boundary, where the affinity is constructed by the $\alpha$-normalization with the Gaussian kernel (see \eqref{W matrix} below);
In \cite{singer2006graph}, the convergence is carried out when the data points are uniformly i.i.d. sampled from a compact manifold without boundary, where the affinity matrix is constructed from the Gaussian kernel;
In \cite{hein2006uniform}, the convergence is carried out when the i.i.d. sampling might be non-uniform from a compact manifold with boundary, where the affinity matrix is constructed from a compactly supported kernel.
However, a pointwise convergence result might not guarantee a spectral convergence.
To this end, there have been several results about the $L^2$ spectral convergence of GL on a compact manifold without boundary, with and without convergence rate. 
To our knowledge, the first result in this direction should be \cite{belkin2007convergence}, where the authors prove the $L^2$ spectral convergence  to the Laplace-Beltrami operator when the point cloud is i.i.d. uniformly sampled from a compact manifold without boundary, where the affinity matrix is constructed from the Gaussian kernel. In this work, the convergence rate is not provided.
By using the variational method over the Dirichlet form, an $L^2$ spectral convergence to the Laplace-Beltrami operator is shown in \cite{burago2015graph} with a convergence rate, where the GL is constructed under a deterministic setting and the manifold is compact without boundary. 
This variational method is further extended in \cite{trillos2018error, calder2019improved} to study the same problem under a probabilistic setting, where the optimal transport technique is applied. In these papers, the authors prove the $L^2$ spectral convergence to a weighted Laplace-Beltrami operator that depends on the density function with a rate when the point cloud is sampled i.i.d. (non-)uniformly from a compact manifold without boundary and the affinity matrix is constructed from a compactly supported kernel.  
When the affinity matrix is constructed by renormalizing the Gaussian kernel, the authors in \cite{wormell2020spectral} prove an $L^\infty$ spectral convergence with a rate to the  Laplace-Beltrami operator when the point cloud is sampled (non-)uniformly i.i.d. on flat tori of any dimensions.  
Based on the results in \cite{trillos2018error, calder2019improved}, the authors in \cite{calder2020lipschitz} show an $L^\infty$ spectral convergence with a rate to a weighted Laplace-Beltrami operator that depends on the density function when the point cloud is sampled i.i.d. (non-)uniformly from a compact manifold without boundary and the affinity matrix is constructed from a compactly supported kernel. 
We mention that the spectral convergence of the graph connection Laplacian is proved in \cite{singer2016spectral}, while the rate is not provided.}

Recall that the Laplace-Beltrami operator is directly related to the heat kernel. Specifically, the heat kernel can be expressed as a series sum in terms of the eigenpairs of the Laplace-Beltrami operator. Therefore, if we can reconstruct the eigenvalues and eigenfunctions of the Laplace-Beltrami operator, we may recover the heat kernel. While there is a rich theoretical literature on the GL, however, to the best of our knowledge, our work is the first to use GL in approximating the heat kernel over an unknown manifold from a point cloud of data points. To achieve this purpose, we are also the first team to provide the spectral convergence of the eigenvectors of the GL to the eigenfunctions of the Laplace-Beltrami operator in the $L^\infty$ sense with the convergence rate.  

In Section 2, we provide background on the heat kernel on manifolds. In Section 3.1, we recall the GL. Section 3.2 contains our main theoretical results.  We show the spectral convergence rate of the GL to the Laplace-Beltrami operator of the manifold in the $L^\infty$ sense. We also provide the convergence rate from the heat kernel we recovered to the actual heat kernel.  Section 4 consists of some fundamental facts to prove the main theorems. The proofs of the main theorems are organized in the Section 5 and Section 6. {The proofs of some technical results are postponed to the Appendix. In the whole paper, we consider a compact and smooth $d$-dim manifold $M$, and use $\| \cdot \|_2:=\|\cdot\|_{L^2(M)}$ be the $L^2$ norm, $\|\cdot \|_\infty:=\|\cdot\|_{L^\infty(M)}$ be the $L^\infty$ norm, and let $\| \cdot \|$ be the operator norm in the relevant Banach space.}

\section{Background knowledge about heat kernel}
In this section, we recall the heat kernel of a manifold. 
Let $M$ be a $d$ dimensional smooth closed Riemannian manifold with the Riemannian metric $g$.
The heat kernel, $\mathsf{H}:M\times M\times[0,\infty)\to \mathbb{R}$, is the fundamental solution of the heat equation; namely, for any $x, x' \in M$,
\begin{align}
\frac{\partial \mathsf{H}(x,x',t)}{\partial t} =\Delta_{x} \mathsf{H}(x,x',t), \nonumber  
\end{align}
where $\Delta_{x}$ is the Laplace-Beltrami operator of the manifold $M$ acting on the first parameter of $\mathsf{H}$ and $t>0$ is the diffusion time. The heat kernel satisfies the initial condition 
$\lim_{t \rightarrow 0} \mathsf{H}(x,x',t) = {\delta_{x'}(x)},$
where {$\delta_{x'}$} is a delta measure {supported at $x'\in M$}, and the limit is understood in the distributional sense. 
When the manifold is closed, the heat kernel on the manifold can be related to the geodesic distance by Varadhan's formula \cite{varadhan1967behavior}:
\begin{align}
\lim_{t \rightarrow 0}-4t \log \mathsf{H}(x,x',t)=d(x,x')^2, \nonumber 
\end{align}
where $d(x,x')$ is the geodesic distance between $x$ and $x'$. 
Take the special case when $M=\mathbb{R}^d$ as an example. In this case, it is well known that the heat kernel is the Gaussian kernel: 
\begin{align}
\mathsf{H}(x,x',t)=(4\pi t)^{-d/2} e^{-\frac{|x-x'|^2}{4t}}. \nonumber 
\end{align}
If we understand $|x-x'|$ as the geodesic distance between $x$ and $x'$ in $\mathbb{R}^d$, then based on the $\mathbb{R}^d$ case and Varadhan's formula, one may expect to use $(4\pi t)^{-d/2} e^{\frac{d(x,x')^2}{4t}}$ to {approximate the heat kernel on a manifold and use the methods like \cite{meng2008improving,singer2012vector,li2019geodesic,malik2019connecting} to estimate the geodesic distance.} Unfortunately, such an approximation can be quite inaccurate.
{When $d(x,x')$ and $t$ are small,  the heat kernel has the following expansion \cite{rosenberg1997laplacian},
\begin{align}
\mathsf{H}(x,y,t)=(4\pi t)^{-d/2} e^{-\frac{d(x,y)^2}{4t}}(u_0+t u_1+O(t^2))\,, \nonumber 
\end{align}}
{where $u_0(x,x')=\frac{1}{\sqrt{\det\big(d(\exp_x)_{\exp^{-1}_x(x')}\big)} }$ and $\exp_x$ is the exponential map at $x$. We mention that $u_0(x,x')=1+O(d(x,x')^2)$}, where the implied constant in $O(d(x,x')^2)$ depends on the {Ricci} curvature of the manifold. Also, $u_1$ is a continuous function of $x$ and $x'$ that only depends on the manifold. {Such an expansion shows that using $(4\pi t)^{-d/2} e^{\frac{d(x,x')^2}{4t}}$ to approximate the heat kernel leads to a large error when $x$ and $x'$ are close, even if we know the geodesic distance.}

\section{Graph Laplacian and heat kernel approximation}

The graph Laplacian (GL) is a fundamental tool in {the spectral} graph theory \cite{Fan:1996}. {In this section, we start with defining the GL. We then show how to apply the kernel normalized GL to approximate the heat kernel of the manifold.} Moreover, we provide a theoretical justification of the convergence rate of the approximation.

\subsection{Graph Laplacian}

Take a dataset $\mathcal{X}:=\{x_i\}_{i=1}^n \subset \mathbb{R}^D$. {Consider} a kernel function
$k_{\epsilon}(x,x')=\exp\Big(-\frac{\|{x-x'}\|^2_{\mathbb{R}^D}}{4\epsilon^2}\Big)$, where $\epsilon>0$ is the bandwidth. We mention that {in practice} we can {consider} a more general kernel, but to simplify the discussion, we focus on this Gaussian-like kernel.

{Take $\alpha\in [0,1]$.} We construct an $n \times n$  affinity matrix $W$ as
\begin{align}\label{W matrix}
W^{(\alpha)}_{ij}&\,:=\frac{k_{\epsilon}(x_i,x_j)}{q_{\epsilon}(x_i)^\alpha q_{\epsilon}(x_j)^\alpha}\,,
\end{align}
where $q_{\epsilon}(x):=\sum_{i=1}^{n} k_{\epsilon}(x,x_i)$ {for $x\in \mathbb{R}^D$}. This affinity matrix {contains affinities among pairs of points after} {\em $\alpha$-normalization} \cite{coifman2006diffusion}.
Note that the term $q_{\epsilon}(x)$ is {related to the kernel density estimation (KDE)} evaluated at $x$. The {normalization of the kernel $k_{\epsilon}(x_i,x_j)$ by $q_{\epsilon}(x_i)q_\epsilon(x_j)$ is aiming to} remove the impact of the non-uniform sampling density. {By defining} an $n \times n$ diagonal matrix $D$ as 
\begin{align}\label{D matrix}
D^{(\alpha)}_{ii}=\sum_{j=1}^{n} W^{(\alpha)}_{ij},
\end{align}
{where $i=1,\ldots,n$,}
we define the row stochastic transition matrix as 
\begin{align}\label{L matrix}
A^{(\alpha)}=D^{(\alpha)-1}W^{(\alpha)}\,.
\end{align}
Our main quantity of interest is the kernel normalized GL matrix, which is defined as 
\begin{equation}
L^{(\alpha)}:=\frac{A^{(\alpha)}-I}{\epsilon^2}\,.
\end{equation}
 {In this work, we have interest in two special cases; $\alpha=0$ and $\alpha=1$.  When $\alpha=1$, we call $L:=L^{(1)}$ the {\em kernel normalized GL}, and $W:=W^{(1)}$, $D:=D^{(1)}$ and $A:=A^{(1)}$; when $\alpha=0$, we call $L^{un}:=L^{(0)}$ the {\em unnormalized} GL, and $W^{un}:=W^{(0)}$, $D^{un}:=D^{(0)}$ and $A^{un}:=A^{(0)}$.}

{\begin{remark}
We shall comment on the nomination of $L^{un}=L^{(0)}$ as the unnormalized GL.  Sometimes, the name ``unnormalized GL'' refers to $D^{(0)}-W^{(0)}$, while $I-A^{(0)}$ is called the random walk GL. However, since $A^{(1)}$ is also a transition matrix associated with a random walk, we decided to call $L^{un}$ the unnormalized GL for the sake of comparing with $L=L^{(1)}$ in which the kernel is normalized as in \eqref{W matrix}. 
\end{remark}
}
Note that $L$ is a linear {translation} and scaling of $A$, so on the linear algebra level, we only need to study $A$ {if we are interested in the spectral structure of $L$}. 
Since $A$ is similar to $\tilde{A} = D^{-1/2}W D^{-1/2}$, $\tilde{A}$ is diagonalizable. The eigenvalues of $A$ and $\tilde{A}$ are the same, with the smallest eigenvalue {bounded below by} zero and all the values falling in $[0,1]$ since the chosen Gaussian kernel is positive definite {by the Bochner's theorem}. { Hence, if $\mu$ is an eigenvalue of $-L$, then $0 \leq \mu \leq \frac{1}{{\epsilon^2}}$ and the smallest eigenvalue of $-L$ {is bounded from below by $0$}. The same {discussion holds} for $L^{un}$.}

\subsection{{Spectral convergence of the} kernel normalized Graph Laplacian and the heat kernel reconstruction}
Now, we discuss how to estimate the heat kernel by the {kernel normalized GL} when the data are sampled from a closed Riemannian manifold $M$ embedded in $\mathbb{R}^D$. 

\begin{assumption}\label{assumption DM}
Let $M$ be a $d$-dimensional smooth, closed and connected Riemannian manifold isometrically embedded in $\mathbb{R}^D$ through $\iota:M\to \mathbb{R}^D$. Suppose that $\mathsf{p}$ is a smooth probability density function {(p.d.f.)} on the manifold $M$. We assume that $\mathsf{p}$ is bounded from below by $\mathsf{p}_{m}>0$. Suppose the point cloud $\mathcal{X}:=\{x_1 \cdots, x_n\}$ is independently and identically (i.i.d.) sampled following the density function $\mathsf{p}$; that is, we have $\mathcal{X}\subset M$.
{When $\mathsf{p}$ is a constant function, that is, $\mathsf{p}=\frac{1}{\texttt{Vol}(M)}$, we say that the sampling is {\em uniform}; otherwise it is {\em non-uniform}.}

\end{assumption}

Let $\Delta$ be the Laplace-Beltrami operator of $M$.  Let ${\sigma(-\Delta):=}\{\lambda_i\}_{i=0}^\infty$ be the spectrum of $-\Delta$. By standard elliptic theory, we have $0=\lambda_0 < \lambda_1\leq \lambda_2 \leq  \cdots$ and each eigenvalue has a finite multiplicity. Denote $\phi_i$ {to be} the corresponding eigenfunction normalized in $L^2(M)$; that is, for each $i\in \mathbb{N}$, we have $\Delta \phi_i =-\lambda_i \phi_i$.
Denote $\mu_{i,n,\epsilon}$ to be the $i$-th eigenvalue of $-L$ with the associated eigenvector $\tilde{v}_{i,n,\epsilon}$ normalized in the $l^2$ norm, where $i=1,\ldots,n$. We order $\mu_{i,n,\epsilon}$ so that {$0=\mu_{0,n,\epsilon}\leq \mu_{1,n,\epsilon}\leq \ldots\leq \mu_{n-1,n,\epsilon}\leq \frac{1}{\epsilon^2}$.}

{Recall that} the heat kernel has the following expression:
\begin{align}
\mathsf{H}(x,x',t)=\sum_{i=0}^{\infty}e^{-\lambda_i t} \phi_i(x)\phi_i(x'). \nonumber
\end{align}
Therefore, intuitively, {if} we are able to recover the eigenfunctions and eigenvalues of the Laplace-Beltrami operator through $L$, we may be able to recover the heat kernel via the same formula, like $\sum_{i=0}^{\infty}e^{-\mu_{i,n,\epsilon} t} \tilde{v}_{i,n,\epsilon}\tilde{v}_{i,n,\epsilon}^\top$. Below, we show that this intuition is correct, if we carefully carry out the estimation.
{We mention} that the idea of using the kernel normalized GL to gain benefit from the heat kernel has been applied {in several algorithms, like} the nonlinear dimension reduction algorithm diffusion map (DM) \cite{coifman2006diffusion}.

It {has been} well known that the matrix $L$ converges ``pointwisely'' to the Laplace-Beltrami operator of the manifold, which we summarize below. 
\begin{theorem}[e.g. \cite{coifman2006diffusion,singer2012vector}]\label{pointwise convergence}
Suppose $f \in C^4(M)$, if $\epsilon=\epsilon(n)$ so that $\frac{\sqrt{\log n}}{n^{1/2}\epsilon^{d/2+2}} \rightarrow 0$ and $\epsilon \rightarrow 0$ as $n \rightarrow \infty$, then with probability greater than $1-n^{-2}$, for all $x_i$, we have
\begin{align}
\sum_{j=1}^n L_{ij}f(x_j)=\Delta f(x_i)+O(\epsilon^2)+O\bigg(\frac{\sqrt{\log n}}{n^{1/2}\epsilon^{d/2+2}}\bigg)\,. \nonumber 
\end{align}
\end{theorem}

Since {this fact} has been proved in several places, we omit the proof.
Note that Theorem \ref{pointwise convergence} does not directly link the eigenvectors of $L$ and the eigenfunctions of the Laplace-Beltrami operator. Moreover, Theorem \ref{pointwise convergence} is only a pointwise convergence result, which is not strong enough to guarantee the spectral convergence. 

\begin{remark}\label{coeffecient 1}
Denote the $k$-th moment {of the kernel} as $m_k:=\int_{\mathbb{R}^d}\|u\|^ke^{-\|u\|^2/4\epsilon^2}dx$, where $k=0,1,\ldots$. Note that we have $m_0=(4\pi \epsilon^2)^{d/2}$ and  $m_2=2 (4\pi \epsilon^2)^{d/2} \epsilon^2$, so a straightforward Taylor expansion in the normal coordinates shows that the coefficient is $1$ in front of $\Delta f(x)$. 
\end{remark}

\begin{remark}
Note that the result shown in Theorem \ref{pointwise convergence} is associated with the $1$-normalization; that is, $\alpha=1$ in the $\alpha$-normalization {proposed in} \cite{coifman2006diffusion}. 
{On the other hand, $L^{un}$} converges to a second order differential operator that {is} {influenced} by the p.d.f. if the p.d.f. is not uniform. In this case, the pointwise convergence rate is $O\Big(\frac{\sqrt{\log n}}{\sqrt{n}\epsilon^{d/2+1}}\Big)$, which is faster than that shown in Theorem \ref{pointwise convergence}. The difference of the convergence rate comes from the {KDE,} $q_\epsilon$.
Hence, if the dataset is uniformly sampled, one can simply consider {$L^{un}$} to achieve a faster pointwise convergence rate. 
\end{remark}

Next we describe our main results about recovering the eigenvalues and eigenfunctions of the Laplace-Beltrami operator from {the kernel normalized GL} in the $L^\infty$ sense. This will be the foundation of reconstructing the heat kernel in the $L^\infty$ sense. %
To state our main contribution, we need the following normalization. 
Since the eigenfunction $\phi_i$ of the Laplace-Beltrami operator is normalized in the $L^2(M)$ norm, in order to compare the $i$-th eigenvector of $L$ and $\phi_i$, we have to make sure that $\tilde{v}_{i,n,\epsilon}$ is properly normalized.

\begin{definition}\label{weighted l2 norm}
Under Assumption \ref{assumption DM}, suppose $\tilde{v}$ is an eigenvector of $L$ which is normalized in the $l^2$ norm. Let 
$\mathbb{N}(i)=|B^{\mathbb{R}^p}_{\epsilon}(\iota(x_i)) \cap \{\iota(x_1), \cdots, \iota(x_n)\}|$, {which is} the number of points in the $\epsilon$ ball in the ambient space.
Then, we define the $l^2$ norm of $\tilde{v}$ with respect to inverse estimated probability density $1/\hat{\mathsf{p}}$ as: 
\begin{align}
\|\tilde{v}\|_{l^2(1/\hat{\mathsf{p}})}:=\sqrt{\frac{|S^{d-1}|\epsilon^d}{d} \sum_{i=1}^n \frac{\tilde{v}^2(i)}{\mathbb{N}(i)}}. \nonumber 
\end{align}
\end{definition}
Note that in this definition, {the term $\frac{d\mathbb{N}(i)}{|S^{d-1}|\epsilon^d}$ is an estimation of the probability density $\mathsf{p}$ at $x_i$ by the $0-1$ kernel over the Euclidean ball centered at $\iota(x_i)$ with radius $\epsilon$. The reader may refer to Lemma \ref{KDE} for details. } A more sophisticated {KDE} can be applied here, but we are not concerned with it to simplify the discussion.
Then, define
\begin{align} \label{normalized V i n epsilon}
v_{i,n,\epsilon}=\frac{\tilde{v}_{i,n,\epsilon}}{\|\tilde{v}_{i,n,\epsilon}\|_{l^2(1/\hat{\mathsf{p}})}}.  
\end{align}
Intuitively, $v_{i,n,\epsilon}$ can be regarded as a discretization of some function that is normalized in the $L^2(M)$ norm. A rigorous discussion can be found in Proposition \ref{Convergence of normalized eigenvectors} in Section \ref{proof of theorem 2}. In the following theorem, we show that $v_{i,n,\epsilon}$ is an approximation of $\phi_i$ over the $n$ data points with high probability.

The following theorem says that, on a closed manifold $M$, if we fix $K$ and we choose the bandwidth $\epsilon$ based on the number of data points $n$,  then for $i < K$,  with high probability, $\mu_{i,n,\epsilon}$ is  an approximation of the $i$-th eigenvalue $\lambda_i$ of $-\Delta$.
 The proof of the theorem is in Section \ref{proof of theorem 2}.

\begin{theorem} \label{spectral convergence of L on closed manifold}(Spectral convergence of {the kernel normalized} GL) Under Assumption \ref{assumption DM}, suppose all the eigenvalues of $\Delta$ are simple. {Fix $K \in \mathbb{N}$, and denote $\mathsf \Gamma_K:=\min_{1 \leq i \leq K}\textup{dist}(\lambda_i, \sigma(-\Delta)\setminus \{\lambda_i\})$.}
{Suppose 
\begin{equation}\label{epsilon less than some constant}
\epsilon \leq \mathcal{K}_1 \min \left(\left(\frac{\min(\mathsf\Gamma_K,1)}{\mathcal{K}_2+\lambda_K^{d/2+5}}\right)^2,\, \frac{1}{(\mathcal{K}_3+\lambda_K^{(5d+7)/4})^2}\right)\,,
\end{equation}
where $\mathcal{K}_1$ and $\mathcal{K}_2, \mathcal{K}_3>1$  are constants depending on $d$, $\mathsf p_m$, the $C^2$ norm of $\mathsf p$, and the volume, the injectivity radius, {the curvature} and the second fundamental form of the manifold.
Then, {when $n$ is sufficiently large so that $\epsilon=\epsilon(n) \geq (\frac{\log n}{n})^{\frac{1}{4d+13}}$} {, with} probability greater than $1-n^{-2}$, for all $0 \leq i < K$, 
\begin{align}\label{Theorem 2 eigenvalue bound}
|\mu_{i,n,\epsilon}-\lambda_{i}|\leq \Omega_1 \epsilon^{3/2}\,. 
\end{align}
Suppose \eqref{epsilon less than some constant} is satisfied. {Then, when $n$ is sufficiently large so that $\epsilon=\epsilon(n)\geq (\frac{\log n}{n})^{\frac{1}{4d+8}} $,} with probability greater than $1-n^{-2}$, there are $a_i \in \{1,-1\}$ such that for all $0 \leq i < K$, 
\begin{align}
\max_{x_j \in \mathcal{X}}|a_i v_{i,n,\epsilon}(j)-\phi_{i}(x_j)|\leq  \Omega_2 \epsilon^{1/2}\,.\label{Theorem 2 eigenvector bound} 
\end{align}
}
{$\Omega_1$ depends on $d$, the diameter of $M$, $\mathsf p_m$, and the $C^2$ norm of $\mathsf p$, and $\Omega_2$ depends on $d$, the diameter and the volume of $M$, $\mathsf p_m$, and the $C^2$ norm of $\mathsf p$.}
\end{theorem}

{
Note that $\mathsf \Gamma_K$ is the smallest spectral gap between the first $K$ eigenvalues. Hence, $\mathsf \Gamma_K$ is a non-increasing function of $K$. When $d=1$, $\mathsf \Gamma_K=\lambda_1$ which is a constant only depending on the arclength of the curve. However, when 
$d \geq 2$, as $K$ increases, it is possible that $\mathsf \Gamma_K=|\lambda_{l+1}-\lambda_{l}|$ for some $l$ close to $K$. No matter in which case, since we consider a fixed number of eigenvalues and eigenvectors, this eigengap dependence is fixed when $K$ is fixed.  

Since $\lambda_K$ grows {polynomially}, the theorem emphasizes that if we want to recover more eigenvalues and eigenvectors, the $\epsilon$ needs to get polynomially smaller. In practice, we may only want to recover a fixed number of eigenvalues and eigenvectors, and we may not need to focus on this dependence.  See Section \ref{Section: spectral convergence independent of K} for a discussion in this direction.
}

We want to point it out that we assume that the eigenvalues of $\Delta$ are simple to simplify the discussion. In the case when the eigenvalues are not simple, the same proof still works by introducing the eigenprojection \cite{chatelin2011spectral}.

{
\begin{remark}\label{consistent relation}
We shall mention that while the error of the eigenvalue and eigenvector estimation listed in \eqref{Theorem 2 eigenvalue bound} and \eqref{Theorem 2 eigenvector bound} seems to simply depend on $\epsilon$, it is actually an interplay of $\epsilon$ and $n$.  In the above theorem, we have two relationships between $\epsilon$ and $n$, one for eigenvalue and one for eigenvector. If we choose $\epsilon=\epsilon(n)=(\frac{\log n}{n})^{\frac{1}{4d+13}}$, asymptotically, the eigenvector converges at the rate $n^{-1/(8d+16)}$ up to a log factor, which is slow compared with the eigenvalue convergence rate at $n^{-3/(8d+26)}$ up to a log factor. 
Usually, when we apply GL to approximate the Laplace-Beltrami operator in the spectral sense, we use one $\epsilon$ instead of two. Based on the above result, if we choose any $\epsilon$ such that \eqref{epsilon less than some constant} is satisfied and $\epsilon\geq (\frac{\log n}{n})^{\frac{1}{4d+13}}$, since $(\frac{\log n}{n})^{\frac{1}{4d+8}}\leq (\frac{\log n}{n})^{\frac{1}{4d+13}}$, then the above theorem implies that
$$
\max_{x_j \in \mathcal{X}}|a_i v_{i,n,\epsilon}(j)-\phi_{i}(x_j)|\leq  \Omega_2 \epsilon^{1/2}\,.
$$
As we need to recover both the eigenvalues and the eigenfunctions of the Laplace-Beltrami operator in oder to recover the heat kernel, we choose this relationship $\epsilon \geq (\frac{\log n}{n})^{\frac{1}{4d+13}}$ later when we estimate the heat kernel.
\end{remark}
}

{
Based on the above theorem,  if we combine the relations $\epsilon \geq (\frac{\log n}{n})^{\frac{1}{4d+13}}$ and \eqref{epsilon less than some constant}, then we have a sufficient condition on the number $n$ to recover the first $K$ eigenvalues of $-\Delta$; that is, if
\begin{align}\label{sufficient condition K and n 1}
n>\max \Bigg(\bigg(\frac{\mathcal{K}_2+\lambda_K^{d/2+5}}{\min(\mathsf\Gamma_K,1)}\bigg)^{8d+26},\,\,  (\mathcal{K}_3+\lambda_K^{(5d+7)/4})^{8d+26}\Bigg),
\end{align}
then for $0 \leq i \leq K$,
$$|\mu_{i,n,\epsilon}-\lambda_{i}|\leq \Omega_1 \left(\frac{\log n}{n}\right)^{\frac{3}{8d+26}}.$$
Similarly, if 
\begin{align}\label{sufficient condition K and n 2}
n>\max \Bigg(\bigg(\frac{\mathcal{K}_2+\lambda_K^{d/2+5}}{\min(\mathsf\Gamma_K,1)}\bigg)^{8d+16},\,\,  (\mathcal{K}_3+\lambda_K^{(5d+7)/4})^{8d+16}\Bigg),
\end{align}
then for $0 \leq i \leq K$,
$$\max_{x_j \in \mathcal{X}}|a_i v_{i,n,\epsilon}(j)-\phi_{i}(x_j)|\leq  \Omega_2  \left(\frac{\log n}{n}\right)^{\frac{1}{8d+16}}.$$
\begin{remark}\label{consistent relation2}
Another remark is about how to determine $\epsilon$ and $K$ in practice. We shall emphasize that the asymptotic relationship between $\epsilon$ and $n$, like $\epsilon= (\frac{\log n}{n})^{\frac{1}{4d+13}}$ or $\epsilon= (\frac{\log n}{n})^{\frac{1}{4d+8}}$, could be a guidance, but it is not practical to choose $\epsilon$ directly from this relationship. It is because the constant $\Omega_2$ is usually unknown for a generic manifold. To our knowledge, despite some practical rules of thumb, in general a good choice of $\epsilon$ depends on some try-and-errors. See \cite{ding2020phase} for a recent theoretically solid choice that aims to choose an $\epsilon$ that is robust to ``big'' high dimensional noise. 
However, when there is no noise or when noise is not big, how to choose an optimal $\epsilon$ for the spectral convergence purpose is still open.
We face the similar challenge in determining the relation between $n$ and $K$. Indeed, while we could obtain an asymptotic relationship between $K$ and $n$ via applying Lemma \ref{laplace eigenvalue lower bound} to \eqref{sufficient condition K and n 1} or \eqref{sufficient condition K and n 2}, the relationship is not practical since the implied constant is in general difficult to control. To our knowledge, there does not exist a universal practical way to determine what $K$ should be, and in practice it depends on try-and-errors. We will report our work in these practical challenges in our future work. 
\end{remark}

}

{Now we introduce our heat kernel estimator.} Fix $K\in \mathbb{N}$ and $t>0$. Consider the matrix
\begin{align}\label{Construction heat kernel approximation}
\mathsf H^{(K)}_{\epsilon,t}=\sum_{i=0}^{K-1} e^{-\mu_{i,n,\epsilon} t} v_{i,n,\epsilon} v^\top_{i,n,\epsilon}\in \mathbb{R}^{n\times n}\,,
\end{align}
which will be our estimation of the heat kernel.
{Note that we} {sum $i$ from $0$} to $K-1$, {which is based on the statement for} $i<K$ in the previous theorem.

To {study} how $\mathsf H^{(K)}_{\epsilon,t}$ approximates the heat kernel $\mathsf H(\cdot,\cdot,t)$, below we list {assumptions about the relationship} among $n$, $\epsilon$, $K$ and $t$.
\begin{assumption}\label{assumption nekt}
The parameters $n$, $\epsilon$, $K$ and $t$ satisfy the  following conditions. 
\begin{enumerate}
{
\item Set $\mathsf \Gamma_K:=\min_{1 \leq i \leq K}\textup{dist}(\lambda_i, \sigma(-\Delta)\setminus \{\lambda_i\})$. Suppose 
\begin{align}
\epsilon \leq  \min \Bigg(\mathcal{K}_1 \bigg(\frac{\min(\mathsf\Gamma_K,1)}{\mathcal{K}_2+\lambda_K^{d/2+5}}\bigg)^2,\,\, \frac{1}{(\mathcal{K}_3+\lambda_K^{(5d+7)/4})^2}, \,\,\frac{\mathcal{K}}{K^4 \lambda_K^{(d-1)/2}} \Bigg), \nonumber 
\end{align}
where  $\mathcal{K}_1$ and $\mathcal{K}_2, \mathcal{K}_3>1$  are the constants in Theorem \ref{spectral convergence of L on closed manifold} and $\mathcal{K}$ depends on $d$ and the Ricci curvature and the diameter of $M$. 
\item $n$ is large enough so that $(\frac{\log n}{n})^{\frac{1}{4d+13}} \leq \epsilon$.
\item  $\frac{C \log K}{ K^{2/d}} \leq t$, where $C$ is a constant depending on $d$ and the Ricci curvature and the diameter of $M$.
}
\end{enumerate}
\end{assumption}

Based on the above {assumption}, we have the following theorem.  The proof of the theorem is in Section \ref{proof of theorem 3}.
\begin{theorem}  \label{heat kernel reconstruction 1}
Suppose Assumptions \ref{assumption DM} and \ref{assumption nekt} hold. Then, with probability greater than $1-n^{-2}$ 
{$$\sup_{x_i,x_j \in \mathcal{X}}|\mathsf{H}^{(K)}_{\epsilon,t}(i,j)-\mathsf{H}(x_i, x_j, t)| < \frac{\Omega_3}{K}(t+1)e^{-\Omega_4 t} +\Omega_5 \epsilon,$$
where $\Omega_3$ depends on $d$, $\mathsf p_m$, the $C^2$ norm of $\mathsf p$, and the injectivity radius, the diameter, the volume and the curvature of $M$, $\Omega_4$ depends on $d$, and the Ricci curvature and the diameter of $M$, and $\Omega_5$ depends on the volume of $M$ and the $C^2$ norm of $\mathsf p$.}
\end{theorem} 

{The above assumption and theorem deserve a discussion. In the above theorem, the error between {the heat kernel estimator} and the actual heat kernel is characterized by $K$, $\epsilon$ and $t$. The error {is composed} of two parts, the bound of the remainder term $\sum_{i=K}^{\infty}e^{-\lambda_i t} \phi_i(x)\phi_i(x')$ in the series expression of the heat kernel and the error between $\mathsf{H}^{(K)}_{\epsilon,t}(i,j)$ and $\sum_{i=0}^{K-1}e^{-\lambda_i t} \phi_i(x_i)\phi_i(x_j)$. 
In fact, in the proof of the above theorem, we show that if $t$ is not too small, i.e.  $\frac{C \log K}{ K^{2/d}} \leq t$, then $e^{-\lambda_i t}$ is small and $|\sum_{i=K}^{\infty}e^{-\lambda_i t} \phi_i(x)\phi_i(x')|$ is bounded by $\frac{1}{K}(t+1)e^{-\Omega_4 t}$ up to a constant for all $x, x' \in M$. On the other hand, based on Theorem \ref{spectral convergence of L on closed manifold}, {(1) and (2) in Assumption \ref{assumption nekt}} enable us to have a good control between the first $K$ eigenpairs of the kernel normalized GL and the corresponding eigenpairs of $-\Delta$. Together with the condition $\epsilon \leq \frac{\mathcal{K}_3}{K^4 \lambda_K^{(d-1)/2}}$ in (1) {in Assumption \ref{assumption nekt}}, we can conclude that the error between $\mathsf{H}^{(K)}_{\epsilon,t}(i,j)$ and $\sum_{i=0}^{K-1}e^{-\lambda_i t} \phi_i(x_i)\phi_i(x_j)$ is also bounded by $\frac{1}{K}(t+1)e^{-\Omega_4 t}$ up to a constant plus $\Omega_5 \epsilon$. Since we do not know the volume of $M$ and $\mu_{0,n, \epsilon}=\lambda_0=0$, the term $\Omega_5 \epsilon$ comes from the error generated when we renormalize  {$\tilde{v}_{0,n,\epsilon}$} as in \eqref{normalized V i n epsilon} to approximate $\phi_0$. Next, we discuss the dependence  on the diffusion time $t$ in the error of the approximation. Note that  $\frac{C \log K}{ K^{2/d}}$  decreases as $K$ increases; hence, the relationship (3) says that we can choose a small $t$, whenever $K$ is large enough. $(t+1)e^{-\Omega_4 t}$ is bounded up by a constant depending on $\Omega_4$. Hence, for {a} small diffusion time $t$, we need to choose a large $K$ in order to have a good approximation. Since $\epsilon<\frac{1}{K}$ when $K$ is a large enough, the error in the approximation is dominated by $\frac{1}{K}$ for {a} small diffusion time $t$. For {a} large diffusion time $t$,  since $(t+1)e^{-\Omega_4 t}$ goes to $0$ as $t$ goes to infinity, the {first term in the} error bound is {smaller when $t$ increases}. }

{\subsection{The spectral convergence {under a special setup}}\label{Section: spectral convergence independent of K}

In the previous subsection, we show the spectral convergence of the kernel normalized GL to the Laplace-Beltrami operator regardless of the distribution of the data points on the manifold in the $L^\infty$ sense so that the spectral geometry tool \cite{berard1994embedding} can be applied for further data analysis.
In this subsection, for the sake of comparing our results with other relevant result, we discuss the spectral convergence under a special setup. Specifically, we focus on the case when the p.d.f. is uniform so that the $\alpha$ normalization is not needed for the recovery of the Laplace-Beltrami operator, and the dependence on $\lambda_K$ is hidden in the implied constant but not shown in the relationship between $n$ and $\epsilon$.

Denote $\mu^{un}_{i,n,\epsilon}$ to be the $i$-th eigenvalue of $-L^{un}$ with the associated eigenvector $\tilde{v}^{un}_{i,n,\epsilon}$ normalized in the $l^2$ norm, where $i=1,\ldots,n$. We order $\mu^{un}_{i,n,\epsilon}$ so that $\mu^{un}_{0,n,\epsilon}\leq \mu^{un}_{1,n,\epsilon}\leq \ldots\leq \mu^{un}_{n-1,n,\epsilon}$. Let $\|\cdot\|_{l^2(1/\hat{\mathsf{p}})}$ be the normalized $l^2$ norm in Definition \ref{weighted l2 norm}. Let 
\begin{align} 
v^{un}_{i,n,\epsilon}=\frac{\tilde{v}^{un}_{i,n,\epsilon}}{\|\tilde{v}^{un}_{i,n,\epsilon}\|_{l^2(1/\hat{\mathsf{p}})}}.  
\end{align}

In the following theorem, we treat the eigengaps of $-\Delta$ and the eigenvalue $\lambda_K$ as constants and disregard the relation between $\epsilon$ and $K$. The proof of the theorem is {almost the same} the proof of Theorem \ref{spectral convergence of L on closed manifold} {except simplification of several steps, so its proof is sketched and postponed to} Appendix \ref{proof of un GL rate}.

\begin{theorem}\label{spectral convergence of un L on closed manifold}(Spectral convergence of the unnormalized GL) 
{ Suppose the sampling is uniform and} all the eigenvalues of $\Delta$ are simple. Fix $K\in \mathbb{N}$. Let $\mathsf \Gamma_K:=\min_{1 \leq i \leq K}\textup{dist}(\lambda_i, \sigma(-\Delta)\setminus \{\lambda_i\})$. { For a sufficiently small $\epsilon>0$,  if $n$ is sufficiently large so that $\epsilon=\epsilon(n)\geq\big(\frac{\log n}{n}\big)^{\frac{1}{2d+12}}$,  then}
 with probability greater than $1-n^{-2}$,  for all $0 \leq i < K$
\begin{align}
& |\mu^{un}_{i,n,\epsilon}-\lambda_{i}|\leq \Omega^{un}_1 \epsilon^{2}\,,\nonumber 
\end{align}
where $\Omega^{un}_1>0$ depends on $\mathsf \Gamma_K$, $\lambda_K$, $d$ and the diameter, the volume, the injectivity radius, the curvature and the second fundamental form of the manifold. 
For a sufficiently small $\epsilon>0$,  if $n$ is sufficiently large {so that $\epsilon=\epsilon(n)\geq \big(\frac{\log n}{n}\big)^{\frac{1}{2d+8}}$,  then} with probability greater than $1-n^{-2}$,   there are $a_i \in \{1,-1\}$  such that for all $0 \leq i < K$
\begin{align}
\max_{x_j \in \mathcal{X}}|a_i v^{un}_{i,n,\epsilon}(j)-\phi_{i}(x_j)|\leq  \Omega^{un}_2 \epsilon^{2}\,,\nonumber 
\end{align}
where $\Omega^{un}_2>0$ depends on $\mathsf \Gamma_K$, $\lambda_K$, $d$ and the diameter, the volume, the injectivity radius, the curvature and the second fundamental form of the manifold. 
\end{theorem}

Note that remarks similar to Remarks \ref{consistent relation} and \ref{consistent relation2} hold for Theorem \ref{spectral convergence of un L on closed manifold}. Also, we see that the relationship between $n$ and $\epsilon$ in Theorem \ref{spectral convergence of un L on closed manifold} is different from that in Theorem \ref{spectral convergence of L on closed manifold}. The difference mainly comes from the kernel normalization for the purpose of density estimation. By this theorem, when the point cloud is sampled uniformly, if we take $\epsilon=\big(\frac{\log n}{n}\big)^{\frac{1}{2d+12}}$, we have 
\begin{align}
|\mu^{un}_{i,n,\epsilon}-\lambda_{i}|=O\left(\left(\frac{\log n}{n}\right)^{\frac{1}{d+6}}\right)\,; \nonumber 
\end{align}
if we take $\epsilon=\big(\frac{\log n}{n}\big)^{\frac{1}{2d+8}}$, we have for $a_i\in\{1,-1\}$,
\begin{align}
\max_{x_j \in \mathcal{X}}|a_i v^{un}_{i,n,\epsilon}(j)-\phi_{i}(x_j)|=O\left(\left(\frac{\log n}{n}\right)^{\frac{1}{d+4}}\right). \nonumber 
\end{align}
We shall compare it with \cite{calder2019improved,calder2020lipschitz}. In these papers, the authors also consider a $d$ dimensional compact manifold $M$ without boundary that is isometrically embedded in $\mathbb{R}^D$. 
In these papers,  to construct the unnormalized GL, they use a compactly supported kernel of the form $\eta(\frac{\|x-y\|_{\mathbb{R}^D}}{\epsilon})$, where $x,y\in M$ and $\eta:[0, \infty) \rightarrow [0,\infty)$ is a non-increasing Lipschitz continuous function with the compact support on $[0,1]$. 
For $n$ points i.i.d. sampled from $M$ based on a density function $\rho$, under this setup, in \cite{calder2019improved}, the authors prove the $L^2$ spectral convergence of the unnormalized GL to the weighted Laplacian $\frac{1}{\rho}\nabla \cdot (\rho^2 \nabla)$ using the variational method for Dirichlet form and the techniques in optimal transport. 
Moreover, the authors show that the eigenvalue and eigenvector convergence is at a rate $O\left(\left(\frac{\log n}{n}\right)^{\frac{1}{d+4}}\right)$. 
Based on the result in \cite{calder2019improved}, the authors in \cite{calder2020lipschitz} improve the $L^2$ convergence of the eigenvectors to the $L^\infty$ convergence. 
They show that the $L^\infty$ convergence rate of the eigenvectors matches the $L^2$ rate; that is, $O\left(\left(\frac{\log n}{n}\right)^{\frac{1}{d+4}}\right)$. Note that in the case when the sampling density is uniform, that is, $\rho$ is constant, then $\frac{1}{\rho}\nabla \cdot (\rho^2 \nabla)$ gives the Laplace-Beltrami operator.
Hence, in this setup, our eigenvector convergence rate of unnormalized  GL in Theorem \ref{spectral convergence of un L on closed manifold} matches the $L^\infty$ result of  \cite{calder2020lipschitz}. However, our eigenvalue convergence rate is slower.  
It is also worth mentioning that our results in Theorem \ref{spectral convergence of L on closed manifold} and Theorem \ref{spectral convergence of un L on closed manifold} rely on the fact that the kernel function is Gaussian, while the kernel considered in \cite{calder2020lipschitz} is compactly supported and Lipschitz, which does not include the widely used Gaussian kernel. See Remark \ref{rely on Gaussian} for more details. 
We emphasize that the techniques used in the current paper and those used in \cite{calder2020lipschitz} are different, and it is potentially interesting to combine these techniques to obtain a more general result. We would expect that the optimal spectral convergence rate is better than what we have reported, in both kernel normalized and unnormalized cases.
}

\section{Some generic facts}

In this section, we summarize some fundamental facts we need in the whole proof. 
The first two lemmas are the facts about the eigenvalues and eigenfunctions of the Laplace-Beltrami operator.

\begin{lemma}[\cite{hormander1968spectral,donnelly2006eigenfunctions}]\label{lemma hormander}
For a compact Riemannian manifold $(M,g)$ and $l>0$, we have the following bound for the $l$-th pair of eigenvalue $\lambda_l$ and normalized eigenfunction $\phi_l$ of the Laplace-Beltrami operator:  
\begin{align}
\|\phi_l\|_\infty \leq C_1 \lambda_l^{\frac{d-1}{4}} \|\phi_l\|_2= C_1 \lambda_l^{\frac{d-1}{4}}\,,   \nonumber 
\end{align}
where $C_1$ is a constant depending on the injectivity radius and sectional curvature of the manifold $M$.
\end{lemma}

{\begin{lemma}[\cite{hassannezhad2016eigenvalue}]\label{laplace eigenvalue lower bound}
For a $d$ dimensional compact and connected Riemannian manifold $(M,g)$, the eigenvalues of the Laplace-Beltrami operator, $0=\lambda_0< \lambda_1\leq\ldots$, satisfy
{
\begin{align}
C_2^{1+d\sqrt{\kappa}}\text{diam}(M)^{-2} l^{2/d} \leq \lambda_l \leq \frac{(d-1)^2}{4}\kappa + \check C_2 V(M)^{-2/d}l^{2/d}\nonumber 
\end{align}
for all $l \geq 1$, where the $\text{Ric}_g\geq -(d-1)\kappa g$ for $\kappa\geq 0$, and $C_2>0$ and $\check C_2>0$ are constants depending on $d$ only. }
\end{lemma}
}

The following lemma describes the behavior of the heat kernel. The proof of can be found in \cite{rosenberg1997laplacian}, \cite{davies1989pointwise} and \cite{grigor1997gaussian}.
\begin{lemma}\label{bounds on heat kernel}
Suppose $M$ is a $d$ dimensional closed smooth manifold. For $x,y \in M$, $d(x,y)$ is the geodesic distance between $x$ and $y$.  
\begin{enumerate}[(a)]
\item
Fix $\gamma>0$ small enough. For $d(x,y)<\gamma$, we have when $t\to 0$:
\begin{align}
\mathsf{H}(x,y,t)=(4\pi t)^{-d/2} e^{-\frac{d(x,y)^2}{4t}}(u_0+t u_1+O(t^2))\,.\nonumber 
\end{align}
$u_0(x,y)=\frac{1}{\sqrt{\det\big(d(\exp_x)_{\exp^{-1}_x(y)}\big)}}$, where $\exp_x$ is the exponential map at $x$.  $u_1$ is a continuous function of $x$ and $y$ that only depends on the manifold. Moreover, $u_0(x,y)=1+O(d(x,y)^2)$, where the constant in $O(d(x,y)^2)$ depends on the curvature of the manifold. 
 {
\item
{There exists a} constant $C_H>0$ depending on $d$ and the volume, the Ricci curvature and the diameter of $M$ {so that the following holds.
First,} for all $x,y \in M$ and $t>0$,
\begin{align}
|\partial_t \mathsf{H}(x,y,t)| \leq C_H t^{-d/2-1} e^{-\frac{d(x,y)^2}{8t}}.\nonumber
\end{align}
For all $x,y \in M$ and {$t>0$} small enough,
\begin{align}
|\nabla_x \mathsf{H}(x,y,t)| \leq C_H t^{-d/2-1/2} e^{-\frac{d(x,y)^2}{8t}}\,.\nonumber
\end{align}
{Moreover, for $\gamma>0$, when $t>0$ is} small enough, if $d(x,y) \geq \gamma$, then
\begin{align}
& \mathsf{H}(x,y,t) \leq C_H  t^{2}, \nonumber \\
& |\partial_t \mathsf{H}(x,y,t)| \leq C_H  t. 
\end{align}
}
\end{enumerate}
\end{lemma}

We need the following projection lemma. The proof can be found in \cite[Proposition 18]{von2008consistency}, so we omit it.

\begin{lemma}\label{projection lemma}
For any two vectors $v$ and $w$ in a Banach space $(E, \|\cdot \|_E)$ with $\|w\|_E=\|v\|_E=1$, let $\textup{Pr}_v$ be the projection operator onto the subspace generated by $v$. Then we can take $a=1$ or $-1$ so that
\begin{align}
\|aw-v\|_E \leq 2\|w-\textup{Pr}_v w\|_E\,. \nonumber
\end{align}
\end{lemma}

The following lemma discusses some basic facts about the spectral convergence of compact self-adjoint operators. { Moreover, we refer readers to \cite{reed2012methods,conway2019course} for the basic spectral properties of the compact self-adjoint operators on a Hilbert space.}

\begin{lemma}\label{ratio of eigenvalues}
Let $A$ and $B$ be compact self-adjoint operators from $L^2(M)$ to $L^2(M)$. Let $(\cdot, \cdot)$ be the inner product of $L^2(M)$. Assume the eigenvalues of $A$, $\lambda_l(A)$, $l=0, 1,\ldots$, are simple and positive and the eigenvalues of $B$, $\lambda_l(B)$, $l=0, 1,\ldots$, are simple and bounded from below so that $1=\lambda_0(A)>\lambda_1(A) > \lambda_2(A) > \cdots\geq 0$ and $\lambda_0(B) >\lambda_1(B) >  \lambda_2(B) > \cdots $. Also denote $\{u_i\}$ to be the orthonormal eigenfunctions of $A$ and $\{w_i\}$ to be orthonormal eigenfunctions of $B$. {For $i=1,2,\ldots$,} denote 
\begin{equation}
\gamma_i(B):=\min(\lambda_i(B)-\lambda_{i-1}(B), \,\lambda_{i+1}(B)-\lambda_i(B))\,.\label{Proof proposition 1 Definition gamma i} 
\end{equation}
Let $E: = A-B$. Then, for $\epsilon>0$ we have the following statements:
\begin{enumerate}[(a)]
\item
If $\big|\frac{(E f,f)}{(Af,f)}\big| \leq \epsilon$ for all $f \in L^2$, then for all $i >0$, we have 
\begin{equation}
\left|\frac{1-\lambda_i(B)}{1-\lambda_i(A)}-1\right|\leq \epsilon\,.\nonumber
\end{equation}
\item
If $\|Bu_i-\lambda_i(B)u_i\|_2 \leq \epsilon$, {where $i\geq 1$,} then we can find $a=1$ or $-1$ so that
\begin{equation}
\|aw_i- u_i\|_2 \leq \frac{2\epsilon}{\gamma_i(B)}. \nonumber
\end{equation}
Moreover, 
\begin{equation}
|(u_i,w_i)| \geq 1-\frac{\epsilon}{\gamma_i(B)}\,.\nonumber
\end{equation}
\item {For $i, j\geq 0$, when $(u_i, w_j)\neq 0$,} the {associated} eigenvalues satisfy
\begin{equation}
|\lambda_i(A)-\lambda_j(B)|\leq \frac{\|Ew_j\|_2}{|(u_i,w_j)|}\,.\nonumber
\end{equation}
\end{enumerate}
\end{lemma}

\begin{remark}
Note that (b) can be understood as a variation of the Davis-Kahan theorem. 
Later we will apply the lemma to operators of the form $I-C$ when $C$ is compact and self-adjoint. Clearly, $I-C$ is not compact but shares the same eigenfunctions with $C$, and the eigenvalues of $I-C$ are simply those of $1-\lambda$ when $\lambda$ is an eigenvalue of $C$. Specifically, later we will set $A$ to be the integral operator with the heat kernel, and $B$ to be the integral operator with a diffusion kernel. We mention that (c) is also used in \cite{calder2019improved} to prove the spectral convergence rate in the $L^2$ sense.
\end{remark}

\begin{proof}
The first statement can be found in \cite[Proposition 4.4]{belkin2007convergence}, so we omit it. 

For the second statement, we express $u_i=b_iw_i+\sum_{j={0}, j\not=i}^\infty b_j w_j$. By a direct expansion {and the assumption}, we have
\begin{align}
\|Bu_i-\lambda_i(B)u_i\|_2=&\, \Big\|B\Big(b_iw_i+\sum_{j={0}, j\not=i}^\infty b_j w_j\Big)-\lambda_i(B)\Big(b_iw_i+\sum_{j={0}, j\not=i}^\infty b_j w_j\Big)\Big\|_2 \nonumber \\
= &\, \Big\|B\Big(\sum_{j={0}, j\not=i}^\infty b_j w_j\Big)-\lambda_i(B)\Big(\sum_{j={0}, j\not=i}^\infty b_j w_j\Big)\Big\|_2 \nonumber \\
= &\,  \Big\|\sum_{j=0, j\not=i}^\infty(\lambda_j(B)-\lambda_i(B)) b_jw_j\Big\|_2\nonumber\\
=&\,\sqrt{\sum_{j=0, j\not=i}^\infty|\lambda_j(B)-\lambda_i(B)|^2 |b_j|^2 } \leq \epsilon\nonumber\,.
\end{align}
Therefore, we have
\begin{align}
\|u_i-(u_i,w_i)w_i\|_2=&\,\Big\|\sum_{j=0, j\not=i}^\infty b_jw_j\Big\|_2=\sqrt{\sum_{j={0}, j\not=i}^\infty |b_j|^2 }\nonumber\\
\leq&\,\frac{1}{\gamma_i(B)}\sqrt{\sum_{j=0, j\not=i}^\infty|\lambda_j(B)-\lambda_i(B)|^2 |b_j|^2 } \leq \frac{\epsilon}{\gamma_i(B)}\,. \nonumber
\end{align}
By Lemma \ref{projection lemma}, we have
\begin{equation}
\|aw_i-u_i\|_2 \leq 2\|u_i-(u_i,w_i)w_i\|_2 \leq \frac{2\epsilon}{\gamma_i(B)}\,,
\end{equation} 
which leads to the conclusion of $|(u_i,w_i)| \geq 1-\frac{\epsilon}{\gamma_i(B)}$ from a direct expansion via the polarization identity.

For the third statement, the self-adjoint assumption leads to
\begin{align}
\lambda_i(A) (u_i, w_j)&\,=(Au_i,w_j)=(u_i, Aw_j)=(u_i, Bw_j)+(u_i, (A-B)w_j)\nonumber\\
&\,=\lambda_j(B)(u_i, w_j)+(u_i, Ew_j)\,. \nonumber
\end{align}
Hence, {when $(u_i, w_j)\neq 0$,} by Cauchy-Schwartz, we have
$|\lambda_i(A)-\lambda_j(B)| =\frac{|(u_i, Ew_j)|}{|(u_i, w_j)|} \leq \frac{\|Ew_j\|_2}{|(u_i,w_j)|}$
\end{proof}

Recall the definition of collective compact convergence.

\begin{definition}
Let $(E, \|\cdot\|_{E})$ be an arbitrary Banach space, $B\subset E$ be the unit ball centered at $0$, and $\{T_n\}_{n=1}^\infty$ be a set of bounded linear operators on $E$. The set $\{T_n\}_{n=1}^\infty$ is called collectively compact if the set $ \cup_n T_n (B)$ is relatively compact in $(E, \|\cdot\|_{E})$. The sequence $\{T_n\}_{n=1}^\infty$ is said to converge collectively compactly to an operator $T$, if it converges pointwise and there exists some $N\in \mathbb{N}$ such that the sequence of operators $\{T_n-T\}_{n=N}^\infty$ is collectively compact.
\end{definition}

Also recall the definition of a resolvent. 

\begin{definition}
Let $T$ be a compact linear operator on an arbitrary Banach space $(E, \|\cdot\|_{E})$ and denote $\sigma(T)\subset \mathbb{C}$ to be the associated spectrum. Then for {$z\in \mathbb{C}\backslash \sigma(T)$}, the resolvent operator is defined as $R_z(T):=(zI-T)^{-1}$. For an eigenvalue $\lambda \in \sigma(T)$, let $\Gamma_r(\lambda)\subset \mathbb{C}$ be a circle centered at $\lambda$ with the radius $r>0$.
\end{definition}

The following generic theorem is the key toward {the spectral convergence}. Specifically, it quantifies the convergence of the corresponding eigenfunctions for a sequence of operators $\{T_n\}_{n=1}^\infty$ that collectively compactly converges to an operator $T$. 

\begin{theorem}\label{atkinson1967}
Let $(E, \|\cdot\|_{E})$ be an arbitrary Banach space. Let $\{T_n\}_{n=1}^\infty$ and $T$ be compact linear operators on $E$ such that $\{T_n\}_{n=1}^\infty$ converges to $T$ collectively compactly. 
For an eigenvalue $\lambda \in \sigma(T)$, denote the corresponding spectral projection by $\textup{Pr}_{\lambda}$. Let $D\subset \mathbb{C}$ be an open neighborhood of $\lambda$ such that $\sigma(T)\cap D$ =$\{ \lambda\}$. There exists some $N\in \mathbb{N}$ such that for all $n > N$, $\sigma(T_n)\cap D=\{\lambda_n\}$. Let $\textup{Pr}_{\lambda_n}$ be the corresponding spectral projection {associated with $\lambda_n$ of $T_n$}. Let $r <|\lambda|$ and $r<\textup{dist}(\{\lambda\}, \sigma(T)\setminus \{\lambda\})$. Then for every $x \in \textup{Pr}_\lambda(E)$, we have
\begin{equation}
\|x-\textup{Pr}_{\lambda_n} x\|_{E} \leq \max_{z\in \Gamma_r(\lambda)} \frac{2r\|R_{z}(T)\|}{\min_{z\in \Gamma_r(\lambda)}|z|}\left(\|(T_n-T)x\|_{E} +\|R_{z}(T)x\|_{E}\|(T-T_n)T_n\|\right)\,.
\nonumber 
\end{equation}
\end{theorem}

This theorem is a restatement of Equation (5) in Theorem 3 in \cite{atkinson1967numerical}. Specifically, we let the set $F=\Gamma_r(\lambda)$, $M=\max_{z\in \Gamma_r(\lambda)} \|R_{z}(T)\|$, and $c_0=\min_{z\in \Gamma_r(\lambda)}|z|$ in \cite[lemma on page 460]{atkinson1967numerical}, set $\epsilon=r$ in \cite[Theorem 3]{atkinson1967numerical}, and note that $\|R_{z}(T)x\|_{E}\leq M\|x\|_E$.

\section{Proof of Theorem \ref{spectral convergence of L on closed manifold}}\label{proof of theorem 2}

\subsection{Some quantities needed in the proof and their properties}\label{Section:basic quantities for proof}
Following Assumption \ref{assumption DM}, for the point cloud $\mathcal{X}:=\{x_i\}_{i=1}^n$ i.i.d. sampled from the random vector with the density function $\mathsf p$ supported on the manifold $M$, we denote the empirical measure associated with the measure $d\mathsf P:=\mathsf pdV_M$, where $dV_M$ is the Riemannian volume measure, {as}
\begin{equation}
\mathsf P_n:=\frac{1}{n}\sum_{i=1}^n\delta_{x_i}\,,
\end{equation}
where $\delta_{x_i}$ is the delta measure supported on $x_i$. 
We start by introducing the following quantities. 
\begin{definition}
For any {continuous} function $f(x,y)$ on $M \times M$, we define two operators, $P$ and $P_n$, with respect to the density measure and empirical measure as 
\begin{align}
Pf (x) &:=\int_{M}f(x, y)d\mathsf P(y)\,, \nonumber \\
P_n f (x)&:= \int_{M}f(x, y)d\mathsf P_n(y)=\frac{1}{n}\sum_{i=1}^{n}f(x, x_i). \nonumber 
\end{align}
For the kernel
\begin{align}
& k_{\epsilon}(x,y)=\exp\Big(-\frac{\|\iota(x)-\iota(y)\|^2_{\mathbb{R}^D}}{4\epsilon^2}\Big), \nonumber
\end{align}
we have the following quantities:
\begin{align}
d_{n,\epsilon}(x):=&\, P_{n} k_{\epsilon}(x) \,,  \quad
d_{\epsilon}(x):=Pk_{\epsilon}(x)\,, \nonumber \\
Q_{n,\epsilon}(x,y):=&\, \frac{k_{\epsilon}(x,y)}{d_{n,\epsilon}(x)d_{n,\epsilon}(y)}\,, \quad
Q_{\epsilon}(x,y):= \frac{k_{\epsilon}(x,y)}{d_{\epsilon}(x)d_{\epsilon}(y)}\,. \nonumber 
\end{align}
\end{definition}
Note that we have $nd_{n,\epsilon}(x)= q_{\epsilon}(x)$ {defined in \eqref{W matrix}}.
We also introduce the following three operators.
\begin{definition}\label{intro of operators}
For any $f(x) \in C(M) $, we define two operators from $C(M)$ to $C(M)$: 
$$T_{n,\epsilon}f(x)= \frac{P_{n}Q_{n,\epsilon}f (x)}{P_{n}Q_{n,\epsilon}(x)},\quad 
T_{\epsilon}f(x)= \frac{PQ_{\epsilon}f (x)}{PQ_{\epsilon}(x)}.$$
Moreover, we define an intermediate operator {connecting $T_{n,\epsilon}$ and $T_{\epsilon}$} that is also from $C(M)$ to $C(M)$: 
\begin{equation}
\hat{T}_{n,\epsilon}f(x)= \frac{P_nQ_{\epsilon}f (x)}{PQ_{\epsilon}(x)}. \nonumber 
\end{equation}

\end{definition}

We first state some known facts about the operator $T_{\epsilon}$. The proof can be found in \cite[Lemma 8]{coifman2006diffusion}. In  \cite{coifman2006diffusion}, the Laplace-Beltrami operator differs from ours by a negative sign, and the bandwidth is {$\sqrt{\epsilon}$} rather than {$\epsilon$}; otherwise the proof is the same, so we omit the proof.
\begin{lemma}\label{properties of T epsilon}
We have the following results.
\begin{enumerate}[(a)]
\item $d_{\epsilon}(x)=m_0 \mathsf p(x) \epsilon^d+O(\epsilon^{d+2})$, 
where $m_0:=\int_{\mathbb{R}^d} e^{-\frac{\|u\|^2}{4}}du$ and the implied constant in $O(\epsilon^{d+2})$ depends on $\mathsf p_m$ and the $C^2$ norm of $\mathsf p$. Hence, 
\begin{align}
Q_{\epsilon}(x,y)= \frac{k_{\epsilon}(x,y)}{m^2_0\mathsf p(x)\mathsf p(y) \epsilon^{2d}}(1+O(\epsilon^{2})), \nonumber 
\end{align}
where the implied constant in $O(\epsilon^{2})$ depends on $\mathsf p_m$ and the $C^2$ norm of $\mathsf p$.
\item
$T_{\epsilon}$ is a self-adjoint operator on $L^2(M)$.
\item
For any $f \in L^2(M)$, $T_{\epsilon}f(x)$ is a smooth function on $M$. In particular, the eigenfunctions of $T_{\epsilon}$ are smooth. Moreover, 
\begin{align}
T_{\epsilon}f(x)=\frac{\int_{M}k_{\epsilon}(x,y)f(y)dV_M}{\int_{M}k_{\epsilon}(x,y)dV_M}+O(\epsilon^2)\,, \nonumber 
\end{align}

where the implied constant in $O(\epsilon^2)$ depends on $\mathsf p_m$,  the $C^2$ norm of $\mathsf p$ and $\|f\|_2$.
\item
For any function $f \in C^4(M)$, 
\begin{align}
\frac{T_{\epsilon}f(x)-f(x)}{\epsilon^2}=\Delta f(x)+O(\epsilon^2)\,, \nonumber 
\end{align}
where the implied constant in $O(\epsilon^2)$ depends on the $L^{\infty}$ norm of $\Delta^2 f$, $\mathsf p_m$, the $C^2$ norm of $\mathsf p$ and the Ricci curvature and the second fundamental form of the manifold. 
\end{enumerate}
\end{lemma}

Define $\mathsf H_{\epsilon}$ to be the integral operator associated with the heat kernel on the manifold $M$ with the diffusion time $\epsilon^2$; that is, for a suitable function $f$, we have
\begin{equation}
\mathsf H_{\epsilon}f(x):=\int_M \mathsf{H}(x,y,\epsilon^2)f(y)dV_M(y)\,.
\end{equation} 
Let
\begin{align}
R_{\epsilon}:=\frac{I-\mathsf H_{\epsilon}}{\epsilon^2}-\frac{I-T_{\epsilon}}{\epsilon^2}. \nonumber 
\end{align}
We discuss the properties of $R_{\epsilon}$ in the next two lemmas.

\begin{lemma}\label{property 1 of R epsilon}
Let $f \in L^2(M)$. There exists a constant $C_3$ depending on  $\mathsf p_m$,  the $C^2$ norm of $\mathsf p$ and the curvature of the manifold $M$, such that when $\epsilon>0$ is sufficiently small, $\|R_{\epsilon}f\|_2 \leq C_3\|f\|_2$.
\end{lemma}
\begin{proof}
By  Lemma \ref{properties of T epsilon} (c),
\begin{align}
T_{\epsilon}f(x)=& \frac{\int_{M}k_{\epsilon}(x,y)f(y)dV_M}{\int_{M}k_{\epsilon}(x,y)dV_M}+O(\epsilon^2) \nonumber  \\
=& \frac{\int_{M} (4\pi \epsilon^2)^{-d/2} k_{\epsilon}(x,y)f(y)dV_M}{\int_{M}(4\pi \epsilon^2)^{-d/2} k_{\epsilon}(x,y)dV_M} +O(\epsilon^2)\,.\nonumber 
\end{align}
Note that $\int_{M}(4\pi \epsilon^2)^{-d/2} k_{\epsilon}(x,y)dV_M=1+O(\epsilon^2)$ by Lemma \ref{properties of T epsilon} (a). Hence,
\begin{align}
T_{\epsilon}f(x)=& \int_{M}(4\pi \epsilon^2)^{-d/2}k_{\epsilon}(x,y)f(y)dV_M+O(\epsilon^2)\,. \nonumber 
\end{align}
{where the implied} constant in $O(\epsilon^2)$ depends on $\mathsf p_m$,  the $C^2$ norm of $\mathsf p$ and $\|f\|_2$. Thus,
\begin{align}
\|R_{\epsilon}f\|_2 \leq &\, \Big\|\frac{(I-\mathsf H_{\epsilon})f(x)}{\epsilon^2}-\frac{f(x)-\int_{M}(4\pi \epsilon^2)^{-d/2}k_{\epsilon}(x,y)f(y)dV_M}{\epsilon^2}\Big\|_2+\tilde{C}_1 \|f\|_2 \\ \nonumber
=&\, \Big\|\frac{\mathsf H_{\epsilon}f(x)-\int_{M}(4\pi \epsilon^2)^{-d/2}k_{\epsilon}(x,y)f(y)dV_M}{\epsilon^2}\Big\|_2+\tilde{C}_1 \|f\|_2\,,
\end{align}
where $\tilde{C}_1$ depends on $\mathsf p_m$ {and} the $C^2$ norm of $\mathsf p$. 
To control $\|\mathsf H_{\epsilon}f(x)-\int_{M}(4\pi \epsilon^2)^{-d/2}k_{\epsilon}(x,y)f(y)dV_M\|_2$, note that
\begin{align}\label{similiarity between kernels}
\Big\|\mathsf H_{\epsilon}f(x)-\int_{M}&(4\pi \epsilon^2)^{-d/2}k_{\epsilon}(x,y)f(y)dV_M\Big\|_2 \leq \Big\|\mathsf H_{\epsilon}f(x)-\int_{M}(4\pi \epsilon^2)^{-d/2} e^{-\frac{d(x,y)^2}{4\epsilon^2}}f(y)dV_M\Big\| \nonumber \\ 
& +\Big\|\int_{M}(4\pi \epsilon^2)^{-d/2} e^{-\frac{d(x,y)^2}{4\epsilon^2}}f(y)dV_M-\int_{M}(4\pi \epsilon^2)^{-d/2}k_{\epsilon}(x,y)f(y)dV_M\Big\|\,.
\end{align}
The former term on the right hand side can be bounded by Lemma \ref{bounds on heat kernel} (a), while the later term can be bounded by using the expansion of $d^2(x,y)$ in terms of $\|\iota(x)-\iota(y)\|^2_{\mathbb{R}^D}$.  In conclusion, we have
\begin{align}
\Big\|\frac{\mathsf H_{\epsilon}f(x)-\int_{M}(4\pi \epsilon^2)^{-d/2}k_{\epsilon}(x,y)f(y)dV_M}{\epsilon^2}\Big\|_2 \leq \tilde{C}_2 \|f\|_2,
\end{align}
where $\tilde{C}_2$ depends on the curvature of the manifold. {We finish the proof by setting} $C_3=\tilde{C}_1+\tilde{C}_2$.
\end{proof}

{
\begin{remark}\label{rely on Gaussian}
Note that the proof of the above lemma relies on the fact that $k_{\epsilon}(x,y)$ is Gaussian. In particular, we use the expansion of $\mathsf{H}(x,y,\epsilon^2)$ and $(4\pi \epsilon^2)^{-d/2}k_{\epsilon}(x,y)$ in terms of $(4\pi \epsilon^2)^{-d/2} e^{-\frac{d(x,y)^2}{4\epsilon^2}}$ in \eqref{similiarity between kernels}, which is related to the parametrix method \cite{rosenberg1997laplacian}.
\end{remark}}

\begin{lemma}\label{property 2 of R epsilon}
Fix $K\in \mathbb{N}$. Suppose $\phi_K \in C^{\infty}(M)$ is the $K$-th eigenfunction of $\Delta$ and  $\|\phi_K\|_2=1$. Then, there is a constant $C_4$ depending on  $\mathsf p_m$, the $C^2$ norm of $\mathsf p$, the Ricci curvature {and the second fundamental form of the manifold,} so that when $\epsilon>0$ is sufficiently small,
\begin{align}
\|R_\epsilon \phi_K\|_2 \leq C_4 \epsilon^2 \big(1+\lambda_K^{d/2+5}\big)\,. \nonumber 
\end{align}
\end{lemma}

We mention that this is the main step controlling the $L^2$ convergence rate due to the fundamental Sobolev embedding theorem. {Also note how the bound depends on the associated eigenvalue. By the lower bound of the eigenvalues shown in Lemma \ref{laplace eigenvalue lower bound}, the bound exponentially increases as $K$ increases.}

\begin{proof}
By Lemma  \ref{properties of T epsilon} (d), {$R_\epsilon  \phi_K(x)\leq \tilde C_1\|\phi_K\|_{C^4(M)}$, where} $\tilde{C}_1$ is a constant depending on $\mathsf p_m$, the $C^2$ norm of $\mathsf p$ and the Ricci curvature and {the second fundamental form of the manifold.} Moreover, by the Sobolev Embedding theorem \cite[Theorem 9.2]{palais1968foundations}, there is a constant $\tilde{C}_2>0$ depending on the Ricci curvature of the manifold so that 
\begin{align}
\| \phi_K\|_{C^4(M)} \leq \tilde{C}_2 \| \phi_K\|_{H^{d/2+5}} \leq  \tilde{C}_2 (1+ \|\Delta^{d/2+5}  \phi_K\|_{2}) \leq \tilde{C}_2 \big(1+\lambda_K^{d/2+5}\big)\,. \nonumber 
\end{align}
By taking $C_4=\tilde{C}_1 \tilde{C}_2$, the conclusion follows. 
\end{proof}

\subsection{Spectral convergence of $\frac{I-T_{\epsilon}}{\epsilon^2}$ to $-\Delta$}

In this subsection we show the spectral convergence of $\frac{I-T_{\epsilon}}{\epsilon^2}$ to $-\Delta$. 
We prove the main results in this subsection by extending the argument in \cite{coifman2006diffusion,belkin2007convergence} with several new arguments and results. The proof is long and contains results of independent interest, so we divide it into several parts.

Recall that the Laplace-Beltrami operator $\Delta$ is not a compact operator. 
To show the convergence of the eigenvalues and the corresponding eigenvectors, we compare $\frac{I-T_{\epsilon}}{\epsilon^2}$ and $\frac{I-\mathsf H_{\epsilon}}{\epsilon^2}$. For the $i$-th eigenvalue  of $-\Delta$, $\lambda_i$, we know that the $i$-th eigenvalue of $\frac{I-\mathsf H_{\epsilon}}{\epsilon^2}$ is $\frac{1-e^{-\lambda_i\epsilon^2}}{\epsilon^2}$, which converges to $\lambda_i$ when $\epsilon$ goes to $0$. We will show the convergence of the $i$-th eigenvalue of $\frac{I-T_{\epsilon}}{\epsilon^2}$ to the $i$-th eigenvalue of  $\frac{I-\mathsf H_{\epsilon}}{\epsilon^2}$ when $i$ is fixed. Moreover, since $\frac{I-\mathsf H_{\epsilon}}{\epsilon^2}$ and $\Delta$ share the same eigenfunctions, the convergence of eigenfunctions will follow from the convergence of eigenvalues. 

\begin{proposition}\label{T epsilon and Delta}
Assume that the eigenvalues of $\Delta$ are simple. Suppose $(\lambda_{i,\epsilon}, \phi_{i,\epsilon})$ is the $i$-th eigenpair of $\frac{I-T_{\epsilon}}{\epsilon^2}$ and $(\lambda_{i}, \phi_{i})$ is the $i$-th eigenpair of $-\Delta$. Assume both $\phi_{i,\epsilon}$ and $\phi_{i}$ are normalized in the $L^2$ norm. 
For $K\in \mathbb{N}$, denote 
\begin{equation}\label{defintion: mathsf Gamma K}
\mathsf \Gamma_K:=\min_{1 \leq i \leq K}\textup{dist}(\lambda_i, \sigma(-\Delta)\setminus \{\lambda_i\})\,.
\end{equation}
Suppose {\eqref{epsilon less than some constant} is satisfied and $\epsilon>0$ is sufficiently small.} Then, there are $a_i \in \{-1, 1\}$ such that for all $1 \leq i < K$, 
\begin{align}
|\lambda_{i,\epsilon}-\lambda_{i}|  \leq \epsilon^{3/2},\quad 
\|a_i\phi_{i,\epsilon}-\phi_i\|_2 \leq \epsilon^{3/2},\quad 
\|a_i\phi_{i,\epsilon}-\phi_{i}\|_{\infty}  \leq \epsilon. \, \label{Proposition main result statement} 
\end{align}
\end{proposition}

The results shown in \eqref{Proposition main result statement} have been highly simplified. {A more delicate error bound depending on $\lambda_K$ and other terms} can be found in \eqref{L^2 convergence of eigenfunctions} and \eqref{Proposition:complicated L infty bound}.
Note that $\mathsf\Gamma_K=\min\{\gamma_i(\Delta)|\,i=1,\ldots,K\}$, where $\gamma_i(\Delta)$ is defined in \eqref{Proof proposition 1 Definition gamma i}. The requirement that $\epsilon\leq \mathsf\Gamma_K$ is somehow intuitive -- to recover the first few eigenvalues with small gaps, we need a small bandwidth.

We now describe the idea of proving Proposition \ref{T epsilon and Delta}. We first use Lemma \ref{ratio of eigenvalues} (a) to obtain a convergence of the eigenvalues at a possibly {\em slow} convergence rate. With the convergence of eigenvalues, we can apply Lemma \ref{ratio of eigenvalues} (b) to show that the angle between the $i$-th eigenfunction of $\frac{I-T_{\epsilon}}{\epsilon^2}$ and the $i$-th eigenfunction of $-\Delta$ in $L^2(M)$ is well bounded from above; for example, this angle is at most $\frac{\pi}{6}$. 
Then we use Lemma \ref{property 2 of R epsilon} and Lemma \ref{ratio of eigenvalues} (c) to achieve the final convergence rate of eigenvalues. With the eigenvalue convergence, we apply Lemma \ref{ratio of eigenvalues} (b) again to show the convergence rate of eigenfunctions in $L^2(M)$ through the convergence rate of the eigenvalues and Lemma \ref{ratio of eigenvalues} (b). At last, we turn the convergence in $L^2(M)$ to the convergence in $L^\infty(M)$ with a slower convergence rate by the Sobolev embedding.

We emphasize that we {\em do not} directly apply the Davis-Kahan theorem to obtain the $L^2$ control of the eigenfunctions, since we {\em do not} have the convergence of $\frac{I-T_{\epsilon}}{\epsilon^2}$ to $\frac{I-\mathsf H_{\epsilon}}{\epsilon^2}$ in the operator norm. Note that Lemma \ref{property 1 of R epsilon} does not imply the convergence, and Lemma \ref{property 2 of R epsilon} only holds for a subset of the whole spectrum.

\begin{proof}
The proof is divided into few steps.

\underline{\textbf{Step 1, control $R_{\epsilon}$: }}
{In this step, we control $
\frac{(R_{\epsilon}f, f)}{|(\frac{I-\mathsf H_{\epsilon}}{\epsilon^2}f,f)|}$.}
For any $f\in L^2(M)$ with $\|f\|_2=1$, we have $f=\sum_{i=0}^\infty a_i \phi_i$, where $ \sum_{i=0}^\infty a_i^2=1$.
Take $J=J(\epsilon)\in \mathbb{N}$, which we will fix later.
Let $f=f_1+f_2$, where $f_1=\sum_{i=0}^J a_i \phi_i $ and $f_2=\sum_{i=J+1}^\infty a_i \phi_i$.
Then, 
\begin{align}
|(R_{\epsilon}f, f)|{\leq }&\, |(R_{\epsilon}f_1, f_1)|+2|(R_{\epsilon}f_1, f_2)|+|(R_{\epsilon}f_2, f_2)| \nonumber \\
\leq &\, \|R_{\epsilon}f_1\|_2  \|f_1\|_2+2 \|R_{\epsilon}f_1\|_2  \|f_2\|_2+\|R_{\epsilon}f_2\|_2  \|f_2\|_2 \nonumber \\
\leq &\, 3 \|R_{\epsilon}f_1\|_2 +\|R_{\epsilon}f_2\|_2 \|f_2\|_2 \,.\nonumber 
\end{align}
By Lemma \ref{property 2 of R epsilon},
\begin{align}
\|R_{\epsilon}f_1\|_2&\,=\Big\|R_{\epsilon}\Big(\sum_{i=1}^J a_i \phi_i\Big)\Big\|_2 \leq \sum_{i=1}^J \|R_{\epsilon}\phi_i\|_2 \nonumber\\
&\,\leq \sum_{i=1}^J C_4 \epsilon^2 (1+\lambda_i^{d/2+5}) \leq C_4 J \epsilon^2 (1+\lambda_J^{d/2+5}). \nonumber 
\end{align}
Moreover,
\begin{align}
\Big|\Big(\frac{I-\mathsf H_{\epsilon}}{\epsilon^2}f,\,f\Big)\Big|&\,=\Big(\sum_{i=1}^{\infty} \frac{1-e^{-\epsilon^2 \lambda_i}}{\epsilon^2}a_i \phi_i, \sum_{i=1}^{\infty}a_i \phi_i\Big)\label{Bound in Repsilon control step 1}\\
&\,=\sum_{i=1}^\infty a_i^2 \frac{1-e^{-\epsilon^2 \lambda_i}}{\epsilon^2}>\frac{1}{2} \min\Big(\lambda_1,\, \frac{1}{\epsilon}\Big)=\frac{\lambda_1}{2}\,, \nonumber 
\end{align}
where we use the fact that when $0<\epsilon^2<0.1$, we have \cite[Theorem 4.3]{belkin2007convergence}
\begin{align}\label{bound for one minus expo square}
\frac{1-e^{-x\epsilon^2}}{\epsilon^2} \geq \frac{1}{2}\min\Big(x,\, \frac{1}{\epsilon}\Big)\,.
\end{align}
Therefore,
\begin{align}
\frac{3 \|R_{\epsilon}f_1\|_2}{|(\frac{I-\mathsf H_{\epsilon}}{\epsilon^2}f,f)|}\leq  \frac{6C_4 J}{\lambda_1}\big(1+\lambda_J^{d/2+5}\big) \epsilon^2. \nonumber 
\end{align}
{By Lemma \ref{laplace eigenvalue lower bound}, when $\epsilon>0$ is sufficiently small,} we can take $J \asymp \epsilon^{-\frac{3d}{2(d+10)}}$ {so that} $\epsilon^{-\frac{3}{2}} \leq \lambda_J^{d/2+5} \leq 2 \epsilon^{-\frac{3}{2}}$. {Thus,}  
\begin{align}
\frac{3 \|R_{\epsilon}f_1\|_2}{|(\frac{I-\mathsf H_{\epsilon}}{\epsilon^2}f,f)|}\leq  \frac{18 C_4}{\lambda_1}\epsilon^{\frac{d^2+21d+140}{2(d+10)}}\,. \nonumber 
\end{align}
For $f_2$, which is in general a $L^2$ function, we apply Lemma \ref{property 1 of R epsilon}:
\begin{align}
\|R_{\epsilon}f_2\|_2\|f_2\|_2 \leq C_3\|f_2\|^2_2\,. \nonumber 
\end{align}
Since $\{\phi_i\}$ forms an orthonormal basis, we have $\Big(\frac{I-\mathsf H_{\epsilon}}{\epsilon^2}f_2,\,f_1\Big)=\Big(\frac{I-\mathsf H_{\epsilon}}{\epsilon^2}f_1,\,f_2\Big)=0$. {For $f_2$, we control $(\frac{I-\mathsf H_{\epsilon}}{\epsilon^2}f,\,f)$ in a way different from \eqref{Bound in Repsilon control step 1}:}
\begin{align}
\Big|\Big(\frac{I-\mathsf H_{\epsilon}}{\epsilon^2}f,\,f\Big)\Big|=&\, \Big(\frac{I-\mathsf H_{\epsilon}}{\epsilon^2}f_1,\,f_1\Big)+\Big(\frac{I-\mathsf H_{\epsilon}}{\epsilon^2}f_2, \,f_2\Big) \geq \Big(\frac{I-\mathsf H_{\epsilon}}{\epsilon^2}f_2,\, f_2\Big) \nonumber  \\
= &\, \sum_{i=J+1}^\infty a_i^2 \frac{1-e^{-\epsilon^2 \lambda_i}}{\epsilon^2} \geq \frac{1}{2} \min\Big(\lambda_{J+1}, \frac{1}{\epsilon}\Big)\|f_2\|^2_2 \geq \frac{1}{2} \min\Big(\lambda_{J}, \frac{1}{\epsilon}\Big)\|f_2\|^2_2\,, \nonumber 
\end{align}
where we use \eqref{bound for one minus expo square} in the second last step.
Hence, {by the choice of $J$, we have}
\begin{align}
\frac{ \|R_{\epsilon}f_2\|_2\|f_2\|_2}{|(\frac{I-\mathsf H_{\epsilon}}{\epsilon^2}f,f)|}\leq  2C_3\max\Big(\epsilon, \frac{1}{\lambda_J}\Big) \leq 2C_3 \epsilon^{\frac{3}{d+10}}\,. \nonumber 
\end{align}
By putting the above together, since $\epsilon^{\frac{3}{d+10}}>\epsilon^{\frac{d^2+21d+140}{2(d+10)}}$, {by the lower bound of $\lambda_1$ stated in Lemma \ref{laplace eigenvalue lower bound} (or the Faber-Krahn lower bound of the first non-trivial eigenvalue),} there is a constant $\tilde{C}_1>0$ satisfying
\begin{align}
\frac{18 C_4}{\lambda_1}+2C_3\leq \tilde{C}_1\nonumber
\end{align}
such that
\begin{align}
\frac{(R_{\epsilon}f, f)}{|(\frac{I-\mathsf H_{\epsilon}}{\epsilon^2}f,f)|}\leq  \tilde{C}_1 \epsilon^{\frac{3}{d+10}}\,. \label{Proof proposition 1 difference ratio bound} 
\end{align}

\underline{\textbf{Step 2, bound the angle between two eigenfunctions: }}
Suppose $\lambda_{i,\epsilon}$ is the $i$-th eigenvalue of $\frac{I-T_{\epsilon}}{\epsilon^2}$. Note that the $i$-th eigenvalue of $\frac{\mathsf I-\mathsf H_{\epsilon}}{\epsilon^2}$ is $\frac{1-e^{-\epsilon^2 \lambda_i}}{\epsilon^2}$. By Lemma \ref{ratio of eigenvalues} (a), \eqref{Proof proposition 1 difference ratio bound} lead to 
\begin{align}
\Big|\lambda_{i,\epsilon}-\frac{1-e^{-\epsilon^2 \lambda_i}}{\epsilon^2}\Big|\leq \tilde{C}_1 \epsilon^{\frac{3}{d+10}}\frac{1-e^{-\epsilon^2 \lambda_i}}{\epsilon^2}\, \nonumber 
\end{align}
for all $i$.
Hence, for any $i{\in \mathbb{N}}$, we have
\begin{align}
|(1-\lambda_{i,\epsilon}\epsilon^2)-e^{-\epsilon^2 \lambda_i}|&\,\leq \tilde{C}_1 \epsilon^{\frac{3}{d+10}}(1-e^{-\epsilon^2 \lambda_i})\leq \tilde{C}_1\epsilon^{\frac{3}{d+10}+2}\lambda_i  \label{weak convergence rate of EV of T epsilon}\,,
\end{align}
{where we use the trivial bound $|1-e^{-x}|\leq x$ for $x>0$. }
Hence, trivially
\begin{align}
&\|(1-\lambda_{i,\epsilon}\epsilon^2)\phi_{i}-e^{-\epsilon^2 \lambda_i}\phi_{i}\|_2 \leq \tilde{C}_1 \epsilon^{\frac{3}{d+10}+2}\lambda_i \,. \label{Proof proposition 1 eigenfunctions difference}
\end{align}
On the other hand, by Lemma \ref{property 2 of R epsilon}, when $i < K$, we have 
\begin{align}
\|T_{\epsilon}\phi_{i}-e^{-\epsilon^2 \lambda_i} \phi_i\|_2=\|(I-\mathsf H_\epsilon)\phi_{i}-(I-T_{\epsilon}) \phi_i\|_2 \leq C_4 \epsilon^4 (1+\lambda_K^{d/2+5}). \label{Proof proposition 1 eigenfunctions difference qq}
\end{align}
Thus, if we combine \eqref{Proof proposition 1 eigenfunctions difference} and \eqref{Proof proposition 1 eigenfunctions difference qq}, when $i < K$, we have
\begin{align} \label{weak bounds 2}
\Big\|\frac{I-T_{\epsilon}}{\epsilon^2}\phi_{i}-\lambda_{i,\epsilon}\phi_{i}\Big\|_2 \leq C_4 \epsilon^2 (1+\lambda_K^{d/2+5})+\tilde{C}_1 \epsilon^{\frac{3}{d+10}} \lambda_K \,. 
\end{align}
Next, we first find the lower bounds for the gaps between $1-\lambda_{i-1,\epsilon}\epsilon^2$, $1-\lambda_{i,\epsilon}\epsilon^2$ and $1-\lambda_{i+1,\epsilon}\epsilon^2$. First, note that (\ref{weak convergence rate of EV of T epsilon}) implies that
\begin{align}
&1-\lambda_{i+1,\epsilon}\epsilon^2 \leq e^{-\epsilon^2\lambda_{i+1}}+\tilde{C}_1 \epsilon^{\frac{3}{d+10}+2}\lambda_{i+1}\,, \nonumber \\
&1-\lambda_{i,\epsilon}\epsilon^2 \geq e^{-\epsilon^2\lambda_{i}}-\tilde{C}_1 \epsilon^{\frac{3}{d+10}+2}\lambda_{i}\,. \nonumber
\end{align}
Thus,
\begin{equation}
(1-\lambda_{i,\epsilon}\epsilon^2)-(1-\lambda_{i+1,\epsilon}\epsilon^2) \geq (e^{-\epsilon^2\lambda_{i}}-e^{-\epsilon^2\lambda_{i+1}})-\tilde{C}_1 \epsilon^{\frac{3}{d+10}+2}(\lambda_i+\lambda_{i+1})\, \label{Proof proposition bound difference}.
\end{equation} 
{To bound $(1-\lambda_{i,\epsilon}\epsilon^2)-(1-\lambda_{i+1,\epsilon}\epsilon^2)$, we control $e^{-\epsilon^2\lambda_{i}}-e^{-\epsilon^2\lambda_{i+1}}$ and $\tilde{C}_1 \epsilon^{\frac{3}{d+10}+2}(\lambda_i+\lambda_{i+1})$ separately.} When 
\begin{align}
\lambda_K \epsilon^2\leq 1\,, \label{condition 1}
\end{align}
{which holds under the assumption of $\epsilon$,}
we have $\epsilon^2 (\lambda_{i+1}-\lambda_{i}) \leq \epsilon^2 \lambda_{i+1} \leq \epsilon^2 \lambda_K \leq 1$ since $i<K$. {Combined with the fact that} $e^{-\epsilon^2 \lambda_{i}}\geq e^{-1}$, we have
\begin{equation}
e^{-\epsilon^2 \lambda_{i}}-e^{-\epsilon^2 \lambda_{i+1}}=e^{-\epsilon^2 \lambda_{i}}\big[1-e^{-\epsilon^2 (\lambda_{i+1}-\lambda_{i})}\big] \geq  \frac{1}{e} \cdot \frac{1}{2} (\lambda_{i+1}-\lambda_{i})\epsilon^2 \geq \frac{1}{6} (\lambda_{i+1}-\lambda_{i})\epsilon^2\,. \label{Proof proposition bound difference1}
\end{equation}
{Next we control $\tilde{C}_1 \epsilon^{\frac{3}{d+10}+2}(\lambda_i+\lambda_{i+1})$, and} we want {to have the following bound:} 
\begin{align}
-\tilde{C}_1 \epsilon^{\frac{3}{d+10}+2}(\lambda_{i+1}+\lambda_{i}) \geq -\frac{1}{12} (\lambda_{i+1}-\lambda_{i}) \epsilon^2\,, \label{Proof proposition bound difference2}
\end{align} 
which is equivalent to {requiring}
\begin{align}
\epsilon^{\frac{3}{d+10}}\leq \frac{1}{12\tilde{C}_1} \frac{\lambda_{i+1}-\lambda_{i}}{\lambda_{i+1}+\lambda_{i}} \leq \frac{\mathsf\Gamma_K}{24\tilde{C}_1\lambda_K}.\label{condition 2}
\end{align}
By plugging \eqref{Proof proposition bound difference1} and \eqref{Proof proposition bound difference2} into \eqref{Proof proposition bound difference}, we have 
\begin{equation}
(1-\lambda_{i,\epsilon}\epsilon^2)-(1-\lambda_{i+1,\epsilon}\epsilon^2)  \geq  \frac{1}{12} (\lambda_{i+1}-\lambda_{i}) \epsilon^2\,.\nonumber
\end{equation} 
A similar argument leads to 
\begin{equation}
\gamma_i\Big(\frac{I-T_{\epsilon}}{\epsilon^2}\Big)\geq  \frac{1}{12} \min\{(\lambda_i-\lambda_{i-1}),\,(\lambda_{i+1}-\lambda_{i})\} \geq\frac{1}{12}\mathsf\Gamma_K \,,\label{Proof proposition 1 gap of the ith eigenvalue}
\end{equation}
{by the assumption of $\mathsf\Gamma_K$.}
Next, with (\ref{weak bounds 2}) and \eqref{Proof proposition 1 gap of the ith eigenvalue}, we apply Lemma \ref{ratio of eigenvalues} (b) to get 
\begin{equation}
|(\phi_i,\phi_{i,\epsilon})|\geq 1-12\frac{C_4 \epsilon^2 (1+\lambda_K^{d/2+5})+\tilde{C}_1 \epsilon^{\frac{3}{d+10}} \lambda_K}{\min\{(\lambda_i-\lambda_{i-1}),\,(\lambda_{i+1}-\lambda_{i})\}}{\geq 1-12\frac{C_4 \epsilon^2 (1+\lambda_K^{d/2+5})+\tilde{C}_1 \epsilon^{\frac{3}{d+10}} \lambda_K}{\mathsf\Gamma_K}}\,.\nonumber
\end{equation}
If 
\begin{align}
12\frac{C_4 \epsilon^2 (1+\lambda_K^{d/2+5})+\tilde{C}_1 \epsilon^{\frac{3}{d+10}} \lambda_K}{\mathsf\Gamma_K} <\frac{1}{2}\,, \nonumber
\end{align}
which holds when 
\begin{align}
& \epsilon^2<\frac{\mathsf\Gamma_K}{48 C_4  (1+\lambda_K^{d/2+5})}  \label{condition 3}
\end{align}
\and
\begin{align}
& \epsilon^{\frac{3}{d+10}}< \frac{\mathsf\Gamma_K}{48\tilde{C}_1 \lambda_K}\,, \label{condition 4}
\end{align}
then {we have}
\begin{equation}
|(\phi_{i}, \phi_{i,\epsilon})| \geq \frac{1}{2}\,.\label{Proof proposition 1 eigenfunctions angle}
\end{equation} 

\underline{\textbf{Step 3, find the final bounds: }}
With the lower bound of the angle among eigenfunctions shown in \eqref{Proof proposition 1 eigenfunctions angle}, we boost the bound of the eigenvalues and hence the eigenfunctions.
By Lemma \ref{property 2 of R epsilon}, \eqref{Proof proposition 1 eigenfunctions angle}, and Lemma \ref{ratio of eigenvalues} (c),  we have
\begin{align}
\Big|\lambda_{i,\epsilon}-\frac{1-e^{-\epsilon^2 \lambda_i}}{\epsilon^2}\Big| \leq \frac{\|R_\epsilon(\phi_i)\|}{|(\phi_i,\phi_{i,\epsilon})|}\leq 2 C_4 \epsilon^2 \big(1+\lambda_K^{d/2+5}\big)\,. \label{proof Proposition 1 enhanced eigenvalue bound}
\end{align}
Note that $-x^2\leq1-e^{-x}-x\leq 0$, {which leads to}
\begin{align}
\Big|\frac{1-e^{-\epsilon^2 \lambda_i}}{\epsilon^2}-\lambda_i\Big| \leq \lambda^2_i\epsilon^2.
\end{align}
By summing up the above two inequalities, 
\begin{align}
|\lambda_{i,\epsilon}- \lambda_i| \leq \Big(2 C_4 +2 C_4\lambda_K^{d/2+5}+\lambda^2_K\Big)\epsilon^2. \nonumber
\end{align}
Taking $\tilde{C}_2=\max\{2C_4, 1\}$, we have 
\begin{align}\label{Strong convergence of EV}
|\lambda_{i,\epsilon}- \lambda_i| \leq \tilde{C}_2 \Big(1+\lambda_K^{d/2+5}+\lambda^2_K\Big)\epsilon^2\,.
\end{align}
{Next}, by \eqref{proof Proposition 1 enhanced eigenvalue bound}, we have
\begin{align}
\Big\|\lambda_{i,\epsilon}\phi_i-\Big(\frac{1-e^{-\epsilon^2 \lambda_i}}{\epsilon^2}\Big)\phi_i\Big\|_2 \leq 2 C_4 \epsilon^2 (1+\lambda_K^{d/2+5})\,. \label{proof Proposition 1 enhanced eigenvector bound1}
\end{align}
By Lemma \ref{property 2 of R epsilon}, 
\begin{align}
\Big\|\frac{I-T_{\epsilon}}{\epsilon^2}\phi_{i}-\frac{I-\mathsf H_{\epsilon}}{\epsilon^2} \phi_i\Big\|_2=\Big\|\frac{I-T_{\epsilon}}{\epsilon^2}\phi_{i}-\Big(\frac{1-e^{-\epsilon^2 \lambda_i}}{\epsilon^2}\Big)\phi_i\Big\|_2 \leq C_4 \epsilon^2(1+\lambda_K^{d/2+5})\,. \label{proof Proposition 1 enhanced eigenvector bound2}
\end{align}
So, by combining \eqref{proof Proposition 1 enhanced eigenvector bound1} and \eqref{proof Proposition 1 enhanced eigenvector bound2}, we have for all $i\leq K$:
\begin{align}
\Big\|\frac{I-T_{\epsilon}}{\epsilon^2}\phi_{i}-\lambda_{i,\epsilon} \phi_i\Big\|_2 \leq 3C_4 \epsilon^2\Big(1+\lambda_K^{d/2+5}\Big)\,. \nonumber
\end{align}
Recall that in \eqref{Proof proposition 1 gap of the ith eigenvalue}, the eigengap for $\lambda_{i,\epsilon}$ {is controlled by} $\gamma_i\Big(\frac{I-T_{\epsilon}}{\epsilon^2}\Big) \geq \frac{1}{12}\mathsf\Gamma_K$. As a result, by Lemma \ref{ratio of eigenvalues} (b), there are $a_i \in \{-1, 1\}$ such that 
\begin{align}
\|a_i\phi_{i,\epsilon}-\phi_i\|_2 \leq  \frac{72 C_4 \epsilon^2(1+\lambda_K^{d/2+5})}{\mathsf\Gamma_K}. \nonumber
\end{align}
Let us temporarily define $C_{1,K}=\tilde{C}_2 \Big(1+\lambda_K^{d/2+5}+\lambda^2_K\Big)$ and $C_{2,K}= \frac{72 C_4 (1+\lambda_K^{d/2+5})}{\mathsf\Gamma_K}$. We have
\begin{align} 
& |\lambda_{i,\epsilon}- \lambda_i| \leq C_{1,K} \epsilon^{2}, \label{L^2 convergence of eigenvalues}\\
& \|a_i\phi_{i,\epsilon}-\phi_i\|_2 \leq  C_{2,K} \epsilon^{2}\,. \label{L^2 convergence of eigenfunctions}
\end{align}

With the above $L^2$ bound of the eigenfunction estimate, we use the Sobolev Embedding theorem \cite[Theorem 9.2]{palais1968foundations} to control the convergence rate in $L^\infty$ norm. There is a constant $\tilde{C}_3$ depending on the Ricci curvature of the manifold, such that 
\begin{align}
\|a_i\phi_{i,\epsilon}-\phi_{i}\|_{\infty} \leq &\, \tilde{C}_3(\|a_i\phi_{i,\epsilon}-\phi_{i}\|_{2}+\|\Delta^{d}(a_i\phi_{i,\epsilon}-\phi_{i})\|_{2}) \nonumber \\
= &\,  \tilde{C}_3\Big[ \|a_i\phi_{i,\epsilon}-\phi_{i}\|_{2}+ \Big\|a_i\Big(\Delta^d-\Big(\frac{T_{\epsilon}-I}{\epsilon^2}\Big)^{d}\Big)\phi_{i,\epsilon}+\Big(\frac{T_{\epsilon}-I}{\epsilon^2}\Big)^{d}a_i\phi_{i,\epsilon}-\Delta^d\phi_{i}\Big\|_{2} \Big] \nonumber \\
\leq &\, \tilde{C}_3\Big\|a_i\phi_{i,\epsilon}-\phi_{i}\Big\|_{2}+ \tilde{C}_3\Big\|\Big(\frac{T_{\epsilon}-I}{\epsilon^2}\Big)^{d}a_i\phi_{i,\epsilon}-\Delta^d\phi_{i}\Big\|_{2}\nonumber\\
&\qquad+ \tilde{C}_3 \Big\|\Big(\Delta^d-\Big(\frac{T_{\epsilon}-I}{\epsilon^2}\Big)^{d}\Big)\phi_{i,\epsilon}\Big\|_{2}.  \label{bounds of infinity norm}
\end{align}
Now we bound the right hand side term by term. {The second term of \eqref{bounds of infinity norm} is controlled by}
\begin{align}
 & \Big\|\Big(\frac{T_{\epsilon}-I}{\epsilon^2}\Big)^{d}a_i\phi_{i,\epsilon}-\Delta^d\phi_{i}\Big\|_{2}=\|\lambda^{d}_{i,\epsilon}a_i\phi_{i,\epsilon}-\lambda^{d}_{i}\phi_{i}\|_{2} \label{term 2} \\
\leq &\,\|\lambda^{d}_{i,\epsilon}a_i\phi_{i,\epsilon}-\lambda^{d}_{i}a_i\phi_{i,\epsilon}+\lambda^{d}_{i}a_i\phi_{i,\epsilon}-\lambda^{d}_{i}\phi_{i}\|_2  \nonumber \\
\leq & \,|\lambda^{d}_{i,\epsilon}-\lambda^{d}_{i}|+\lambda^{d}_{K}\|a_i\phi_{i,\epsilon}-\phi_{i}\|_{2} 
\leq  2^d d \lambda^{d-1}_{K} |\lambda_{i,\epsilon}-\lambda_{i}|+\lambda^{d}_{K}\|a_i\phi_{i,\epsilon}-\phi_{i}\|_{2} \nonumber\,,
\end{align}
{where} in the last step, we use the bound $ \lambda_{i,\epsilon} \leq \lambda_i+C_{1,K} \epsilon^{2}\leq 2\lambda_K$, which holds when
\begin{align}
\epsilon^{2} \leq \frac{\lambda_K}{C_{1,K}} \label{condition 5}\,.
\end{align}
{For the third term of \eqref{bounds of infinity norm},}  {by Lemma \ref{properties of T epsilon} (d) and a straightforward expansion,  there is a constant $ \tilde{C}_4$ depending on $d$, $\mathsf p_m$, the $C^2$ norm of $\mathsf p$, the Ricci curvature, the second fundamental form of the manifold such that 
\begin{align}
\left\|\Big(\Delta^d-\Big(\frac{T_{\epsilon}-I}{\epsilon^2}\Big)^{d}\Big)\phi_{i,\epsilon}(x)\right\|_2 \leq & \tilde{C}_4 \left\|\Delta \phi_{i,\epsilon}-\frac{T_{\epsilon}-I}{\epsilon^2}\phi_{i,\epsilon}\right\|_2 \left\|\Big(\frac{T_{\epsilon}-I}{\epsilon^2}\Big)^{d+1}\phi_{i,\epsilon}\right\|_\infty.
\end{align}
Based on Lemma \ref{L^2 regularity} in Appendix, there is a constant $C'_6$ depending on $d$, the volume, the curvature and the diameter of $M$ such that
\begin{align}
&\left\|\Big(\Delta^d-\Big(\frac{T_{\epsilon}-I}{\epsilon^2}\Big)^{d}\Big)\phi_{i,\epsilon}(x)\right\|_2\label{term 3}\\
\leq\, & \tilde{C}_4 \left\|\Delta \phi_{i,\epsilon}-\frac{T_{\epsilon}-I}{\epsilon^2}\phi_{i,\epsilon}\right\|_2 \left\|\Big(\frac{T_{\epsilon}-I}{\epsilon^2}\Big)^{d+1}\phi_{i,\epsilon}\right\|_\infty\, \nonumber\\
\leq\, & \tilde{C}_4 \lambda^{d+1}_{i,\epsilon} (1+C_{1,K}+(C'_6+2\lambda_K) C_{2,K}+4\lambda_K^2)  \|\phi_{i,\epsilon}\|_\infty  \epsilon^2 \nonumber \\
\leq \,&  \tilde{C}_4 2^{d+1} \lambda^{d+1}_{K} (1+C_{1,K}+(C'_6+2\lambda_K) C_{2,K}+4\lambda_K^2)  \|\phi_{i,\epsilon}\|_\infty  \epsilon^2 . \nonumber
\end{align}
In the last step, we use again $ \lambda_{i,\epsilon} \leq \lambda_i+C_{1,K} \epsilon^{2}\leq 2\lambda_K$, when $\epsilon^{2} \leq \frac{\lambda_K}{C_{1,K}}$. 

Therefore, if we {plug} (\ref{L^2 convergence of eigenfunctions}), (\ref{term 2}) {and} (\ref{term 3}) into (\ref{bounds of infinity norm}), we have
\begin{align}
\|a_i\phi_{i,\epsilon}-\phi_{i}\|_{\infty} \leq&\, \tilde{C}_3 (C_{2,K}(1+\lambda^{d}_{K}) +2^d d \lambda^{d-1}_{K}C_{1,K})\epsilon^{2} \label{Proposition:complicated L infty bound}  \\
&+\tilde{C}_3\tilde{C}_4 2^{d+1} \lambda^{d+1}_{K} (1+C_{1,K}+(C'_6+2\lambda_K) C_{2,K}+4\lambda_K^2)  \|\phi_{i,\epsilon}\|_\infty  \epsilon^2 \nonumber \\
\leq &\, \tilde{C}_5(C_{2,K}(1+\lambda^{d}_{K}) +\lambda^{d-1}_{K}C_{1,K})\epsilon^{2} \nonumber \\
&+\tilde{C}_5\lambda^{d+1}_{K} (1+C_{1,K}+(C'_6+2\lambda_K) C_{2,K}+4\lambda_K^2)  \|\phi_{i,\epsilon}\|_\infty  \epsilon^2, \nonumber
\end{align}
where $\tilde{C}_5:=\tilde{C}_3  \max\{2^d d, 2^{d+1}\tilde{C}_4 \}$.
If we require $\tilde{C}_5(C_{2,K}(1+\lambda^{d}_{K}) +\lambda^{d-1}_{K}C_{1,K})\epsilon^{2} \leq 1$ and $\tilde{C}_5\lambda^{d+1}_{K} (1+C_{1,K}+(C'_6+2\lambda_K) C_{2,K}+4\lambda_K^2)  \leq \frac{1}{2}$, then we have 
\begin{align}
|\|\phi_{i,\epsilon}\|_\infty-\|\phi_{i}\|_\infty| \leq \|a_i\phi_{i,\epsilon}-\phi_{i}\|_{\infty} \leq \frac{1}{2}\|\phi_{i,\epsilon}\|_\infty+1.
\end{align}
Hence, by Lemma \ref{lemma hormander}, we have
\begin{align}
\|\phi_{i,\epsilon}\|_\infty \leq 2\|\phi_{i}\|_{\infty}+2 \leq  2(C_1 \lambda_K^{\frac{d-1}{4}}+1).
\end{align}
{Thus, \eqref{Proposition:complicated L infty bound} becomes}
\begin{align}
\|a_i\phi_{i,\epsilon}-\phi_{i}\|_{\infty} \leq&\,  \tilde{C}_5(C_{2,K}(1+\lambda^{d}_{K}) +\lambda^{d-1}_{K}C_{1,K})\epsilon^{2}  \\
&+2 \tilde{C}_5\lambda^{d+1}_{K}(1+C_{1,K}+(C'_6+2\lambda_K) C_{2,K}+4\lambda_K^2)(C_1 \lambda_K^{\frac{d-1}{4}}+1) \epsilon^{2}\,. \nonumber
\end{align}
{Finally,} if we require that 
\begin{align}
& \tilde{C}_5C_{2,K}(1+\lambda^{d}_{K}) \epsilon \leq \frac{1}{4} \,, \\
& \tilde{C}_5\lambda^{d-1}_{K}C_{1,K}\epsilon \leq \frac{1}{4}\,,\\
& 2 \tilde{C}_5  \lambda^{d+1}_{K}  \big(1+C_{1,K}+4\lambda_K^2\big) (C_1 \lambda_K^{\frac{d-1}{4}}+1) \epsilon \leq \frac{1}{4}\,,\\
&  2\tilde{C}_5 \lambda^{d+1}_{K} (C'_6+2\lambda_K) C_{2,K}(C_1 \lambda_K^{\frac{d-1}{4}}+1) \epsilon \leq \frac{1}{4}\,,
\end{align}
then $\|a_i\phi_{i,\epsilon}-\phi_{i}\|_{\infty} \leq \epsilon$.
}

{\underline{\textbf{Step 4, manage the relationship among parameters: }}
Finally,} we list the relationship in Table \ref{Table:Relations} between $\epsilon$ and $\lambda_K$ to obtain the convergence of eigenvalues and eigenvectors. 
\begin{table}[ht]\label{Table:Relations}
\caption{Relationships between $\epsilon$ and $\lambda_K$ used in this paper}
\begin{tabular}{l|l} 
\hline No. & Relationship \\ 
\hline \hline 
\\ [-.85em]
(1'') & {$\lambda_K \epsilon^2\leq \frac{1}{8}$} \\
\\ [-.85em]
(2'') & $\epsilon^{\frac{3}{d+10}}\leq \frac{1}{24\tilde{C}_1}\frac{\mathsf\Gamma_K}{\lambda_K}$ \\ 
\\ [-.85em]
(3'') & $\epsilon^2\leq\frac{\mathsf\Gamma_K}{48 C_4  (1+\lambda_K^{d/2+5})}$\\ 
\\ [-.85em]
(4'') & $\epsilon^{\frac{3}{d+10}}\leq \frac{\mathsf\Gamma_K}{48\tilde{C}_1 \lambda_K}$\\ 
\\ [-.85em]
(5'') & $\epsilon^{2} \leq \frac{\lambda_K}{\tilde{C}_2 \big(1+\lambda_K^{d/2+5}+\lambda^2_K\big)}$ \\
\\ [-.85em]
(6'') & $\tilde{C}_2 \big(1+\lambda_K^{d/2+5}+\lambda^2_K\big) \epsilon^{1/2} \leq 1$ \\
\\ [-.85em]
(7'') & $\frac{72 C_4 (1+\lambda_K^{d/2+5})}{\mathsf\Gamma_K} \epsilon^{1/2} \leq 1$ \\
\\ [-.85em]
(8'') & $\tilde{C}_5(1+\lambda^{d}_{K}) \frac{72 C_4 (1+\lambda_K^{d/2+5})}{\mathsf\Gamma_K}\epsilon\leq \frac{1}{4}$ \\
\\ [-.85em]
(9'') & $\tilde{C}_5 \tilde{C}_2 \lambda^{d-1}_{K} \big(1+\lambda_K^{d/2+5}+\lambda^2_K\big) \epsilon \leq \frac{1}{4}$ \\ 
\\ [-.85em]
(10'') & $2 \tilde{C}_5\lambda^{d+1}_{K}  \Big(1+\tilde{C}_2 \big(1+\lambda_K^{d/2+5}+\lambda^2_K\big)+4\lambda_K^2\Big) (C_1 \lambda_K^{\frac{d-1}{4}}+1) \epsilon \leq \frac{1}{4}$ \\
\\ [-.85em]
(11'') & $\epsilon \leq \frac{\min(\mathsf\Gamma_K,1)}{C_1 \lambda_K^{(d-1)/4}+\epsilon+1},$ \\
\\ [-.85em]
(12'') & $C_1\lambda^{(d-1)/4}_K \epsilon^{1/4} \leq 1$ \\
\\ [-.85em]
(13'') & $\frac{C^2_1}{\mathsf{p}_m}\lambda^{(d-1)/2}_K \epsilon \leq \frac{1}{3}$ \\
\\ [-.85em]
(14'') & $144\tilde{C}_5  C_4 \lambda^{d+1}_{K} (C'_6+2\lambda_K) \frac{(1+\lambda_K^{d/2+5})}{\mathsf\Gamma_K}(C_1 \lambda_K^{\frac{d-1}{4}}+1) \epsilon \leq \frac{1}{4}$\\[+.6em]
\hline 
\end{tabular}
\end{table}
{ Note that (1'') is stronger than \eqref{condition 1}, and (2'')-(5'') come from \eqref{condition 2}, \eqref{condition 3}, \eqref{condition 4}, and \eqref{condition 5}}. If (6'')-(10'') are satisfied, then we have
$$ |\lambda_{i,\epsilon}- \lambda_i| \leq  \epsilon^{\frac{3}{2}},\quad 
\|a_i\phi_{i,\epsilon}-\phi_i\|_2 \leq   \epsilon^{\frac{3}{2}},\quad 
\|a_i\phi_{i,\epsilon}-\phi_{i}\|_{\infty} \leq \epsilon.$$
Moreover, (11''), (12'') and (13'') will be used in the proof of propositions later. Note that (2''), (4''), (5''), (6''), (7'') and (11'') are satisfied whenever we have
\begin{align}
\epsilon^{1/2} \leq \frac{\tilde{C}_6 \min(\mathsf\Gamma_K,1)}{\tilde{C}_7+\lambda_K^{d/2+5}+\lambda_K^{2}+\lambda_K^{(d+10)/6}+\lambda_K^{(d-1)/4}}. \label{relation between lambda K and epsilon 1}
\end{align}
Observe that if $\lambda_K<1$ , then $\lambda_K^{d/2+5}+\lambda_K^{2}+\lambda_K^{(d+10)/6}+\lambda_K^{(d-1)/4}<4$. If $\lambda_K \geq 1$, then $\lambda_K^{d/2+5}+\lambda_K^{2}+\lambda_K^{(d+10)/6}+\lambda_K^{(d-1)/4} \leq 4 \lambda_K^{d/2+5}$. Hence, \eqref{relation between lambda K and epsilon 1} is true whenever 
\begin{align}
\epsilon^{1/2} \leq \frac{\tilde{C}_6}{4} \frac{\min(\mathsf\Gamma_K,1)}{\tilde{C}_7+\lambda_K^{d/2+5}}. \label{relation between lambda K and epsilon 2}
\end{align}
where $\tilde{C}_7>1$. Based on the requirement \eqref{relation between lambda K and epsilon 1}, (1''), (8''), (9''),  (12'') and (13'') are satisfied when we further require $\epsilon^{2} \leq \frac{\tilde{C}_8}{\lambda_K}$, $\epsilon^{1/2}\leq \frac{\tilde{C}_8}{1+\lambda_K^d}$, $\epsilon^{1/2} \leq \frac{\tilde{C}_8}{\lambda_K^{d-1}}$, $\epsilon \leq \frac{\tilde{C}_8}{\lambda_K^{d-1}}$ and $\epsilon^{1/2} \leq \frac{\tilde{C}_8}{\lambda_K^{(d-1)/4}}$ respectively. Hence it is sufficient to require
\begin{align}
\epsilon^{1/2} \leq \frac{\tilde{C}_8}{1+\lambda_K^{d}}.  \label{relation between lambda K and epsilon 3}
\end{align} 
{Based on \eqref{relation between lambda K and epsilon 2} and \eqref{relation between lambda K and epsilon 3}, (10'') and (14'') hold when
\begin{align}
\epsilon^{1/2} \leq \frac{\tilde{C}_9}{\lambda^{d+2}_K (1+C_1\lambda_K^{\frac{d-1}{4}})}.  \label{relation between lambda K and epsilon 4}
\end{align} 
If we combine \eqref{relation between lambda K and epsilon 2} and \eqref{relation between lambda K and epsilon 4} , we {set} the relation between $\epsilon$ and $\lambda_K$ to be
\begin{align} 
\epsilon \leq \mathcal{K}_1 \min \Bigg(\bigg(\frac{\min(\mathsf\Gamma_K,1)}{\mathcal{K}_2+\lambda_K^{d/2+5}}\bigg)^2, \frac{1}{(\mathcal{K}_3+\lambda_K^{(5d+7)/4})^2}\Bigg),
\end{align}
where $\mathcal{K}_2, \mathcal{K}_3>1$. Note that $\mathcal{K}_1$, $\mathcal{K}_2$, and $\mathcal{K}_3$ depend on all the constants in Table \ref{Table:Relations}, hence they depend on $d$, $\mathsf p_m$, the $C^2$ norm of $\mathsf p$ and the volume,  the injectivity radius,  the  curvature and the second fundamental form of the manifold. 
}
\end{proof}

\subsection{Relation between the operator $T_{n,\epsilon}$ and the operator $T_{\epsilon}$.}  
We now show that the sequence of operators $T_{n,\epsilon}$ collectively compactly converge to $T_{\epsilon}$. We need to study the following spaces of functions.

{
\begin{definition}\label{4 spaces}
Take $u\in C(M)$. Define
\begin{align}
& \mathcal K_{\epsilon} =\{k_\epsilon(x, \cdot)|\,x \in M \}\,, \nonumber \\
& u \mathcal{Q}_{\epsilon}=\{u(\cdot) Q_{\epsilon}(x,\cdot) |\,x \in M \} \,, \nonumber  \\
&\mathcal{Q}_{\epsilon} \mathcal{Q}_{\epsilon}=\{Q_{\epsilon}(x,\cdot)Q_{\epsilon}(y,\cdot) |\,x , y \in M \}\,. \nonumber 
\end{align}
\end{definition}

In general, we cannot show $\sup_{f \in C(M)}|P_n f -Pf |$ goes to $0$. However, the following proposition
shows that it is true for functions in the above spaces. The similar statements can be found in several places. However, for the sake of self-containedness, we provide a proof in the Appendix.

\begin{proposition}(Gilvenko-Cantelli Class)\label{GC class}
Take $u\in C(M)$ with $\|u\|_\infty \leq 1$. Fix {$\epsilon>0$} small enough. 
Then there is a constant {$C_{gc}>0$} depending on $d$, the diameter of $M$, $\mathsf p_m$, and the $C^0$ norm of $\mathsf p$, such that then with probability greater than $1-n^{-2}$, we have 
\begin{align}
& \sup_{f \in \mathcal K_{\epsilon}}|P_nf-Pf| \leq  \frac{C_{gc}}{\sqrt{n}}\left( \sqrt{-\log\epsilon}+\sqrt{\log n}\right)\,, \\
& \sup_{f \in u\mathcal{Q}_{\epsilon}}|P_nf-Pf| \leq  \frac{C_{gc}}{\sqrt{n} \epsilon^{2d}}\left( \sqrt{-\log\epsilon}+\sqrt{\log n}\right)\,, \\
& \sup_{f \in \mathcal{Q}_{\epsilon} \mathcal{Q}_{\epsilon}}|P_nf-Pf| \leq  \frac{C_{gc}}{\sqrt{n}\epsilon^{4d}}\left( \sqrt{-\log\epsilon}+\sqrt{\log n}\right)\,.
\end{align}
\end{proposition}
}

In the next few Lemmas, we provide bounds so that we can apply Theorem \ref{atkinson1967}. { Recall that the operators $T_{\epsilon}$, $T_{n,\epsilon}$ and $\hat{T}_{n,\epsilon}$ {in Definition \ref{intro of operators} are} operators from $(C(M),\|\cdot\|_\infty)$ to $(C(M),\|\cdot\|_\infty)$. We provide the bounds of them in the next lemma.}

\begin{lemma}\label{bounds on operators}
Suppose $f\in C(M)$ and $\|f\|_\infty \leq 1$.
{For $\epsilon>0$ small enough,} there are constants $C_5>0$, $C_6>0$ and $C_7>0$ depending  on $\mathsf p_m$ and the $C^0$ norm of $\mathsf p$ so that:
\begin{enumerate} 
\item[(a)] For any $x \in M$, $C_5 \epsilon^d \leq d_{\epsilon}(x) \leq C_6 \epsilon^d$.
\item[(b)]  For any $x \in M$, $|P Q_{\epsilon} f(x)| \leq  \frac{C_6}{C_5^2 \epsilon^{d}}$.
\item[(c)] For any $x \in M$, $P Q_{\epsilon}(x) \geq \frac{C_5}{C_6^2 \epsilon^d}$.
\item[(d)] $\|T_{\epsilon}\| \leq \frac{C_6^3}{C_5^3}$.
\end{enumerate}
{ Fix {$\epsilon>0$} small enough. Suppose $n$ is large enough so that $\frac{1}{\sqrt{n}\epsilon^d} (\sqrt{-\log\epsilon}+\sqrt{\log n})< \mathcal{C}_1$, where $\mathcal{C}_1$ depends on $d$, the diameter of $M$, $\mathsf p_m$, and the $C^0$ norm of $\mathsf p$.} Then, with probability greater than $1-n^{-2}$, the following bounds hold. 
\begin{enumerate} 
\item[(a')] For any $x \in M$, $C_5 \epsilon^d \leq d_{n,\epsilon}(x) \leq C_6 \epsilon^d$.
\item[(b')] For any $x \in M$, $|P_n Q_{\epsilon} f(x)| \leq  \frac{C_6}{C_5^2 \epsilon^{d}}$ and $|P_n Q_{n,\epsilon} f(x)| \leq  \frac{C_6}{C_5^2 \epsilon^{d}}$.
\item[(c')] For any $x \in M$, $P_n Q_{n,\epsilon}(x) \geq \frac{C_5}{C_6^2 \epsilon^d}$.
\item[(d')] $\|T_{n,\epsilon}\| \leq \frac{C_6^3}{C_5^3}$ and $\|\hat{T}_{n,\epsilon}\| \leq \frac{C_6^3}{C_5^3}$.
\end{enumerate}
Moreover, there is a constant {$C_7>0$} depending  on  $d$, the diameter of $M$, $\mathsf p_m$, and the $C^0$ norm of $\mathsf p$, such that with probability greater than $1-n^{-2}$, 
\begin{enumerate}
\item[(e')] For any $x,y \in M$, $|Q_{n,\epsilon}(x,y)-Q_{\epsilon}(x,y)|\leq \frac{C_7}{\sqrt{n}\epsilon^{3d}}(\sqrt{-\log\epsilon}+\sqrt{\log n})$.
\end{enumerate}
\end{lemma}

\begin{proof}
The proof of (a) follows from Lemma \ref{properties of T epsilon} (a), i.e. there are constants {$0<\tilde{C}_1\leq\tilde{C}_2$} such that
\begin{align}
\tilde{C}_1 \epsilon^d \leq d_{\epsilon}(x) \leq \tilde{C}_2\epsilon^d, \nonumber
\end{align}
where $\tilde{C}_1$ and $\tilde{C}_2$ depend on $\mathsf p_m$ and the $C^0$ norm of $\mathsf p$.
To prove (a'), by Proposition \ref{GC class}, we have 
\begin{align}
|d_{n,\epsilon}(x) -d_{\epsilon}(x) |\leq \frac{C_{gc}}{\sqrt{n}}\left( \sqrt{-\log\epsilon}+\sqrt{\log n}\right)\,. \nonumber
\end{align}
Hence, as long as $\frac{C_{gc}}{\sqrt{n}} (\frac{1}{\epsilon}+\sqrt{\log n})< \frac{{\tilde{C}_1}}{2}\epsilon^{d}$
, which is equivalent to  
\begin{align}
\frac{1}{\sqrt{n}\epsilon^d} \left( \sqrt{-\log\epsilon}+\sqrt{\log n}\right) < \frac{{\tilde{C}_1}}{2C_{gc}}=:\mathcal{C}_1,\label{Definition mathcal C_1}
\end{align}
{we have} 
\begin{align}
\frac{\tilde{C}_1}{2} \epsilon^d \leq d_{n,\epsilon}(x) \leq 2\tilde{C}_2 \epsilon^d. \nonumber
\end{align}
Note that $\sqrt\frac{\log n}{n}$ goes to $0$ as $n$ goes to infinity. Hence, the term $\frac{1}{\sqrt{n}\epsilon^d} ( \sqrt{-\log\epsilon}+\sqrt{\log n})$ will be smaller than the constant $\mathcal{C}_1$, when $n$ is large enough. The conclusion follows {by setting $C_5:=\frac{\tilde{C}_1}{2}$ and $C_6:=2\tilde{C}_2$}. 

For (b), by a direct expansion and the bound of $d_\epsilon$ in (a), we have
\begin{align}
|PQ_{\epsilon} f(x)| = \Big|\int_{M}\frac{k_{\epsilon}(x,y) f(y)}{d_{\epsilon}(x)d_{\epsilon}(y)} dP(y)\Big|\leq \frac{ \|f\|_\infty \sup_{x \in M} d_{\epsilon}(x)}{(\inf_{x \in M}d_{\epsilon}(x) )^2} \leq \frac{C_6}{C_5^2 \epsilon^{d}}. \nonumber
\end{align}
For (b'), with probability greater than $1-n^{-2}$, the following bounds hold:
\begin{align}
|P_n Q_{n,\epsilon} f(x)| &\,= \Big|\frac{1}{n} \sum_{i=1}^n \frac{k_{\epsilon}(x,x_i) f(x_i)}{d_{n,\epsilon}(x)d_{n,\epsilon}(x_i)} \Big|\leq \frac{ \|f\|_\infty \sup_{x \in M} d_{n,\epsilon}(x)}{(\inf_{x \in M}d_{n,\epsilon}(x) )^2} \leq \frac{C_6}{C_5^2 \epsilon^{d}} \nonumber\\
|P_n Q_{\epsilon} f(x)| &\,= \Big|\frac{1}{n} \sum_{i=1}^n \frac{k_{\epsilon}(x,x_i) f(x_i)}{d_{\epsilon}(x)d_{\epsilon}(x_i)} \Big|\leq \frac{ \|f\|_\infty \sup_{x \in M} d_{n,\epsilon}(x)}{(\inf_{x \in M}d_{\epsilon}(x) )^2} \leq \frac{C_6}{C_5^2 \epsilon^{d}}\,, \nonumber
\end{align} 
where we apply the bound for $d_{n,\epsilon}$ in (a').

For (c), again, by (a),
\begin{align}
Q_{\epsilon}(x,y) \geq \frac{k_{\epsilon}(x,y)}{(\sup_{x \in M}d_{\epsilon}(x))^2} \geq \frac{k_{\epsilon}(x,y)}{C^2_2\epsilon^{2d}}\,, \nonumber
\end{align}
and hence $PQ_{\epsilon}(x) \geq \frac{d_{\epsilon}(x)}{C^2_2\epsilon^{2d}} \geq \frac{C_5}{C_6^2 \epsilon^d}$.
(c') follows from (a') and a direct expansion: 
\begin{align}
Q_{n,\epsilon}(x,y) \geq \frac{k_{\epsilon}(x,y)}{(\sup_{x \in M}d_{n,\epsilon}(x))^2} \geq \frac{k_{\epsilon}(x,y)}{C^2_2\epsilon^{2d}}. \nonumber
\end{align}
Hence $P_nQ_{n,\epsilon}(x) \geq \frac{d_{n,\epsilon}(x)}{C^2_2\epsilon^{2d}} \geq \frac{C_5}{C_6^2 \epsilon^d}$.
Similarly, (d) follows from (b) and (c) and (d') follows from (b') and (c').

For (e'), we have 
\begin{align}
|Q_{n,\epsilon}(x,y)-Q_{\epsilon}(x,y)| & \leq  |k_{\epsilon}(x,y)|\Big|\frac{1}{d_{n,\epsilon}(x)d_{n,\epsilon}(y)}-\frac{1}{d_{\epsilon}(x)d_{\epsilon}(y)}\Big| \label{Proof Bound of Qeps and Qneps at x} \\
& \leq \Big|\frac{d_{\epsilon}(x) d_{\epsilon}(y)-d_{n,\epsilon}(x) d_{n,\epsilon}(y)}{d_{n,\epsilon}(x)d_{n,\epsilon}(y)d_{\epsilon}(x)d_{\epsilon}(y)}\Big| \nonumber \\
& \leq \frac{|d_{\epsilon}(x)-d_{n,\epsilon}(x)|}{|d_{n,\epsilon}(x)d_{n,\epsilon}(y)d_{\epsilon}(x)|}+\frac{|d_{\epsilon}(y)-d_{n,\epsilon}(y)|}{|d_{n,\epsilon}(y)d_{n,\epsilon}(y)d_{\epsilon}(x)|} \nonumber \\
& \leq \frac{C_7}{\sqrt{n}\epsilon^{3d}}\left( \sqrt{-\log\epsilon}+\sqrt{\log n}\right)\nonumber
\end{align}
{for a constant $C_7>0$,} where the second bound holds since $|k_{\epsilon}(x,y)|\leq 1$ and we use (a) and Proposition \ref{GC class} in the last step.
\end{proof}

Next, we need to bound the term $\|T_{n,\epsilon}u-T_{\epsilon}u\|_{\infty}$, {which is needed when we apply} Theorem \ref{atkinson1967}.

\begin{lemma}\label{fix u , T n espilon-T epsilon}
Fix $u\in C(M)$ with $\|u\|_{\infty} \leq 1$. { Fix any $\epsilon$ small enough. Suppose $n$ is large enough so that $\frac{1}{\sqrt{n}\epsilon^d} ( \sqrt{-\log\epsilon}+\sqrt{\log n})< \mathcal{C}_1${, where $\mathcal C_1$ is defined in \eqref{Definition mathcal C_1}}. There is a constant {$C_8>0$} depending  on  $d$, the diameter of $M$,  $\mathsf p_m$ and the $C^0$ norm of $\mathsf p$ 
such that with probability greater than $1-n^{-2}$, we have
\begin{align}
|T_{n,\epsilon}u(x)-T_{\epsilon}u(x)| \leq \frac{C_8}{\sqrt{n}\epsilon^{2d}}\left( \sqrt{-\log\epsilon}+\sqrt{\log n}\right) \nonumber
\end{align}
for any $x \in M$.}
\end{lemma}

\begin{proof}
By a direct expansion and triangular inequality,  
\begin{align}
|T_{n,\epsilon}u(x)-T_{\epsilon}u(x)| = &\, \bigg|\frac{P_{n}Q_{n,\epsilon}u (x)}{P_{n}Q_{n,\epsilon}(x)}-\frac{PQ_{\epsilon}u (x)}{PQ_{\epsilon}(x)}\bigg| \nonumber \\ 
\leq &\, \bigg|\frac{P_{n}Q_{n,\epsilon}u (x)}{P_{n}Q_{n,\epsilon}(x)}-\frac{P_nQ_{n,\epsilon}u (x)}{PQ_{\epsilon}(x)}\bigg| +\bigg|\frac{P_{n}Q_{n,\epsilon}u (x)}{PQ_{\epsilon}(x)}-\frac{PQ_{\epsilon}u (x)}{PQ_{\epsilon}(x)}\bigg|\,,\nonumber 
\end{align}
which is further bounded by
\begin{align}
&\, \bigg|\frac{P_{n}Q_{n,\epsilon}u (x)}{P_{n}Q_{n,\epsilon}(x)PQ_{\epsilon}(x)}\bigg| |P_{n}Q_{n,\epsilon}(x)-PQ_{\epsilon}(x)|+\bigg|\frac{P_{n}Q_{n,\epsilon}u (x)-P_nQ_{\epsilon}u (x)}{PQ_{\epsilon}(x)}\bigg| \nonumber \\
& \qquad\qquad+\bigg|\frac{P_{n}Q_{\epsilon}u (x)-PQ_{\epsilon}u (x)}{PQ_{\epsilon}(x)}\bigg| \nonumber \\
\leq &\,  \bigg(\frac{|P_{n}Q_{n,\epsilon}u (x)|}{P_{n}Q_{n,\epsilon}(x)PQ_{\epsilon}(x)}+\frac{1}{PQ_{\epsilon}(x)}\bigg) |P_{n}Q_{n,\epsilon}u (x)-P_nQ_{\epsilon}u (x)| \nonumber \\
& \qquad\qquad+ \bigg(\frac{|P_{n}Q_{n,\epsilon}u (x)|}{P_{n}Q_{n,\epsilon}(x)PQ_{\epsilon}(x)}+\frac{1}{PQ_{\epsilon}(x)}\bigg) |P_{n}Q_{\epsilon}u (x)-PQ_{\epsilon}u (x)| \nonumber\,.
\end{align}
By Lemma \ref{bounds on operators},
\begin{equation}
\frac{|P_{n}Q_{n,\epsilon}u (x)|}{P_{n}Q_{n,\epsilon}(x)PQ_{\epsilon}(x)}+\frac{1}{PQ_{\epsilon}(x)}\leq \Big(\frac{C_6^5}{C_5^4}+\frac{C_6}{C_5}\Big)\epsilon^d\,, \nonumber
\end{equation}
where $C_5$ and $C_6$ are the constants in Lemma \ref{bounds on operators}.
{By (e')} in Lemma  \ref{bounds on operators},
\begin{equation}
|P_{n}Q_{n,\epsilon}u (x)-P_nQ_{\epsilon}u (x)|\leq \frac{C_7}{\sqrt{n}\epsilon^{3d}}\left( \sqrt{-\log\epsilon}+\sqrt{\log n}\right)\,. \nonumber
\end{equation}
By Proposition \ref{GC class},
\begin{equation}
|P_{n}Q_{\epsilon}u (x)-PQ_{\epsilon}u (x)| \leq \frac{C_{gc}}{\sqrt{n} \epsilon^{2d}}\left( \sqrt{-\log\epsilon}+\sqrt{\log n}\right)\,. \nonumber
\end{equation}
Hence, by putting the above together, $|T_{n,\epsilon}u(x)-T_{\epsilon}u(x)| \leq  \frac{C_8}{\sqrt{n}\epsilon^{2d+1}},$ for some constant {$C_8>0$ as claimed}.
\end{proof}

{Next, we need to} bound the term $\|(T_{\epsilon}-T_{n,\epsilon})T_{n,\epsilon}\|$ {so that we can apply} Theorem \ref{atkinson1967}. Note that { for any fixed $\epsilon$, $\|T_{\epsilon}-T_{n,\epsilon}\|$ does not go to $0$ in general as $n$ goes to infinity}, neither does $\|T_{\epsilon}-\hat{T}_{n,\epsilon}\|$. Otherwise, the spectral convergence is straightforward and there is no need to introduce the concept of collective compact convergence. {Indeed, although we can establish the inequality shown in Lemma \ref{fix u , T n espilon-T epsilon}, the inequality only holds for a {\em single} function. To achieve $\|T_{\epsilon}-T_{n,\epsilon}\|\to 0$, we need to control all $u\in C^0$ satisfying $\|u\|_\infty\leq 1$. However, this cannot be achieved directly with the uniform bound. This bound is the critical step in controlling the convergence rate when we have a finite sampling set.} In other words, we cannot show that $\|(T_{\epsilon}-T_{n,\epsilon})T_{n,\epsilon}\|$ goes to $0$ by showing that $\|(T_{\epsilon}-T_{n,\epsilon})\|$ goes to $0$. To bypass this difficulty, we need the following intermediate lemma. 

\begin{lemma}\label{Proof Tn and Tneps relationship intermittent}
{ Fix {$\epsilon>0$} small enough. Suppose $n$ is {sufficiently} large so that $\frac{1}{\sqrt{n}\epsilon^d} ( \sqrt{-\log\epsilon}+\sqrt{\log n})< \mathcal{C}_1${, where $\mathcal C_1$ is defined in \eqref{Definition mathcal C_1}}. There are constants $C_9$ and $C_{10}$ depending  on  $d$, the diameter of $M$,  $\mathsf p_m$ and the $C^0$ norm of $\mathsf p$ such that with probability greater than $1-n^{-2}$, 
\begin{equation}
\|\hat{T}_{n,\epsilon}-T_{n,\epsilon}\|\leq \frac{C_9}{\sqrt{n}\epsilon^{2d}} \left( \sqrt{-\log\epsilon}+\sqrt{\log n}\right),\quad 
\|(T_{\epsilon}-\hat{T}_{n,\epsilon})\hat{T}_{n,\epsilon}\| \leq \frac{C_{10}}{\sqrt{n}\epsilon^{2d}}\left( \sqrt{-\log\epsilon}+\sqrt{\log n}\right)\,. \nonumber
\end{equation}}
\end{lemma}
\begin{proof}
For any $f\in C(M)$ such that $\|f\|_{\infty} \leq 1$,
\begin{align}
&\sup_{x\in M}|\hat{T}_{n,\epsilon}f(x)-T_{n,\epsilon}f(x)|\leq \sup_{x \in M}\left|\frac{P_{n}Q_{n,\epsilon}f (x)}{P_{n}Q_{n,\epsilon}(x)}-\frac{P_nQ_{\epsilon}f (x)}{PQ_{\epsilon}(x)}\right| \nonumber \\
\leq\,& \sup_{x \in M}\Big(\frac{|P_{n}Q_{n,\epsilon}f (x)|}{P_{n}Q_{n,\epsilon}(x)PQ_{\epsilon}(x)}|P_{n}Q_{n,\epsilon}(x)-PQ_{\epsilon}(x)|\Big) +\sup_{x \in M}\Big(\frac{1}{PQ_{\epsilon}(x)}|P_{n}Q_{n,\epsilon}f (x)-P_nQ_{\epsilon}f (x)|\Big)\,. \nonumber
\end{align}
{We control the last step term by term. First,}
\begin{align}
|P_{n}Q_{n,\epsilon}(x)-PQ_{\epsilon}(x)|\leq & |P_{n}Q_{n,\epsilon}(x)-P_nQ_{\epsilon}(x)|+|P_nQ_{\epsilon}(x)-PQ_{\epsilon}(x)| \nonumber \\
\leq & \left(\frac{C_7}{\sqrt{n}\epsilon^{3d}}+\frac{C_{gc}}{\sqrt{n}\epsilon^{2d}}\right)( \sqrt{-\log\epsilon}+\sqrt{\log n})\nonumber
\end{align}
by Proposition \ref{GC class} and (e') in Lemma \ref{bounds on operators}. Moreover,
\begin{align}
|P_{n}Q_{n,\epsilon}f (x)-P_nQ_{\epsilon}f (x)|\leq  \frac{C_7}{\sqrt{n}\epsilon^{3d}} \left( \sqrt{-\log\epsilon}+\sqrt{\log n}\right) \nonumber 
\end{align}
by (e') in Lemma \ref{bounds on operators}. In summary, 
\begin{align}
\sup_{x\in M}|\hat{T}_{n,\epsilon}f(x)-T_{n,\epsilon}f(x)|\leq \frac{C_9}{\sqrt{n}\epsilon^{2d}}\left( \sqrt{-\log\epsilon}+\sqrt{\log n}\right)\nonumber 
\end{align}
for some constant {$C_9>0$}, which is our first claim.

For the second claim,  {similarly, for any $f\in C(M)$ such that $\|f\|_{\infty} \leq 1$,}
\begin{align}
& \sup_{x \in M} |(T_{\epsilon}-\hat{T}_{n,\epsilon})\hat{T}_{n,\epsilon}f(x)|   \nonumber \\
=&\, \sup_{x \in M} \Big|\frac{1}{PQ_{\epsilon}(x)}\Big(\int_{M}Q_{\epsilon}(x,y)\frac{\int_{M}Q_{\epsilon}(y,z)f(z)d\mathsf P_n(z)}{PQ_{\epsilon}(y)}d\mathsf P(y)\nonumber\\
&\qquad\qquad\qquad-\int_{M}Q_{\epsilon}(x,y)\frac{\int_{M}Q_{\epsilon}(y,z)f(z)d\mathsf P_n(z)}{PQ_{\epsilon}(y)}d\mathsf P_n(y)\Big)\Big| \nonumber 
\end{align}
is bounded by
\begin{align}
&\, \frac{1}{[\inf_{x \in M} PQ_{\epsilon}(x)]^2} \sup_{x \in M}\Big|\int_{M}Q_{\epsilon}(x,y)\int_{M}Q_{\epsilon}(y,z)f(z)d\mathsf P_n(z)d\mathsf P(y)\nonumber\\
&\qquad\qquad\qquad-\int_{M}Q_{\epsilon}(x,y)\int_{M}Q_{\epsilon}(y,z)f(z)d\mathsf P_n(z)d\mathsf P_n(y)\Big| \nonumber \\
\leq &\, \frac{1}{[\inf_{x \in M} PQ_{\epsilon}(x)]^2} \sup_{x \in M}\Big|\int_{M}Q_{\epsilon}(x,y)\int_{M}Q_{\epsilon}(y,z)f(z)d\mathsf P_n(z)d\mathsf P(y)\nonumber\\
&\qquad\qquad\qquad-\int_{M}Q_{\epsilon}(x,y)\int_{M}Q_{\epsilon}(y,z)f(z)d\mathsf P_n(z)d\mathsf P_n(y)\Big|\nonumber \\
=&\, \frac{1}{[\inf_{x \in M} PQ_{\epsilon}(x)]^2} \sup_{x \in M}\Big|\int_{M}\Big[\int_{M}Q_{\epsilon}(x,y)Q_{\epsilon}(y,z)d\mathsf P(y)\nonumber\\
&\qquad\qquad\qquad-\int_{M}Q_{\epsilon}(x,y)Q_{\epsilon}(y,z)d\mathsf P_n(y)\Big]f(z)d\mathsf P_n(z)\Big| \,,\nonumber 
\end{align}
where we use Fubini's Theorem in the {last equality}. As a result, by Proposition \ref{GC class} and Lemma \ref{bounds on operators}, 
\begin{align}
\sup_{x \in M} |(T_{\epsilon}-\hat{T}_{n,\epsilon})\hat{T}_{n,\epsilon}f(x)|\leq \frac{C_6^4 \epsilon^{2d}}{C_5^2}  \frac{C_{gc}}{\sqrt{n}\epsilon^{4d}}\left( \sqrt{-\log\epsilon}+\sqrt{\log n}\right) = \frac{C_{10}}{\sqrt{n}\epsilon^{2d}}\left( \sqrt{-\log\epsilon}+\sqrt{\log n}\right),\nonumber
\end{align}
{by setting $C_{10}:=\frac{C_6^4 C_{gc}}{C_5^2}>0$}, and hence the proof follows. 
\end{proof}
\begin{lemma}\label{(T espilon- Tn epsilon)Tn epsilon}
{Fix  {$\epsilon>0$} small enough. Suppose $n$ is {sufficiently} large so that $\frac{1}{\sqrt{n}\epsilon^d} ( \sqrt{-\log\epsilon}+\sqrt{\log n})< \mathcal{C}_1${, where $\mathcal C_1$ is defined in \eqref{Definition mathcal C_1}}. There is a constant {$C_{11}>0$} depending  on  $d$, the diameter of $M$,  $\mathsf p_m$ and the $C^0$ norm of $\mathsf p$, such that with probability greater than $1-n^{-2}$, we have
\begin{align}
\|(T_{\epsilon}-T_{n,\epsilon})T_{n,\epsilon}\|\leq  \frac{C_{11}}{\sqrt{n}\epsilon^{2d}} \left( \sqrt{-\log\epsilon}+\sqrt{\log n}\right). \nonumber 
\end{align}}
\end{lemma}
\begin{proof}
By a direct expansion and Lemma \ref{Proof Tn and Tneps relationship intermittent}, we have
\begin{align}
\|(T_{\epsilon}-T_{n,\epsilon})T_{n,\epsilon}\|\leq &\, \|T_{\epsilon}T_{n,\epsilon}-T_{\epsilon}\hat{T}_{n,\epsilon}\|+\|T_{\epsilon}\hat{T}_{n,\epsilon}-\hat{T}_{n,\epsilon}\hat{T}_{n,\epsilon}\| \nonumber \\
&+\|\hat{T}_{n,\epsilon}\hat{T}_{n,\epsilon}-\hat{T}_{n,\epsilon}T_{n,\epsilon}\|+\|\hat{T}_{n,\epsilon}T_{n,\epsilon}-T_{n,\epsilon}T_{n,\epsilon}\| \nonumber \\
\leq &\, (\|T_{\epsilon}\|+\|\hat{T}_{n,\epsilon}\|+\|T_{n,\epsilon}\|)\|\hat{T}_{n,\epsilon}-T_{n,\epsilon}\|  +\|(T_{\epsilon}-\hat{T}_{n,\epsilon})\hat{T}_{n,\epsilon}\| \nonumber \\
\leq &\, \frac{3C_6^3}{C_5^3} \frac{C_9}{\sqrt{n}\epsilon^{2d}} \left( \sqrt{-\log\epsilon}+\sqrt{\log n}\right)+\frac{C_{10}}{\sqrt{n}\epsilon^{2d}} \left( \sqrt{-\log\epsilon}+\sqrt{\log n}\right) \,.\nonumber 
\end{align}
The conclusion follows {by setting $C_{11}:=\frac{3C_6^3C_9}{C_5^3}+C_{10}$}.
\end{proof}

In the next Lemma, we show that $T_{\epsilon}$ is a compact operator.

\begin{lemma}  \label{uniform convergence of operator Q}
$T_{\epsilon}$ is a compact operator. In particular, for any function $f \in C(M)$ such that $\|f\|_\infty\leq 1$, when $\|\iota(x)-\iota(x')\|_{\mathbb{R}^D}$ is small enough, for $x,x'\in M$, we have
\begin{equation}
|T_{\epsilon}f(x)-T_{\epsilon}f(x')| \leq C_{12}  \epsilon^{-d} \|\iota(x)-\iota(x')\|_{\mathbb{R}^D}\,, \nonumber 
\end{equation}
where {$C_{12}>0$} only depends on the curvature of $M$, $\mathsf p_m$ and the $C^1$ norm of $\mathsf p$.
\end{lemma}
\begin{proof}
First, by (d) in Lemma \ref{bounds on operators}, we have $\|T_{\epsilon}\| \leq \frac{C_6^3}{C_5^3}$.
Next, since {$Q_{\epsilon}(\cdot,y)$} is a smooth function on $M$ and $M$ is compact, for $x, x' \in M$ with $\|\iota(x)-\iota(x')\|_{\mathbb{R}^D}$ small enough, we have  
\begin{align}
|Q_{\epsilon}(x,y)-Q_{\epsilon}(x',y)| \leq \tilde{C}_1\epsilon^{-2d} d(x,x') \leq {2\tilde{C}_1} \epsilon^{-2d} \|\iota(x)-\iota(x')\|_{\mathbb{R}^D}\,,\label{Proof Bound of Qeps from x to x'} 
\end{align}
where {$\tilde{C}_1>0$} depends on the curvature of $M$, $\mathsf p_m$ and the $C^1$ norm of $\mathsf p$.
Here $d(x,x')$ is the geodesic distance between $x$ and $x'$, and we use the fact that $d(x,x')$ can be bounded by $2\|\iota(x)-\iota(x')\|_{\mathbb{R}^D}$ whenever $\|\iota(x)-\iota(x')\|_{\mathbb{R}^D}$ is small enough. Thus, 
\begin{align}
\Big|\int_M Q_{\epsilon}(x,y) d\mathsf P(y)-\int_M Q_{\epsilon}(x',y) d\mathsf P(y)\Big| \leq &\, \int_M |Q_{\epsilon}(x,y)-Q(x',y)|d\mathsf P(y) \nonumber  \\
\leq &\, {2\tilde{C}_1} \epsilon^{-2d}\|\iota(x)-\iota(x')\|_{\mathbb{R}^D}\,,\nonumber
\end{align}
where we use the fact that $\int_M d\mathsf P(y)=1$.
Similarly, for any function $f \in C(M)$ such that $\|f\|_\infty\leq 1$,
\begin{equation}
\Big|\int_M Q_{\epsilon}(x,y) f (y) d\mathsf P(y)-\int_M Q_{\epsilon}(x',y) f (y) d\mathsf P(y)\Big| \leq {2\tilde{C}_1}\epsilon^{-2d} \|\iota(x)-\iota(x')\|_{\mathbb{R}^D}. \nonumber 
\end{equation}
As a result, 
\begin{align}
|T_{\epsilon}f(x)-T_{\epsilon}f(x')| 
{=}  \Big|\frac{\int_M Q_{\epsilon}(x,y)f(y) d\mathsf P(y)}{\int_M Q_{\epsilon}(x,y) d\mathsf P(y)}-\frac{\int_M Q_{\epsilon}(x',y)f(y) d\mathsf P(y)}{\int_M Q_{\epsilon}(x',y) d\mathsf P(y)} \Big| \nonumber 
\end{align}
is bounded by
\begin{align}
& \frac{|\int_M Q_{\epsilon}(x,y) f(y) d\mathsf P(y)-\int_M Q_{\epsilon}(x',y) f(y) d\mathsf P(y)|}{|\int_M Q_{\epsilon}(x,y) dP(y)|} \nonumber \\
& \qquad +\frac{|\int_M Q_{\epsilon}(x',y) f(y) d\mathsf P(y)||\int_M Q_{\epsilon}(x,y) d\mathsf P(y)-\int_M Q_{\epsilon}(x',y) d\mathsf P(y)|}{|\int_M Q_{\epsilon}(x,y) d\mathsf P(y)||\int_M Q_{\epsilon}(x',y) d\mathsf P(y)|} \nonumber \\
\leq &\, \frac{|\int_M Q_{\epsilon}(x,y) f(y) d\mathsf P(y)-\int_M Q_{\epsilon}(x',y) f(y) d\mathsf P(y)|}{\inf_{x \in M}|\int_M Q_{\epsilon}(x,y) d\mathsf P(y)|} \nonumber \\
& \qquad+\sup_{x \in M}\Big|\int_M Q_{\epsilon}(x,y) f(y) d\mathsf P(y)\Big| \frac{|\int_M Q_{\epsilon}(x,y) d\mathsf P(y)-\int_M Q_{\epsilon}(x',y) d\mathsf P(y)|}{(\inf_{x \in M}|\int_M Q_{\epsilon}(x,y) d\mathsf P(y)|)^2} \nonumber \\
\leq &\, C_{12} \epsilon^{-d} \|\iota(x)-\iota(x')\|_{\mathbb{R}^D}\,, \nonumber 
\end{align}
for a constant $C_{12}>0$. Hence, for a sequence of functions $\{f_i\}_{i=1}^\infty\subset C(M)$ such that $\|f_i\|_\infty\leq 1$, {$\{T_{\epsilon}f_i\}_{i=1}^\infty$} is equi-continuous. Then, $T_{\epsilon}$ is compact by the Arzel\`{a}-Ascoli theorem.
\end{proof}

In the next Lemma, we show that $T_{n,\epsilon}$ almost surely converges collectively compactly to $T_{\epsilon}$. 

{
\begin{lemma}\label{ collectively compactly covergence}
For {$\epsilon>0$ small enough,} as $n \rightarrow \infty$, $T_{n,\epsilon}$ almost surely converges collectively compactly to $T_{\epsilon}$ on the Banach space $(C(M),\,\|\cdot\|_\infty)$.
\end{lemma}

\begin{proof}
 Denote $B_1\subset C(M)$ to be the unit ball in $(C(M),\|\cdot\|_\infty)$ centered at $0$. By Lemma \ref{fix u , T n espilon-T epsilon},  $T_{n,\epsilon}$ almost surely converges pointwisely to $T_{\epsilon}$. Hence, we only need to show that $\cup_{n=1}^\infty (T_{n,\epsilon} -T_{\epsilon})B_1$ is almost surely relatively compact. Since $T_{\epsilon}$ is a compact operator, it suffices to show that  $\cup_{n=1}^\infty T_{n,\epsilon} B_1 \subset C(M)$ is almost surely relatively compact. 
 
For any sequence $\{g_i\}_{i=1}^\infty \subset \cup_{n=1}^\infty T_{n,\epsilon} B_1$, {without loss of generality, assume} we can choose a subsequence of the form $\{T_{1,\epsilon}f_1, T_{2,\epsilon}f_2, \cdots\}$. {Clearly}, there is a ${N=}N(\epsilon)$ such that if $i>N(\epsilon)$, then $\frac{1}{\sqrt{i}\epsilon^d} (\sqrt{-\log\epsilon}+\sqrt{\log i}) < \mathcal{C}_1${, where $\mathcal C_1$ is defined in \eqref{Definition mathcal C_1}}. Hence, {by (d') in Lemma \ref{bounds on operators},} with probability greater than $1-\frac{1}{i^2}$, $\|T_{i,\epsilon}f_i\|_\infty$ is bounded by a constant. Hence,  by the Borel-Cantelli lemma, %
$\{T_{N,\epsilon}f_N, T_{N+1,\epsilon}f_{N+1}, \cdots\}$ is almost surely 
{bounded}.

Consider any element $T_{j,\epsilon}f_j${, where $j\geq N$,} in the above subsequence.  By {\eqref{Proof Bound of Qeps from x to x'}}, we know that if $\|\iota(x)-\iota(x')\|_{\mathbb{R}^D}$ is small enough,
\begin{align}
|Q_{\epsilon}(x,y)-Q_{\epsilon}(x',y)| \leq {2\tilde{C}_1}\epsilon^{-2d}\|\iota(x)-\iota(x')\|_{\mathbb{R}^D}. \nonumber 
\end{align}
Moreover, {by \eqref{Proof Bound of Qeps and Qneps at x} in the proof of} Lemma \ref{bounds on operators}, with probability greater than $1-j^{-2}$, for any $x,y \in M$,
\begin{align}
|Q_{j,\epsilon}(x,y)-Q_{\epsilon}(x,y)|\leq \frac{C_7}{\sqrt{j}\epsilon^{3d}}\left(\sqrt{-\log\epsilon}+\sqrt{\log j}\right)\,. \nonumber 
\end{align}
Hence, with probability greater than $1-j^{-2}$,
\begin{align}
|Q_{j,\epsilon}(x,y)-Q_{j,\epsilon}(x',y)| \leq {2\tilde{C}_1} \epsilon^{-2d} \|\iota(x)-\iota(x')\|_{\mathbb{R}^D}+\frac{2C_7}{\sqrt{j}\epsilon^{3d}}\left(\sqrt{-\log\epsilon}+\sqrt{\log j}\right)\,. \nonumber 
\end{align}
Since $\|f_j\|_\infty\leq 1$,
\begin{align}
 |T_{j,\epsilon}f_j(x)-T_{j,\epsilon}f_j(x')| 
= \Big|\frac{\int_M Q_{j,\epsilon}(x,y)f_j(y) d\mathsf P_j(y)}{\int_M Q_{j,\epsilon}(x,y) d\mathsf P_j(y)}-\frac{\int_M Q_{j,\epsilon}(x',y)f_j(y) d\mathsf P_j(y)}{\int_M Q_{j,\epsilon}(x',y) d\mathsf P_j(y)} \Big| \nonumber 
\end{align}
is bounded by
\begin{align}
 & \frac{|\int_M Q_{j,\epsilon}(x,y) f_j(y) d\mathsf P_j(y)-\int_M Q_{j,\epsilon}(x',y) f_j(y) d\mathsf P_j(y)|}{|\int_M Q_{\epsilon}(x,y) d\mathsf P_j(y)|} \nonumber \\
& \qquad+\frac{|\int_M Q_{j,\epsilon}(x',y) f_j(y) d\mathsf P_j(y)||\int_M Q_{j,\epsilon}(x,y) d\mathsf P_j(y)-\int_M Q_{j,\epsilon}(x',y) d\mathsf P_j(y)|}{|\int_M Q_{j,\epsilon}(x,y) d\mathsf P_j(y)||\int_M Q_{\epsilon}(x',y) d\mathsf P_j(y)|} \nonumber \\
\leq &\, \Big(\frac{1}{\inf_{x \in M}|\int_M Q_{j,\epsilon}(x,y) d\mathsf P_j(y)|}+\frac{\sup_{x \in M}|\int_M Q_{j,\epsilon}(x,y) f(y) d\mathsf P_j(y)|}{(\inf_{x \in M}|\int_M Q_{j,\epsilon}(x,y) d\mathsf P_j(y)|)^2}\Big)\nonumber\\
&\qquad\times |Q_{j,\epsilon}(x,y)-Q_{j,\epsilon}(x',y)| \nonumber \\
\leq &\, \tilde{C} \epsilon^d \left(\epsilon^{-2d} \|\iota(x)-\iota(x')\|_{\mathbb{R}^D}+\frac{1}{\sqrt{j}\epsilon^{3d}}\left(\sqrt{-\log\epsilon}+\sqrt{\log j}\right)\right)\nonumber 
\end{align}
{for $\tilde{C}>0$,} 
where we use (b') and (c') in Lemma \ref{bounds on operators}.
As a result, for any {$\delta>0$, there exists $N'(\delta,\epsilon)\in \mathbb{N}$ so that when} $j>N'(\delta,\epsilon)$, $\frac{1}{\sqrt{j}\epsilon^{3d}}(\sqrt{-\log\epsilon}+\sqrt{\log j})<\frac{\delta}{2}$. 
{Thus,} if $ \|\iota(x)-\iota(x')\|_{\mathbb{R}^D} \leq \frac{\epsilon^d\delta}{2\tilde{C}}$ and $j>N'(\delta,\epsilon)$, with probability greater than $1-\frac{1}{j^2}$, 
\[
|T_{j,\epsilon}f_j(x)-T_{j,\epsilon}f_j(x')| <\delta\,.
\] 
{As a result}, we can further refine the subsequence of $\{T_{N,\epsilon}f_N, T_{N+1,\epsilon}f_{N+1}, \cdots\}$ into a subsequence that is almost surely { bounded} and equicontinuous {by the Borel-Cantelli Lemma. It} follows from the Arzel\`a-Ascoli theorem that $\cup_{n=1}^\infty T_{n,\epsilon} B_1 \subset C(M)$ is almost surely relatively compact. 
\end{proof}
}

Next, we introduce the following lemma to bound the resolvent. The proof follows the standard approach, but we provide details for the sake of self-containedness. A similar proof can also be found in \cite{wang2015spectral}.
\begin{lemma}\label{proof Lemma: residue bound}
Let $\lambda$ be an isolated eigenvalue of an operator $T_{\epsilon}$. If $r<\frac{1}{2}\textup{dist}(\{\lambda\},\sigma(T_{\epsilon}) \setminus \{\lambda\})$, then 
\begin{equation}
\max_{z \in \Gamma_r(\lambda)} \|R_z(T_{\epsilon})\| \leq \frac{1}{r}. \nonumber 
\end{equation}
\end{lemma}
\begin{proof}
Fix any $z \in \Gamma_r(\lambda)$. {Set} $|\lambda_z| =\|R_z(T_{\epsilon})\| $, where $\lambda_z$ is the eigenvalue of $R_z(T_{\epsilon})$ with the largest magnitude. If $v$ is the corresponding eigenvector, we have $(zI-T_{\epsilon})^{-1}v=\lambda_z v$. Therefore, $(zI-T_{\epsilon})v=\frac{1}{\lambda_z} v$. We conclude that $T_{\epsilon}v=\big(z-\frac{1}{\lambda_z}\big) v$; that is, $z-\frac{1}{\lambda_z}$ is an eigenvalue of $T_{\epsilon}$. Since $r<\frac{1}{2}\textup{dist}(\{\lambda\},\sigma(T_{\epsilon}) \setminus \{\lambda\})$, we have $\textup{dist}(z, \sigma(T_{\epsilon}))=r \leq |z-(z-\frac{1}{\lambda_z})|$. We conclude that $|\lambda_z| \leq \frac{1}{r}$.
\end{proof}

Suppose $\lambda_{i,\epsilon}$ is the $i$-th smallest eigenvalue of $\frac{I-T_{\epsilon}}{\epsilon^2}$, then $\bar{\lambda}_{i,\epsilon}=1-\lambda_{i,\epsilon}\epsilon^2$ is the $i$-th largest eigenvalue of $T_{\epsilon}$. In other words, we have $\cdots\leq \bar{\lambda}_{2,\epsilon}\leq\bar{\lambda}_{1,\epsilon} \leq \bar{\lambda}_{0,\epsilon}$ for $T_{\epsilon}$. Note that $\frac{I-T_{\epsilon}}{\epsilon^2}$ and $T_{\epsilon}$ share the same eigenfunctions; hence, we denote $\phi_{i,\epsilon}$ to be the eigenfunction corresponding to $\lambda_{i,\epsilon}$ (hence to $\bar{\lambda}_{i,\epsilon}$). Similarly, let $\lambda_{i,n,\epsilon}$ be the $i$-th smallest eigenvalue of $\frac{I-T_{n,\epsilon}}{\epsilon^2}$, and $\bar{\lambda}_{i,n,\epsilon}=1-\lambda_{i,n,\epsilon}\epsilon^2$ be the $i$-th largest eigenvalue of $T_{n,\epsilon}$. Denote $\phi_{i, n, \epsilon}$ to be the eigenfunction corresponding to $\lambda_{i,n, \epsilon}$ (hence to $\bar{\lambda}_{i,n,\epsilon}$). We assume that both $\phi_{i,\epsilon}$ and $\phi_{i, n, \epsilon}$ are normalized in the $L^2$ norm. The following proposition describes the spectral convergence of the operator $T_{n,\epsilon}$ to the operator $T_{\epsilon}$. 
Since we assume that  the eigenvalues of $\Delta$ are simple, due to the convergence of the eigenvalues, we will have that for any $K\in \mathbb{N}$, there exists $\epsilon>0$ sufficiently small so that the first $K$ eigenvalues of $T_{\epsilon}$ are simple. Therefore, even though we study the spectral convergence from $T_{n,\epsilon}$ to $T_{\epsilon}$ in the next proposition, to be consistent, we still only assume that  the eigenvalues of $\Delta$ are simple rather than making the assumption on $T_{\epsilon}$.

\begin{proposition}\label{T epsilon and T n espilon 1}
Fix $K\in \mathbb{N}$. Assume that the eigenvalues of $\Delta$ are simple. {Suppose $\epsilon>0$ is sufficiently small so that \eqref{epsilon less than some constant} is satisfied. If  $n$ is large enough so that $\frac{1}{\sqrt{n}\epsilon^d} ( \sqrt{-\log\epsilon}+\sqrt{\log n})< \mathcal{C}_1$, {where $\mathcal C_1>0$ is defined in \eqref{Definition mathcal C_1}},  then there are $a_i \in \{1,-1\}$} such that with probability greater than $1-n^{-2}$,  for all $1 \leq i < K$, we have
\begin{align}
& |\lambda_{i,\epsilon}-\lambda_{i, n,\epsilon}|\leq  \frac{C_{13}}{\sqrt{n}\epsilon^{2d+5}}\left( \sqrt{-\log\epsilon}+\sqrt{\log n}\right) , \nonumber  \\
& \|a_i\phi_{i,n,\epsilon}-\phi_{i,\epsilon}\|_{\infty} \leq \frac{C_{13}}{\sqrt{n}\epsilon^{2d+3}}\left( \sqrt{-\log\epsilon}+\sqrt{\log n}\right) . \nonumber 
\end{align}
where {$C_{13}>0$} is a constant depending on { $d$, the diameter of $M$, $\mathsf p_m$ and the $C^0$ norm of $\mathsf p$.}
\end{proposition}

\begin{proof}
Suppose $\lambda_i$, $i=1,\ldots$, are the eigenvalues of $-\Delta$ so that $0<\lambda_1 < \lambda_2 < \ldots$. 
Recall the definition \eqref{defintion: mathsf Gamma K}: 
\begin{equation}
\mathsf\Gamma_K=\min_{1 \leq i \leq K}\textup{dist}(\lambda_i, \sigma(-\Delta)\setminus \{\lambda_i\})\,. \nonumber
\end{equation}
Consider an $i < K$. Based on the choice of $\epsilon$, we have shown in \eqref{Proof proposition 1 gap of the ith eigenvalue} in the proof of Proposition \ref {T epsilon and Delta} that 
\begin{equation}
\gamma_i\Big(\frac{I-T_{\epsilon}}{\epsilon^2}\Big) \geq\frac{1}{12}\mathsf\Gamma_K \,.
\end{equation}
Hence, 
\begin{equation}
\textup{dist}(\{\bar{\lambda}_{i,\epsilon}\}, \sigma(T_{\epsilon})\setminus \{\bar{\lambda}_{i,\epsilon}\})>\frac{\mathsf\Gamma_K}{12} \epsilon^2. 
\end{equation}
Choose $r= \frac{\mathsf\Gamma_K}{24} \epsilon^2$. 
By Lemma \ref{proof Lemma: residue bound}, we then have 
\begin{equation}
\max_{z \in \Gamma_r(\tilde{\lambda}_{i,\epsilon})}\|R_z(T_{\epsilon})\| \leq \frac{1}{r}=\frac{24}{\mathsf \Gamma_K \epsilon^2}\,.\label{proof RzTeps bound by GammaKeps2}
\end{equation} 
Moreover, if $\epsilon$ is small enough, $|z|>\frac{1}{2}$ for $z\in \Gamma_r(\tilde{\lambda}_{i,\epsilon})$.

Let $\textup{Pr}_{\phi_{i, n,\epsilon}}$ be the projection onto {$\texttt{span}\{\phi_{i, n,\epsilon}\}$}. { By Theorem \ref{atkinson1967}, we have }
\begin{align}
&\|\phi_{i, \epsilon}-\textup{Pr}_{\phi_{i, n,\epsilon}} \phi_{i, \epsilon}\|_{\infty} \nonumber \\
\leq &\,  \max_{z \in \Gamma_{r}} \frac{2r\|R_z(T_{\epsilon})\|}{\bar{\lambda}_{i,\epsilon}-r}(\|(T_{n,\epsilon}-T_{\epsilon}) \phi_{i, \epsilon}\|_{\infty} +\|R_z(T_{\epsilon}) \phi_{i, \epsilon}\|_{\infty}\|(T_{\epsilon}-T_{n,\epsilon})T_{n,\epsilon}\|) \nonumber \,,
\end{align}
where we use the fact that since $T_{\epsilon}$ and $T_{n,\epsilon}$ are compact self-adjoint operators, the eigenvalues are all real, {and the fact that} $\min_{z\in \Gamma_r(\bar \lambda_{i,\epsilon})}=\bar{\lambda}_{i,\epsilon}-r$.  By Proposition \ref {T epsilon and Delta}, we have $\lambda_{i,\epsilon} \leq \lambda_{i}+\epsilon^{\frac{3}{2}}$, and hence 
\begin{align}
\bar{\lambda}_{i,\epsilon}-r&\,=1-\lambda_{i,\epsilon}\epsilon^2- \frac{\mathsf\Gamma_K}{24} \epsilon^2 \geq 1-(\lambda_{i}+\epsilon^{\frac{3}{2}})\epsilon^2- \frac{\mathsf\Gamma_K}{24} \epsilon^2\nonumber\\
&\,\geq 1-(2 \lambda_{K}+\epsilon^{\frac{3}{2}})\epsilon^2 \geq \frac{1}{2}\,,
\end{align}
where in the last step, we use  (11') in Table \ref{Table:Relations}. By Lemma \ref{proof Lemma: residue bound}, $r\|R_z(T_{\epsilon})\|\leq 1$, $\|R_z(T_{\epsilon}) \phi_{i, \epsilon}\|_{\infty}\leq \|R_z(T_{\epsilon})\| \|\phi_{i, \epsilon}\|_{\infty}$ {and \eqref{proof RzTeps bound by GammaKeps2}, we have}
\begin{align}
\|\phi_{i, \epsilon}-\textup{Pr}_{\phi_{i, n,\epsilon}} \phi_{i, \epsilon}\|_{\infty} \leq 4\left( \left\|(T_{n,\epsilon}-T_{\epsilon})\frac{ \phi_{i,\epsilon}}{\|\phi_{i,\epsilon}\|_{\infty}}\right\|_{\infty} +\frac{24}{\mathsf\Gamma_K \epsilon^2}\|(T_{\epsilon}-
T_{n,\epsilon})T_{n,\epsilon}\| \right) \|\phi_{i,\epsilon}\|_{\infty}\nonumber\,.
\end{align}
{Next,} by Lemma \ref {fix u , T n espilon-T epsilon} and Lemma \ref{(T espilon- Tn epsilon)Tn epsilon}, {suppose $n$ is large enough so that $\frac{1}{\sqrt{n}\epsilon^d} ( \sqrt{-\log\epsilon}+\sqrt{\log n})< \mathcal{C}_1$}, with probability greater than $1-n^{-2}$, {we further have} 
\begin{align}
\|\phi_{i, \epsilon}-\textup{Pr}_{\phi_{i, n,\epsilon}} \phi_{i, \epsilon}\|_{\infty} \leq  \left(\frac{4C_8}{\sqrt{n}\epsilon^{2d}}+\frac{96C_{11}}{\mathsf\Gamma_K}\frac{1}{\sqrt{n}\epsilon^{2d+2}}\right) \left( \sqrt{-\log\epsilon}+\sqrt{\log n}\right) \|\phi_{i,\epsilon}\|_{\infty}. \label{proof phiieps and projection difference} 
\end{align}
By {\eqref{proof phiieps and projection difference} and} Lemma \ref{projection lemma}, there are $a_i \in \{1,-1\}$ such that 
\begin{align}
\|a_i\phi_{i,n,\epsilon}-\phi_{i,\epsilon}\|_{\infty} &\,\leq 2\|\phi_{i,\epsilon}-\textup{Pr}_{\phi_{i,n,\epsilon}} \phi_{i,\epsilon}\|_{\infty} \nonumber \\
&\,\leq  \left(\frac{8C_8}{\sqrt{n}\epsilon^{2d}}+\frac{192 C_{11}}{\mathsf\Gamma_K}\frac{1}{\sqrt{n}\epsilon^{2d+2}}\right) \left( \sqrt{-\log\epsilon}+\sqrt{\log n}\right) \|\phi_{i,\epsilon}\|_{\infty}\,.\label{convergence of eigenfunctions prop 2}
\end{align}

Next, we discuss the eigenvalues. We have 
\begin{align}
&\, |\bar{\lambda}_{i, \epsilon}-\bar{\lambda}_{i,n,\epsilon}|\|\phi_{i,\epsilon}\|_{\infty} 
= \|\bar{\lambda}_{i,\epsilon}\phi_{i,\epsilon}-\bar{\lambda}_{i, n,\epsilon}\phi_{i,\epsilon}\|_{\infty} \nonumber \\
\leq &\, \|\bar{\lambda}_{i, \epsilon}\phi_{i,\epsilon}-\bar{\lambda}_{i, n,\epsilon} a_i \phi_{i,n,\epsilon}\|_{\infty}+\|\bar{\lambda}_{i, n,\epsilon} a_i \phi_{i,n,\epsilon}-\bar{\lambda}_{i, n,\epsilon}\phi_{i,\epsilon}\|_{\infty} \nonumber\\
\leq &\, \|T_{\epsilon}\phi_{i,\epsilon}-T_{n,\epsilon} a_i \phi_{i,n,\epsilon}\|_{\infty}+|\bar{\lambda}_{i, n,\epsilon}|\| a_i \phi_{i,n,\epsilon}-\phi_{i,\epsilon}\|_{\infty}\,, \nonumber
\end{align}
where we control the difference by taking $\phi_{i,n,\epsilon}$ into account. The right hand side is further bounded by
\begin{align}
 &\, \|T_{\epsilon}\phi_{i,\epsilon}-T_{n,\epsilon}  \phi_{i,\epsilon}\|_{\infty}+\|T_{n,\epsilon}  \phi_{i,\epsilon}-T_{n,\epsilon} a_i \phi_{i,n,\epsilon}\|_{\infty}+|\bar{\lambda}_{i, n,\epsilon}|\| a_i \phi_{i,n,\epsilon}-\phi_{i,\epsilon}\|_{\infty} \nonumber \\
\leq &\,  \Big\|(T_{\epsilon}-T_{n,\epsilon})\frac{\phi_{i,\epsilon}}{\|\phi_{i,\epsilon}\|_{\infty}}\Big\|_{\infty}\|\phi_{i,\epsilon}\|_{\infty}+2\|T_{n,\epsilon}\|\|  \phi_{i,\epsilon}- a_i \phi_{i,n,\epsilon}\|_{\infty} \nonumber \\
\leq &\, \left[\frac{C_{8}}{\sqrt{n}\epsilon^{2d}}+\frac{2C_6^3}{C_5^3}\Big(\frac{8C_8}{\sqrt{n}\epsilon^{2d}}+\frac{{192}C_{11}}{\mathsf\Gamma_K}\frac{1}{\sqrt{n}\epsilon^{2d+2}}\Big) \right]\left( \sqrt{-\log\epsilon}+\sqrt{\log n}\right) \|\phi_{i,\epsilon}\|_{\infty} \nonumber\,, 
\end{align}
where we use Lemma \ref{bounds on operators} and Lemma \ref{fix u , T n espilon-T epsilon} in the last step.
Hence, we can cancel $\|\phi_{i,\epsilon}\|_{\infty} $ on both sides, and get
\begin{equation}\label{convergence of eigenvalues prop 2}
|\bar{\lambda}_{{i},\epsilon}-\bar{\lambda}_{{i},n,\epsilon}| \leq \left[\frac{C_{8}}{\sqrt{n}\epsilon^{2d}}+\frac{2C_6^3}{C_5^3}\Big(\frac{8C_8}{\sqrt{n}\epsilon^{2d}}+\frac{192 C_{11}}{\mathsf\Gamma_K}\frac{1}{\sqrt{n}\epsilon^{2d+2}}\Big) \right]\left( \sqrt{-\log\epsilon}+\sqrt{\log n}\right)\, .
\nonumber 
\end{equation}

Next, we simplify  \eqref{convergence of eigenfunctions prop 2}.
By  Lemma \ref{lemma hormander}, $\|\phi_{i}\|_\infty \leq C_1 \lambda_K^{\frac{d-1}{4}}$. Hence, by Proposition \ref {T epsilon and Delta}, we have $\|\phi_{i,\epsilon}\|_\infty \leq C_1 \lambda_K^{\frac{d-1}{4}}+\epsilon\leq 2C_1 \lambda_K^{\frac{d-1}{4}}$ when $\epsilon$ is sufficiently small. Therefore, if $\frac{2C_1 \lambda_K^{\frac{d-1}{4}}}{\min(\mathsf\Gamma_K,1)} \leq \frac{1}{\epsilon}$, which is sufficient when
\begin{align}\label{relation epsilon and lambda prop 2}
\epsilon \leq \frac{\min(\mathsf\Gamma_K,1)}{2C_1 \lambda_K^{\frac{d-1}{4}}+1},
\end{align}
then $\|\phi_{i,\epsilon}\|_\infty$, $\frac{\|\phi_{i,\epsilon}\|_\infty}{\mathsf\Gamma_K}$ and $\frac{1}{\mathsf\Gamma_K}$ are all bounded above by $\frac{1}{\epsilon}$. Note that \eqref{relation epsilon and lambda prop 2} is (11'') in Table \ref{Table:Relations}, and it is sufficient to require $\epsilon \leq \mathcal{K}_1 \min \Bigg(\bigg(\frac{\min(\mathsf\Gamma_K,1)}{\mathcal{K}_2+\lambda_K^{d/2+5}}\bigg)^2, \frac{1}{(\mathcal{K}_3+\lambda_K^{(5d+7)/4})^2}\Bigg)$.

In conclusion, we have
\begin{equation}
\|a_i\phi_{i,n,\epsilon}-\phi_{i,\epsilon}\|_{\infty} \leq \frac{C_{13}}{\sqrt{n}\epsilon^{2d+3}}\left( \sqrt{-\log\epsilon}+\sqrt{\log n}\right)\nonumber
\end{equation}
and
\begin{equation}
|\lambda_{{i},\epsilon}-\lambda_{{i}, n,\epsilon}|\leq \frac{|\bar{\lambda}_{i,\epsilon}-\bar{\lambda}_{i,n,\epsilon|}}{\epsilon^2} \leq \frac{C_{13}}{\sqrt{n}\epsilon^{2d+5}}\left( \sqrt{-\log\epsilon}+\sqrt{\log n}\right) \,, \nonumber 
\end{equation}
where {$C_{13}>0$} is a constant depending on the constants from $C_5$ to $C_{11}$, and hence it depends on $d$, the diameter of $M$, $\mathsf p_m$ and the $C^0$ norm of $\mathsf p$.
\end{proof}

{
\begin{remark}
We explain the reason why the eigenvalue's convergence rate is slower than that of the eigenfunction that is suggested in the previous proposition. Note that we can write
\begin{align}
&\, |\lambda_{i, \epsilon}-\lambda_{i,n,\epsilon}|\|\phi_{i,\epsilon}\|_{\infty} \nonumber\\
\leq & \|\frac{I-T_{\epsilon}}{\epsilon^2}\phi_{i,\epsilon}-\frac{I-T_{n,\epsilon}}{\epsilon^2}  \phi_{i,\epsilon}\|_{\infty}+\|\frac{I-T_{n,\epsilon}}{\epsilon^2}   \phi_{i,\epsilon}-\frac{I-T_{n,\epsilon}}{\epsilon^2} a_i \phi_{i,n,\epsilon}\|_{\infty}+|\lambda_{i, n,\epsilon}|\| a_i \phi_{i,n,\epsilon}-\phi_{i,\epsilon}\|_{\infty} \nonumber \\
= & \|\frac{T_{\epsilon}-T_{n,\epsilon}}{\epsilon^2}\phi_{i,\epsilon}\|_{\infty}+\|\frac{I-T_{n,\epsilon}}{\epsilon^2}  (\phi_{i,\epsilon}- a_i \phi_{i,n,\epsilon})\|_{\infty}+|\lambda_{i, n,\epsilon}|\| a_i \phi_{i,n,\epsilon}-\phi_{i,\epsilon}\|_{\infty} \nonumber 
\end{align}
We have a good control of the first term on the right hand side so that it matches the convergence rate of the eigenfunctions. However, we do not have a good control of the operator norm of $\frac{I-T_{n,\epsilon}}{\epsilon^2}$ in $L^\infty(M)$ since we use a trivial bound. Such trivial bound leads to an extra $\frac{1}{\epsilon^2}$, and hence the convergence rate of eigenvalues is slower.
\end{remark}

In the previous proposition, we assume $\epsilon$ is fixed and small enough and $n$ is large enough, then we characterize the spectral convergence rate of $\frac{I-T_{n,\epsilon}}{\epsilon^2}$ to $\frac{I-T_{\epsilon}}{\epsilon^2}$ in terms of $n$ and $\epsilon$. Next, we impose a relation between $n$ and $\epsilon$ to further simplify the above proposition. 

{\begin{corollary}\label{T epsilon and T n espilon 2}
Fix $K\in \mathbb{N}$. Assume that the eigenvalues of $\Delta$ are simple. {Suppose  \eqref{epsilon less than some constant} holds and $n$ is sufficiently large so that $\epsilon=\epsilon(n) \geq (\frac{\log n}{n})^{\frac{1}{4d+13}} $.} Then with probability greater than $1-n^{-2}$,  for all $1 \leq i < K$, we have
\begin{align}
|\lambda_{i,\epsilon}-\lambda_{i, n,\epsilon}|\leq  2C_{13} \epsilon^{\frac{3}{2}}.\nonumber  
\end{align}
{Suppose  \eqref{epsilon less than some constant} holds  and $n$ is sufficiently large so that $\epsilon=\epsilon(n) \geq (\frac{\log n}{n})^{\frac{1}{4d+8}}$.} Then there are $a_i \in \{1,-1\}$ such that with probability greater than $1-n^{-2}$,  for all $1 \leq i < K$, 
we have
\begin{align}
\|a_i\phi_{i,n,\epsilon}-\phi_{i,\epsilon}\|_{\infty} \leq 2C_{13} \epsilon. \nonumber 
\end{align}
The constant {$C_{13}>0$ is defined in Proposition \ref{T epsilon and T n espilon 1}.} 
\end{corollary}

\begin{proof}
When $(\frac{\log n}{n})^{\frac{1}{4d+13}} \leq \epsilon$, we have 
\begin{align}
\frac{\sqrt{\log{n}}}{\sqrt{n}}\frac{1}{\epsilon^{2d+5}}  \leq \epsilon^{\frac{3}{2}} \quad \hbox{and} \quad -\log\epsilon \leq \frac{\log n}{4d+13} \,.\nonumber
\end{align}
When $(\frac{\log n}{n})^{\frac{1}{4d+8}}\leq \epsilon$, we have 
\begin{align}
\frac{\sqrt{\log{n}}}{\sqrt{n}}\frac{1}{\epsilon^{2d+3}} \leq \epsilon  \quad \hbox{and} \quad -\log\epsilon \leq \frac{\log n}{4d+8} \,.\nonumber
\end{align}
Hence, we can check that  the condition $\frac{1}{\sqrt{n}\epsilon^d} ( \sqrt{-\log\epsilon}+\sqrt{\log n})< \mathcal{C}_1$ is satisfied in both cases when $n$ is sufficiently large, where $\mathcal{C}_1$ is defined in Proposition \ref{T epsilon and T n espilon 1}. The conclusion follows by substituting the above inequality into the bounds into Proposition \ref{T epsilon and T n espilon 1}.
\end{proof}
}
}

\subsection{Relation between $I-A$ and the operator $I-T_{n,\epsilon}$.} The following proposition relates the eigenfunctions of $I-T_{n,\epsilon}$ and the eigenvectors of the {kernel normalized} GL matrix $I-A$ { by using the Nystr{\"o}m extension}.

\begin{proposition}\label{relation between W and Q}
There is a bijection between the eigenpairs of $I-A$ and the eigenpairs of $I-T_{n,\epsilon}$ in the following sense: 
\begin{enumerate}
\item If $(\lambda, f)$ is an eigenpair of $I-T_{n,\epsilon}$, where $f \in L^\infty(M)$, then $(\lambda, v)$ is an eigenpair of $I-A$, where $v \in \mathbb{R}^n$ and $v(i)=f(x_i)$ for any $i$.
\item If $(\lambda, v)$ is an eigenpair of $I-A$, where $v \in \mathbb{R}^n$, let 
\begin{align}
f(x):=\frac{1}{\lambda} \frac{\frac{1}{n}\sum_{i=1}^n \frac{k_{\epsilon}(x,x_i)}{\frac{1}{n}q_{\epsilon}(x)\frac{1}{n} q_{\epsilon}(x_i)}v(i)}{\frac{1}{n}\sum_{i=1}^n \frac{k_{\epsilon}(x,x_i)}{\frac{1}{n}q_{\epsilon}(x) \frac{1}{n}q_{\epsilon}(x_i)}}, \nonumber
\end{align}
then $(\lambda, f)$ is an  eigenpair of $I-T_{n,\epsilon}$. Moreover $f \in C^\infty(M)$ and $v(i)=f(x_i)$ for any $i$.
\end{enumerate}
\end{proposition}

\begin{proof}
It is sufficient to prove the theorem for $T_{n,\epsilon}$ and $A$. 
For any $g \in C(M)$, denote $\vec{g}=(g(x_1), \cdots, g(x_n))^\top$. Fix any $x_k \in\{x_i\}_{i=1}^n$. By the definition of $A$, we have
\begin{align}
A\vec{g}^\top(k)=& \sum_{i=1}^n A_{ki}g(x_i)=\sum_{i=1}^n \frac{W_{ki}g(x_i)}{D_{kk}}= \frac{\sum_{i=1}^nW_{ki}g(x_i)}{\sum_{i=1}^n W_{ki}} \nonumber  \\
=&  \frac{\frac{1}{n}\sum_{i=1}^nW_{ki}g(x_i)}{\frac{1}{n}\sum_{i=1}^n W_{ki}}= \frac{\frac{1}{n}\sum_{i=1}^n \frac{k_{\epsilon}(x_k,x_i)}{q_{\epsilon}(x_k) q_{\epsilon}(x_i)}g(x_i)}{\frac{1}{n}\sum_{i=1}^n \frac{k_{\epsilon}(x_k,x_i)}{q_{\epsilon}(x_k) q_{\epsilon}(x_i)}}= \frac{\frac{1}{n}\sum_{i=1}^n \frac{k_{\epsilon}(x_k,x_i)}{\frac{1}{n}q_{\epsilon}(x_k)\frac{1}{n} q_{\epsilon}(x_i)}g(x_i)}{\frac{1}{n}\sum_{i=1}^n \frac{k_{\epsilon}(x_k,x_i)}{\frac{1}{n}q_{\epsilon}(x_k) \frac{1}{n}q_{\epsilon}(x_i)}} \nonumber \\
=& T_{n,\epsilon}g(x_k). \nonumber
\end{align}
Hence, if $\lambda$ is an eigenvalue of $T_{n,\epsilon}$ and $f(x)$ is the corresponding eigenfunction, then for any $x_k$, we have $T_{n,\epsilon}f(x_k)=\lambda f(x_k)$. Therefore, the vector $v \in \mathbb{R}^n$ with $v(k)= f(x_k)$ satisfies $Av=\lambda v$.

For the second statement, observe that 
\begin{align}
f(x_j)=\frac{1}{\lambda} \frac{\frac{1}{n}\sum_{i=1}^n \frac{k_{\epsilon}(x_j,x_i)}{\frac{1}{n}q_{\epsilon}(x_j)\frac{1}{n} q_{\epsilon}(x_i)}v(i)}{\frac{1}{n}\sum_{i=1}^n \frac{k_{\epsilon}(x_j,x_i)}{\frac{1}{n}q_{\epsilon}(x_j) \frac{1}{n}q_{\epsilon}(x_i)}}=\frac{Av}{\lambda}(j)=v(j). \nonumber
\end{align}
Next,
\begin{align}
T_{n,\epsilon}f(x)=\frac{\frac{1}{n}\sum_{i=1}^n \frac{k_{\epsilon}(x,x_i)}{\frac{1}{n}q_{\epsilon}(x)\frac{1}{n} q_{\epsilon}(x_i)}f(x_i)}{\frac{1}{n}\sum_{i=1}^n \frac{k_{\epsilon}(x,x_i)}{\frac{1}{n}q_{\epsilon}(x) \frac{1}{n}q_{\epsilon}(x_i)}}=\frac{\frac{1}{n}\sum_{i=1}^n \frac{k_{\epsilon}(x,x_i)}{\frac{1}{n}q_{\epsilon}(x)\frac{1}{n} q_{\epsilon}(x_i)}v(i)}{\frac{1}{n}\sum_{i=1}^n \frac{k_{\epsilon}(x,x_i)}{\frac{1}{n}q_{\epsilon}(x) \frac{1}{n}q_{\epsilon}(x_i)}}=\lambda f(x). \nonumber
\end{align}
At last, the smoothness of $f$ implies that it is a finite summation and quotient of nonzero smooth functions. {Finally, note that $W$ is positively definite by the Bochner theorem. We thus finish the proof.}
\end{proof}

\subsection{Renormalization of the eigenvectors of $\frac{I-A}{\epsilon^2}$.}  

Denote $\mu_{i,n,\epsilon}$ to be the $i$-th eigenvalue of $\frac{I-A}{\epsilon^2}$ with the associated eigenvector $\tilde{v}_{i,n,\epsilon}$ normalized in the $l^2$ norm. {By} Proposition \ref{relation between W and Q},  
\begin{align}
\tilde{\phi}_{i,n,\epsilon}(x):=\frac{1}{1-\mu_{i,n,\epsilon}\epsilon^2} \frac{\frac{1}{n}\sum_{j=1}^n \frac{k_{\epsilon}(x,x_j)}{\frac{1}{n}q_{\epsilon}(x)\frac{1}{n} q_{\epsilon}(x_j)}\tilde{v}_{i,n,\epsilon}(j)}{\frac{1}{n}\sum_{j=1}^n \frac{k_{\epsilon}(x,x_j)}{\frac{1}{n}q_{\epsilon}(x) \frac{1}{n}q_{\epsilon}(x_j)}}
\end{align}
is the $i$-th eigenfunction of $\frac{I-T_{n,\epsilon}}{\epsilon^2}$ with $\tilde{\phi}_{i,n,\epsilon}(x_j)=\tilde{v}_{i,n,\epsilon}(j)$. Note that by Corollary \ref{T epsilon and T n espilon 2} and Proposition\ref{relation between W and Q}, it is intuitive to expect that $\tilde{\phi}_{i,n,\epsilon}$ is close to a discretization of an eigenfunction of $T_\epsilon$. In the previous propositions, we compare the eigenfunctions of operators $\frac{I-T_{n,\epsilon}}{\epsilon^2}$, $\frac{I-T_{\epsilon}}{\epsilon^2}$ and $\Delta$ normalized in $L^2(M)$, we would expect to ``normalize'' $\tilde{v}_{i,n,\epsilon}$ in the $L^2(M)$ sense. In this subsection, we will make this intuition rigorous.
We have the following lemma by using the convergence of $\frac{\tilde{\phi}_{i,n,\epsilon}(x_j)}{\|\tilde{\phi}_{i,n,\epsilon}\|_{2}}$ to $\phi_i(x_j)$.

\begin{lemma}\label{2 norm of eigenvectors}
Fix $K\in \mathbb{N}$. Suppose  \eqref{epsilon less than some constant} holds  and $n$ is sufficiently large so that $\epsilon=\epsilon(n) \geq (\frac{\log n}{n})^{\frac{1}{4d+8}}$. For all $1 \leq i < K$, with probability greater than $1-n^{-2}$,  we have 
\begin{align}
& \max_{x_j}\left| \frac{\tilde{\phi}_{i,n,\epsilon}(x_j)}{\|\tilde{\phi}_{i,n,\epsilon}\|_{2}}\right| \epsilon^{\frac{1}{4}} \leq 2
\end{align}
and
\begin{align}
&\left|\frac{1}{n}\sum_{j=1}^n \frac{\tilde{\phi}^2_{i,n,\epsilon}(x_j)}{\mathsf p(x_j)}-\frac{1}{n}\sum_{j=1}^n \frac{\|\tilde{\phi}_{i,n,\epsilon}\|^2_{2} \phi^2_{i}(x_j)}{\mathsf p(x_j)}\right| \leq \frac{4(C_{13}+1)  \|\tilde{\phi}_{i,n,\epsilon}\|^2_{2} }{\mathsf{p}_m}\epsilon^{\frac{3}{4}}\,, \nonumber 
\end{align}
{where $C_{13}>0$ is defined in Proposition \ref{T epsilon and T n espilon 1} and it depends on $d$, the diameter of $M$, $\mathsf p_m$ and the $C^0$ norm of $\mathsf p$.}
\end{lemma}

\begin{proof}
By Proposition \ref{T epsilon and Delta} and Corollary \ref{T epsilon and T n espilon 2}, there are $a_i\in \{1,-1\}$ such that with probability greater than $1-n^{-2}$,  
\begin{align} 
\max_{x_j}\bigg|a_i \frac{\tilde{\phi}_{i,n,\epsilon}(x_j)}{\|\tilde{\phi}_{i,n,\epsilon}\|_{2}}-\phi_{i}(x_j)\bigg|\leq {2C_{13}} \epsilon+\epsilon \leq (2C_{13}+1) \epsilon\,,
\end{align}
{where the last bound holds when $\epsilon$ is sufficiently small.}
By the  triangular inequality, we have
\begin{align}
\max_{x_j}\left| \frac{\tilde{\phi}_{i,n,\epsilon}(x_j)}{\|\tilde{\phi}_{i,n,\epsilon}\|_{2}}\right| \leq \|\phi_i\|_\infty + (2C_{13}+1) \epsilon. \nonumber 
\end{align}
Therefore, 
\begin{align}
\max_{x_j}\left| \frac{\tilde{\phi}_{i,n,\epsilon}(x_j)}{\|\tilde{\phi}_{i,n,\epsilon}\|_{2}}\right| \epsilon^{\frac{1}{4}} \leq \|\phi_i\|_\infty \epsilon^{\frac{1}{4}}+  (2C_{13}+1)  \epsilon^{\frac{5}{4}}. \nonumber 
\end{align}
By Lemma \ref{lemma hormander}, $\|\phi_i\|_\infty \epsilon^{\frac{1}{4}}\leq C_1\lambda^{\frac{d-1}{4}}_K \epsilon^{\frac{1}{4}}$. 
Note that $C_1\lambda^{\frac{d-1}{4}}_K \epsilon^{\frac{1}{4}} \leq 1$ which is (12'') in Table \ref{Table:Relations}. 
Hence,
\begin{align}
\|\phi_i\|_\infty \epsilon^{\frac{1}{4}}\leq 1 \nonumber
\end{align}
{when \eqref{epsilon less than some constant} is satisfied.} 
Therefore, we conclude that 
\begin{align}
\max_{x_j}\Big| \frac{\tilde{\phi}_{i,n,\epsilon}(x_j)}{\|\tilde{\phi}_{i,n,\epsilon}\|_{2}}\Big| \epsilon^{\frac{1}{4}} \leq {2}\,, \nonumber 
\end{align}
{where the last bound holds when $\epsilon$ is sufficiently small.}
By the  triangular inequality again, we also have
\begin{align}
\max_{x_j}\Big|a_i \frac{\tilde{\phi}_{i,n,\epsilon}(x_j)}{\|\tilde{\phi}_{i,n,\epsilon}\|_{2}}+\phi_{i}(x_j)\Big|&\,\leq \max_{x_j}\bigg|a_i \frac{\tilde{\phi}_{i,n,\epsilon}(x_j)}{\|\tilde{\phi}_{i,n,\epsilon}\|_{2}}-\phi_{i}(x_j)\bigg|+2 \|\phi_i\|_\infty\leq 2\|\phi_i\|_\infty +   (2C_{13}+1)   \epsilon. \nonumber 
\end{align}
Therefore, we conclude that
\begin{align}
\max_{x_j}\left|\tilde{\phi}^2_{i,n,\epsilon}(x_j)-\|\tilde{\phi}_{i,n,\epsilon}\|^2_{2} \phi^2_{i}(x_j)\right| \leq \|\tilde{\phi}_{i,n,\epsilon}\|^2_{2} \Big(2(C_{13}+1)  \|\phi_i\|_\infty  \epsilon+ 2(C_{13}+1) ^2 \epsilon^2 \Big)\,, \nonumber 
\end{align}
and hence the claim {follows from}
\begin{align}
&\left|\frac{1}{n}\sum_{j=1}^n \frac{\tilde{\phi}^2_{i,n,\epsilon}(x_j)}{\mathsf p(x_j)}-\frac{1}{n}\sum_{j=1}^n \frac{\|\tilde{\phi}_{i,n,\epsilon}\|^2_{2} \phi^2_{i}(x_j)}{\mathsf p(x_j)}\right|\nonumber\\
\leq &\, \left( \frac{1}{n}  \sum_{j=1}^n\frac{1}{\mathsf p(x_j)} \right)\|\tilde{\phi}_{i,n,\epsilon}\|^2_{2} \Big(2(C_{13}+1)  \|\phi_i\|_\infty  \epsilon+ 2(C_{13}+1) ^2 \epsilon^2 \Big) \nonumber \\
\leq&\, \frac{\|\tilde{\phi}_{i,n,\epsilon}\|^2_{2}}{\mathsf{p}_m}\Big(2(C_{13}+1)  \|\phi_i\|_\infty  \epsilon+ 2(C_{13}+1) ^2 \epsilon^2 \Big) \,,\nonumber 
\end{align}
where we use the trivial bound $ \frac{1}{n}  \sum_{j}\frac{1}{\mathsf p(x_j)}\leq 1/\mathsf{p}_m$. By using $\|\phi_i\|_\infty \epsilon^{\frac{1}{4}}\leq 1$ again, the conclusion follows when $\epsilon$ is small enough.
\end{proof}

Let $\mathbb{N}(j):=|B^{\mathbb{R}^p}_{\epsilon}(\iota(x_j)) \cap \{\iota(x_1), \cdots, \iota(x_n)\}|.$
By usual kernel density estimation (Lemma B.5 and Lemma E.1 in \cite{wu2018think}), we have the following lemma.
\begin{lemma}\label{KDE}
When $n$ is large enough, we have with probability greater than $1-n^{-2}$ that for all $j=1,\ldots,n$ 
{\begin{equation}
\hat{\mathsf p}(x_j):=\frac{d\mathbb{N}(j)}{|S^{d-1}|n\epsilon^d} =\mathsf p(x_j) +O(\epsilon^2)+ O\Big(\frac{\sqrt{\log (n)}}{n^{1/2}\epsilon^{d/2}}\Big)\,,\nonumber
\end{equation}
}
where $|S^{d-1}|$ is the volume of $S^{d-1}$ and the {implied} constant in $O(\epsilon^2)$ depends on the $C^2$ norm of $\mathsf p$.
\end{lemma}

Since $\phi_i$ is normalized in $L^2(M)$, in next lemma, we show that we can construct an approximation of the constant $1$, and hence an approximation of $\|\tilde{\phi}_{i,n,\epsilon}\|_2$ by using $\phi_{i}$.

\begin{lemma}\label{deviation in L2 norm}
Fix $K\in \mathbb{N}$. 
Suppose  \eqref{epsilon less than some constant} holds  and $n$ is sufficiently large so that $\epsilon=\epsilon(n) \geq (\frac{\log n}{n})^{\frac{1}{4d+8}}$. For all $0 \leq i <K$, we have with probability greater than $1-n^{-2}$ that
\begin{equation}
\Big|\frac{1}{n}\sum_{j=1}^n \frac{\phi^2_{i}(x_j)}{\mathsf p(x_j)}-1\Big|< \epsilon. \nonumber 
\end{equation}
\end{lemma}
\begin{proof}
Suppose $\mathsf p$ is the p.d.f. of the random variable $X$ with the range $M$. Define a random variable $F:=\frac{\phi_i(X)^2}{\mathsf p(X)}$. Then, $\mathbb{E}[F]=\|\phi_i\|^2_2=1$ and ${F_j:}=\frac{\phi^2_{i}(x_j)}{\mathsf p(x_j)}$ can be regarded as i.i.d. samples from $F$. {By Lemma \ref{lemma hormander}, we} have
\begin{align}
b:=\Big\|\frac{\phi_i^2}{\mathsf p}\Big\|_{\infty} \leq  \frac{C^2_1}{\mathsf{p}_m}\lambda^{(d-1)/2}_K, \nonumber 
\end{align}
where $C_1$ is defined in Lemma \ref{lemma hormander}, and 
\begin{align}
\mathbb{E}[F^2] =& \int_{M} \frac{\phi_i(x)^4}{\mathsf p^2(x)} \mathsf p(x) dx \leq \Big\|\frac{\phi_i^2}{\mathsf p}\Big\|_{\infty} \int_{M} \phi_i(x)^2 dx = \Big\|\frac{\phi_i^2}{\mathsf p}\Big\|_{\infty}\|\phi_i\|^2_2 \leq b. \nonumber
\end{align}
Hence, 
\begin{align}
\sigma^2:=\text{Var}(F)  \leq b-1\,. \nonumber 
\end{align} 
We apply Bernstein's inequality to provide a large deviation bound.  Recall that for $\beta>0$, the Bernstein's inequality is 
\begin{equation}
\Pr \left\{\Big|\frac{1}{n}\sum_{j=1}^n F_j- \mathbb{E}[F]\Big| > \beta\right\} \leq e^{-\frac{n\beta^2}{2\sigma^2 + \frac{2}{3}b\beta}}\,.\nonumber
\end{equation}
For $\beta=\epsilon<1$,
\begin{equation}
\frac{n\epsilon^2}{2\sigma^2 + \frac{2}{3}b\beta} \geq \frac{n\epsilon^2}{3b}.
\end{equation} 
{Note that} $\frac{n\epsilon^2}{2\sigma^2 + \frac{2}{3}b\beta} \geq n \epsilon^3$ when $3b\epsilon \leq 1${, which} is satisfied when 
\begin{align}
\frac{C^2_1}{\mathsf{p}_m}\lambda^{(d-1)/2}_K \epsilon \leq \frac{1}{3}, \nonumber
\end{align}
which is (13'') in Table \ref{Table:Relations}. Note that (13'') in Table \ref{Table:Relations} is satisfied when {\eqref{epsilon less than some constant} is satisfied.}
Therefore, we have
\begin{equation}
\Pr \left\{\Big|\frac{1}{n}\sum_{j=1}^n F_j- \mathbb{E}[F]\Big| > \epsilon \right\} \leq e^{- n \epsilon^3}\,.\nonumber
\end{equation}
Note that $e^{- n \epsilon^3} \leq \frac{1}{n^2}$ when $\epsilon^3 \geq \frac{2\log n}{n}$, {which is satisfied by the assumption $(\frac{\log n}{n})^{\frac{1}{4d+8}} \leq \epsilon$.}
\end{proof}

If we combine the above three lemmas, we have the following proposition.

\begin{proposition}\label{Convergence of normalized eigenvectors}
Fix $K\in \mathbb{N}$. 
Suppose  \eqref{epsilon less than some constant} holds  and $n$ is sufficiently large so that $\epsilon=\epsilon(n) \geq (\frac{\log n}{n})^{\frac{1}{4d+8}}$. Then, with probability greater than $1-n^{-2}$,  for all $1 \leq i < K$, 
we have
\begin{align}
\max_{x_j} \Big|\frac{\tilde{v}_{i,n,\epsilon}(j)}{\|\tilde{v}_{i,n,\epsilon}\|_{l^2(1/\hat{\mathsf p})}}-\frac{\tilde{\phi}_{i,n,\epsilon}(x_j)}{\|\tilde{\phi}_{i,n,\epsilon}\|_{2}}\Big| \leq {\mathcal{K}_4} \epsilon^{1/2} \,,\nonumber 
\end{align}
where the constant ${\mathcal{K}_4}>0$ is a constant depending on {$d$, the diameter of $M$, $\mathsf p_m$, and the $C^2$ norm of $\mathsf p$.}
\end{proposition}
\begin{proof}By Lemma \ref{KDE}, with probability greater than $1-n^{-2}$,
\begin{align}
\frac{|S^{d-1}|\epsilon^d}{d} \sum_{j=1}^n \frac{\tilde{v}^2_{i,n,\epsilon}(j)}{\mathbb{N}(j)}= \frac{1}{n}\sum_{j=1}^n \frac{\tilde{v}^2_{i,n,\epsilon}(j)}{\frac{d\mathbb{N}(j)}{|S^{d-1}|n\epsilon^d}}= \frac{1}{n}\sum_{j=1}^n \frac{\tilde{v}^2_{i,n,\epsilon}(j)}{\mathsf p(x_j)}\Big[1+O(\epsilon^2)+ O\Big(\frac{\sqrt{\log (n)}}{n^{1/2}\epsilon^{d/2}}\Big)\Big]\,, \nonumber 
\end{align}
{where the constant in $O(\epsilon^2)$ depends on the $C^2$ norm of $\mathsf p$.}
Hence, 
\begin{align}\label{nbhd and pdf}
\frac{|S^{d-1}|\epsilon^d}{d} \sum_{j=1}^n \frac{\tilde{v}^2_{i,n,\epsilon}(j)}{\mathbb{N}(j)}\Big[1+O(\epsilon^2)+ O\Big(\frac{\sqrt{\log (n)}}{n^{1/2}\epsilon^{d/2}}\Big)\Big]= \frac{1}{n}\sum_{i=1}^n \frac{\tilde{v}^2_{i,n,\epsilon}(j)}{\mathsf p(x_j)}.
\end{align}
By the assumption $ (\frac{\log n}{n})^{\frac{1}{4d+8}} {=}\epsilon$, $\frac{\sqrt{\log (n)}}{n^{1/2}\epsilon^{d/2}}{\leq}\epsilon^2$. Therefore, we have
\begin{align}
\frac{|S^{d-1}|\epsilon^d}{d} \sum_{j=1}^n \frac{\tilde{v}^2_{i,n,\epsilon}(j)}{\mathbb{N}(j)}(1+O(\epsilon^2))= \frac{1}{n}\sum_{j=1}^n \frac{\tilde{v}^2_{i,n,\epsilon}(j)}{\mathsf p(x_j)}=\frac{1}{n}\sum_{i=1}^n \frac{\tilde{\phi}^2_{i,n,\epsilon}(x_j)}{\mathsf p(x_j)}, \nonumber 
\end{align}
{where the constant in $O(\epsilon^2)$ depends on $d$, the diameter of $M$, $\mathsf p_m$, and the $C^2$ norm of $\mathsf p$.}
By Lemma \ref{deviation in L2 norm},  with probability greater than $1-n^{-2}$,
\begin{align}
 \left|\frac{1}{n}\sum_{j=1}^n \frac{\|\tilde{\phi}_{i,n,\epsilon}\|^2_{2}\phi^2_{i}(x_j)}{\mathsf p(x_j)}-\|\tilde{\phi}_{i,n,\epsilon}\|^2_{2}\right|< \epsilon \|\tilde{\phi}_{i,n,\epsilon}\|^2_{2}\,. \label{proof tildephiphi2/p -tilde phi22} 
\end{align}
Hence, {by \eqref{proof tildephiphi2/p -tilde phi22} and} Lemma \ref{2 norm of eigenvectors}, {when $\epsilon$ is sufficiently small, we have}
\begin{align}
\left|\frac{1}{n}\sum_{j=1}^n \frac{\tilde{\phi}^2_{i,n,\epsilon}(x_j)}{\mathsf p(x_j)}-\|\tilde{\phi}_{i,n,\epsilon}\|^2_{2}\right| \leq  \frac{8(C_{13}+1)}{\mathsf{p}_m} \|\tilde{\phi}_{i,n,\epsilon}\|^2_{2} \epsilon^{\frac{3}{4}}\,, \nonumber 
\end{align}
{where $C_{13}$ depends on $d$, the diameter of $M$, $\mathsf p_m$ and the $C^0$ norm of $\mathsf p$.}
Hence,
\begin{align}
\bigg|\frac{|S^{d-1}|\epsilon^d}{d} \sum_{j=1}^n \frac{\tilde{v}^2_{i,n,\epsilon}(j)}{\mathbb{N}(j)}(1+O(\epsilon^2))-\|\tilde{\phi}_{i,n,\epsilon}\|^2_{2}\bigg| \leq    \frac{8(C_{13}+1)}{\mathsf{p}_m} \|\tilde{\phi}_{i,n,\epsilon}\|^2_{2} \epsilon^{\frac{3}{4}},\, \nonumber 
\end{align}
{where the constant in $O(\epsilon^2)$ depends on $d$, the diameter of $M$, $\mathsf p_m$, and the $C^2$ norm of $\mathsf p$.}
A straightforward expansion shows that, 
\begin{align}
\bigg|\frac{|S^{d-1}|\epsilon^d}{d} \sum_{j=1}^n \frac{\tilde{v}^2_{i,n,\epsilon}(j)}{\mathbb{N}(j)}-\|\tilde{\phi}_{i,n,\epsilon}\|^2_{2}\bigg| \leq \tilde{C}_{1} \epsilon^{\frac{3}{4}}\|\tilde{\phi}_{i,n,\epsilon}\|^2_{2}\,, \nonumber 
\end{align}
for { a constant $\tilde{C}_{1}$ depending on $d$, the diameter of $M$, $\mathsf p_m$, and the $C^2$ norm of $\mathsf p$}. Simply expand the absolute value and take the square root, we have 
\begin{align}
(1-{\tilde{C}_1}\epsilon^{\frac{3}{4}})\|\tilde{\phi}_{i,n,\epsilon}\|_{2} \leq \|\tilde{v}_{i,n,\epsilon}\|_{l^2({1/\hat{\mathsf{p}}})} \leq (1+{\tilde{C}_1}\epsilon^{\frac{3}{4}})\|\tilde{\phi}_{i,n,\epsilon}\|_{2} \,.
\end{align}

{At last, {by the fact that $\tilde{\phi}_{i,n,\epsilon}(x_j)=\tilde{v}_{i,n,\epsilon}(j)$, when $\epsilon$ is sufficiently small,} we have
\begin{align}
&\bigg|\frac{\tilde{v}_{i,n,\epsilon}(j)}{\|\tilde{v}_{i,n,\epsilon}\|_{l^2(1/\hat{\mathsf p})}}-\frac{\tilde{\phi}_{i,n,\epsilon}(x_j)}{\|\tilde{\phi}_{i,n,\epsilon}\|_{2}}\bigg| \leq \frac{{\tilde{C}_1}}{1+{\tilde{C}_1} \epsilon^{\frac{3}{4}}} \epsilon^{\frac{3}{4}} \frac{|{\tilde{\phi}_{i,n,\epsilon}(x_j)}|}{\|\tilde{\phi}_{i,n,\epsilon}\|_{2}}\nonumber\\
 \leq&\, \frac{{2\tilde{C}_1}}{1{+\tilde{C}_1} \epsilon^{\frac{3}{4}}} \epsilon^{1/2}<2{\tilde{C}_1}\epsilon^{1/2}=\mathcal{K}_4 \epsilon^{1/2} \,, \nonumber 
\end{align}
where we apply Lemma \ref{2 norm of eigenvectors} again in the second step, {and $\mathcal{K}_4=2{\tilde{C}_1}>0$} is a constant depending on $d$, the diameter of $M$, $\mathsf p_m$, and the $C^2$ norm of $\mathsf p$.}
\end{proof}

\subsection{Finish the Proof of Theorem \ref{spectral convergence of L on closed manifold}}
With the above preparation, we are finally ready to prove Theorem \ref{spectral convergence of L on closed manifold}.

\begin{proof}
Combining Propositions \ref{T epsilon and Delta}, {\ref{relation between W and Q}} and Corollary \ref{T epsilon and T n espilon 2}, {when $(\frac{\log n}{n})^{\frac{1}{4d+13}}\leq \epsilon$,} we have for each $1 \leq i <K$, 
\begin{align}
|\mu_{i,n,\epsilon}-\lambda_{i}|\leq \epsilon^{3/2}+{2C_{13}}\epsilon^{3/2}\,. \nonumber 
\end{align}
{Further with Proposition \ref{Convergence of normalized eigenvectors}, when $(\frac{\log n}{n})^{\frac{1}{4d+8}} \leq \epsilon$,}  we know that there are  $a_i \in \{1,-1\}$ such that for all $1 \leq i< K$, with probability greater than $1-
n^{-2}$,   
\begin{align}
\max_{x_j}|a_i v_{i,n,\epsilon}(j)-\phi_{i}(x_j)|\leq {2C_{13}} \epsilon + \epsilon+ {\mathcal{K}_4} \epsilon^{1/2}. \nonumber 
\end{align}

In summary, we have 
\begin{align}
& |\mu_{i,n,\epsilon}-\lambda_{i}|\leq \Omega_1 \epsilon^{3/2}, \nonumber \\
& |1-\epsilon^2\mu_{i,n,\epsilon}-e^{-\epsilon^2\lambda_{i}}|\leq \Omega_1 \epsilon, \nonumber \\
& \max_{x_j}|a_i v_{i,n,\epsilon}(j)-\phi_{i}(x_j)|\leq  \Omega_2 \epsilon^{1/2}, \nonumber 
\end{align}
where $\Omega_1={2C_{13}}+1$ and $ \Omega_2={2C_{13}+\mathcal{K}_4}+1$. Hence $\Omega_1$ and $\Omega_2$ depend on {$d$, the diameter of $M$, $\mathsf p_m$, and the $C^2$ norm of $\mathsf p$.}.

Note that Lemma \ref{lemma hormander} does not include $\lambda_0=0$ case, {so we need to further prove the case when $i=0$}.  Recall that $\mu_{0,n,\epsilon}$ is the smallest eigenvalue of $\frac{I-A}{\epsilon^2}$. The eigenvalues of $A$ are between $[0,1]$ and $1$ is an eigenvalue of $A$ corresponding to the constant eigenvector. Therefore $\mu_{0,n,\epsilon}=0$. And $ |\mu_{0,n,\epsilon}-\lambda_{0}|\leq \Omega_1 \epsilon^{3/2}$ holds trivially. If $\tilde{v}_{0,n,\epsilon}$ is the eigenvector corresponding to $\mu_{0,n,\epsilon}$ normalized in $l^2$, then $\tilde{v}_{0,n,\epsilon}=[\frac{1}{\sqrt{n}}, \cdots, \frac{1}{\sqrt{n}}]$ or $[-\frac{1}{\sqrt{n}}, \cdots, -\frac{1}{\sqrt{n}}]$. We know that $\phi_0=\frac{1}{\sqrt{\texttt{Vol}(M)}}$, hence by the same argument as Lemma \ref{deviation in L2 norm}, we have with probability greater than $1-n^{-2}$ that
\begin{equation}
\Big|\frac{1}{n}\sum_{j=1}^n \frac{\frac{1}{\texttt{Vol}(M)}}{\mathsf p(x_j)}-1\Big|< \epsilon. \label{proof: estimate the volume by finite p sum} 
\end{equation}
By Lemma \ref{KDE}, with probability greater than $1-n^{-2}$,
\begin{align}
\frac{|S^{d-1}|\epsilon^d}{d} \sum_{j=1}^n \frac{\tilde{v}^2_{0,n,\epsilon}(j)}{\mathbb{N}(j)}= \frac{1}{n^2}\sum_{j=1}^n \frac{1}{\frac{d\mathbb{N}(j)}{|S^{d-1}|n\epsilon^d}}= \frac{1}{n^2}\sum_{j=1}^n \frac{1}{\mathsf p(x_j)}\Big[1+O(\epsilon)+ O\Big(\frac{\sqrt{\log (n)}}{n^{1/2}\epsilon^{d/2}}\Big)\Big]\,. \nonumber 
\end{align}
{When $ (\frac{\log n}{n})^{\frac{1}{4d+8}} \leq\epsilon$, {with \eqref{proof: estimate the volume by finite p sum},} the above equation implies that
\begin{align}
\frac{|S^{d-1}|\epsilon^d}{d} \sum_{j=1}^n \frac{\tilde{v}^2_{0,n,\epsilon}(j)}{\mathbb{N}(j)}= \frac{\texttt{Vol}(M)}{n}(1+O(\epsilon))\,. \nonumber 
\end{align}
Therefore,  $v_{0,n,\epsilon}(j)=\frac{1}{\sqrt{\texttt{Vol}(M)}}(1+O(\epsilon))$, where the {implied constant depends on the} volume of $M$, and the $C^2$ norm of $\mathsf p$.
In this case,
\begin{align}\label{0 case in the proof}
\max_{x_j}|a_n v_{0,n,\epsilon}(j)-\phi_{0}(x_j)|\leq  \Omega_2 \epsilon, 
\end{align}
where $\Omega_2$ depends on the volume of $M$, and the $C^2$ norm of $\mathsf p$.} With the $i>0$ cases, the conclusion follows.
\end{proof}

\section{Proof of Theorem \ref{heat kernel reconstruction 1}}\label{proof of theorem 3}
Note that 
\begin{align}
&\left|[\mathsf H^{(K)}_{\epsilon,t}]_{ij}-\mathsf{H}(x_i,x_j,t)\right| \nonumber \\
\leq & \bigg|\sum_{l=0}^{K-1} \big(e^{-\mu_{l,n,\epsilon} t} v_{l,n,\epsilon}(i) v^\top_{l,n,\epsilon}(j)-e^{-\lambda_l t} \phi_l(x_i)\phi_l(x_j)\big)\bigg| +
\bigg|\sum_{l=K}^\infty e^{-\lambda_l t} \phi_l(x_i)\phi_l(x_j) \bigg|. \nonumber
\end{align}
 {
Observe that by the Cauchy-Schwartz inequality, we have
\begin{align}
\bigg|\sum_{l=K}^\infty e^{-\lambda_l t} \phi_l(x)\phi_l(y) \bigg| \leq \sup_{x \in M} \sum_{l=K}^\infty e^{-\lambda_l t} \phi_l(x)^2.
\end{align}
By  \cite[Page 393]{berard1994embedding}, 
\begin{align}
\sup_{x \in M} \sum_{l=K}^\infty e^{-\lambda_l t} \phi_l(x)^2 \leq \tilde{c}_1 t^{-\frac{d}{2}}\int_{\lambda_K t}^\infty s^{\frac{d}{2}}e^{-s}ds,
\end{align}
where $\tilde{c}_1$ depends on $d$, the Ricci curvature and the diameter of $M$.
By Lemma \ref{laplace eigenvalue lower bound},
\begin{align}
\sup_{x \in M} \sum_{l=K}^\infty e^{-\lambda_l t} \phi_l(x)^2 \leq \tilde{c}_1 t^{-\frac{d}{2}}\int_{C_2 K^{\frac{2}{d}}t}^\infty s^{\frac{d}{2}}e^{-s}ds,
\end{align}
where $C_2$ depends on $d$, the Ricci curvature and the diameter of $M$.
Note that the Gamma function with $z>0$,
\begin{align}
\int_{z}^\infty s^{\alpha}e^{-s}ds \leq z^{\alpha+1}e^{-z},
\end{align}
when $z-\alpha \geq 1$. Hence, if $C_2 K^{\frac{2}{d}}t \geq \frac{d}{2}+1$, i.e. $\frac{1+\frac{d}{2}}{C_2} K^{-\frac{2}{d}} \leq t$,  then 
\begin{align}
\bigg|\sum_{l=K}^\infty e^{-\lambda_l t} \phi_l(x)\phi_l(y) \bigg| \leq & \sup_{x \in M} \sum_{l=K}^\infty e^{-\lambda_l t} \phi_l(x)^2 \leq \tilde{c}_1 t^{-\frac{d}{2}}\int_{C_2 K^{\frac{2}{d}}t}^\infty s^{\frac{d}{2}}e^{-s}ds \label{proof heat kernel approximation high frequency part}\\
\leq & \tilde{c}_1 t^{-\frac{d}{2}} (C_2 K^{\frac{2}{d}}t)^{\frac{d}{2}+1}e^{-C_2 K^{\frac{2}{d}}t}=\tilde{c}_1 C_2^{\frac{d}{2}+1} K^{\frac{d+1}{d}}t e^{-C_2 K^{\frac{2}{d}}t} \nonumber \\
\leq&  2 \tilde{c}_1 C_2^{\frac{d}{2}}K \frac{C_2}{2} K^{\frac{2}{d}}t e^{-C_2 K^{\frac{2}{d}}t}.\nonumber 
\end{align}
Set $A:=\frac{C_2}{2} K^{\frac{2}{d}}t$. {The right hand side of \eqref{proof heat kernel approximation high frequency part} becomes $2\tilde{c}_1C_2^{d/2}KAe^{-2A}$. Note that the bound} $KAe^{-2A} \leq \frac{1}{K}e^{-A}$ is equivalent to $A-\log A \geq 2 \log K$. {Recall the trivial bound} $x-\log x \geq \frac{x}{2}$ {for any $x\geq 0$. Thus, to have $A-\log A \geq 2 \log K$,} it is sufficient to require $A \geq 4 \log K$, i.e. $\frac{8}{C_2}\frac{\log K}{K^{\frac{2}{d}}} \leq t$.
{Therefore}, if ${\frac{8}{C_2}} \frac{\log K}{K^{\frac{2}{d}}} \leq t$, 
\begin{align}
\bigg|\sum_{l=K}^\infty e^{-\lambda_l t} \phi_l(x)\phi_l(y) \bigg| \leq  {2\tilde{c}_1 C_2^{\frac{d}{2}}} e^{-\frac{C_2}{2} K^{\frac{2}{d}}t} \leq {2\tilde{c}_1 C_2^{\frac{d}{2}}} e^{-\frac{C_2}{2} t},\nonumber 
\end{align}
where we use $K \geq 1$ in the last step.}

Next, we control the term with $l<K$. By the triangle inequality, we have
\begin{align}
& \Big|\sum_{l=0}^{K-1} \big(e^{-\mu_{l,n,\epsilon} t} v_{l,n,\epsilon}(i) v^\top_{l,n,\epsilon}(j)-e^{-\lambda_l t} \phi_l(x_i)\phi_l(x_j)\big)\Big| \nonumber \\
\leq & \sum_{l=0}^{K-1} \Big(|e^{-\mu_{l,n,\epsilon} t}-e^{-\lambda_l t}||\phi_l(x_i)\phi_l(x_j)|+|v_{l,n,\epsilon}(i) v^\top_{l,n,\epsilon}(j)-\phi_l(x_i)\phi_l(x_j)||e^{-\mu_{l,n,\epsilon} t}|\Big). \nonumber
\end{align}
By Theorem \ref{spectral convergence of L on closed manifold} { and Remark \ref{consistent relation}}, with probability greater than $1-n^{-2}$, for all $l < K$, 
\begin{align}
& |\mu_{l,n,\epsilon}-\lambda_{l}|\leq  \Omega_1\epsilon^{\frac{3}{2}}, \nonumber  \\
& \max_{x_i}|a_n v_{l,n,\epsilon}(i)-\phi_{l}(x_i)|\leq  \Omega_2 \epsilon^{1/2}. \nonumber 
\end{align}
 {
Note that $\mu_{0,n,\epsilon}=\lambda_0=0$. Hence 
$$\sum_{l=0}^{K-1} |e^{-\mu_{l,n,\epsilon} t}-e^{-\lambda_l t}||\phi_l(x_i)\phi_l(x_j)|=\sum_{l=1}^{K-1}|e^{-\mu_{l,n,\epsilon} t}-e^{-\lambda_l t}||\phi_l(x_i)\phi_l(x_j)|.$$ 
For $l \geq 1$, by Lemma \ref{laplace eigenvalue lower bound},
\begin{align}
|e^{-\mu_{l,n,\epsilon} t}-e^{-\lambda_l t}| \leq e^{-\lambda_l t} \Omega_1\epsilon^{\frac{3}{2}} t \leq \Omega_1\epsilon^{\frac{3}{2}} t e^{-\lambda_1 t} \leq  \Omega_1\epsilon^{\frac{3}{2}} t e^{-\frac{C_2}{2}t}. 
\end{align}
Moreover, since $1 \leq l$, 
\begin{align}
|\phi_l(x_i)\phi_l(x_j)| \leq C^2_1 \lambda_l^{(d-1)/2} \leq  C^2_1 \lambda_K^{(d-1)/2}\,. \nonumber 
\end{align}
Therefore,
\begin{align}
\sum_{l=0}^{K-1} |e^{-\mu_{l,n,\epsilon} t}-e^{-\lambda_l t}||\phi_l(x_i)\phi_l(x_j)| \leq \Omega_1 C^2_1 K\lambda_K^{(d-1)/2} \epsilon^{\frac{3}{2}} t e^{-\frac{C_2}{2}t}  \,. \nonumber  
\end{align}
}
Next, we bound the term $|v_{l,n,\epsilon}(i) v^\top_{l,n,\epsilon}(j)-\phi_l(x_i)\phi_l(x_j)||e^{-\mu_{l,n,\epsilon} t}|$. Note that for $l \geq 1$
\begin{align}
|v_{l,n,\epsilon}(i) v^\top_{l,n,\epsilon}(j)-\phi_l(x_i)\phi_l(x_j)||e^{-\mu_{l,n,\epsilon} t}| \leq & |v_{l,n,\epsilon}(i) v^\top_{l,n,\epsilon}(j)-\phi_l(x_i)\phi_l(x_j)| e^{-\lambda_1 t} \\
\leq& |v_{l,n,\epsilon}(i) v^\top_{l,n,\epsilon}(j)-\phi_l(x_i)\phi_l(x_j)| e^{- C_2t} . \nonumber 
\end{align}
Also, for $1 \leq l < K$,
\begin{align}
& |v_{l,n,\epsilon}(i) v^\top_{l,n,\epsilon}(j)-\phi_l(x_i)\phi_l(x_j)| \nonumber \\
\leq &\, \max_{x_i}|a_nv_{l,n,\epsilon}(i)-\phi_l(x_i)|(\max_{x_i}|\phi_l(x_i)|+\max_{x_i}|v_{l,n,\epsilon}(i)|) \nonumber \\
\leq &\, (2 \max_{x_i}|\phi_l(x_i)| + \Omega_2 \epsilon^{1/2})\max_{x_i}|a_nv_{l,n,\epsilon}(i)-\phi_l(x_i)| \nonumber \\
\leq &\, (2C_1 \lambda_l^{\frac{d-1}{4}} +\Omega_2 \epsilon^{1/2})\Omega_2 \epsilon^{1/2} \leq  2C_1 \lambda_K^{\frac{d-1}{4}} \Omega_2 \epsilon^{1/2} +\Omega^2_2 \epsilon \,,\nonumber
\end{align}
where the last two bounds come from Theorem \ref{spectral convergence of L on closed manifold}.

{When $l=0$, similarly, by \eqref{0 case in the proof}, we have 
\begin{align}
&|v_{0,n,\epsilon}(i) v^\top_{0,n,\epsilon}(j)-\phi_0(x_i)\phi_0(x_j)||e^{-\mu_{0,n,\epsilon} t}|\nonumber\\
=& |v_{0,n,\epsilon}(i) v^\top_{0,n,\epsilon}(j)-\phi_0(x_i)\phi_0(x_j)|
\leq  \left(\frac{2}{\sqrt{\texttt{Vol}(M)}}+\tilde{c}_2 \epsilon\right)\tilde{c}_2 \epsilon \leq \tilde{c}_3 \epsilon,  \nonumber 
\end{align}
for $\tilde{c}_3>0$ depending on the volume of $M$, and the $C^2$ norm of $\mathsf p$.
Hence, 
\begin{align}
& \sum_{l=0}^{K-1} \Big(|v_{l,n,\epsilon}(i) v^\top_{l,n,\epsilon}(j)-\phi_l(x_i)\phi_l(x_j)||e^{-\mu_{l,n,\epsilon} t}|\Big) \\
\leq &   K (2C_1 \lambda_K^{\frac{d-1}{4}} \Omega_2 \epsilon^{1/2} +\Omega^2_2 \epsilon)e^{- C_2t}+\tilde{c}_3 \epsilon. \nonumber 
\end{align}
If we sum the above terms together, we have 
\begin{align}
& |[\mathsf H^{(K)}_{\epsilon,t}]_{ij}-\mathsf{H}(x_i,x_j,t)| \nonumber \\
\leq & \frac{\tilde{c}_1 C^{\frac{d}{2}}_2}{K}e^{-\frac{C_2}{2} t}+K\left(\frac{2\Omega_1 C^2_1}{e C_2}K\lambda_K^{(d-1)/2} \epsilon^{\frac{3}{2}}  t e^{-\frac{C_2}{2}t} +2C_1 \lambda_K^{\frac{d-1}{4}} \Omega_2 \epsilon^{1/2}e^{- C_2t} +\Omega^2_2 \epsilon e^{- C_2t}\right)+\tilde{c}_3 \epsilon  \nonumber \\
\leq  & \left[ \frac{\tilde{c}_1 C^{d/2}_2}{K}+\frac{2\Omega_1 C^2_1}{e} C_2^{\frac{d-3}{2}} K^2 \left(\frac{\lambda_K}{C_2}\right)^{\frac{d-1}{2}} \epsilon^{\frac{3}{2}}+2C_1 C_2^{\frac{d-1}{4}} \Omega_2 K \left(\frac{\lambda_K}{C_2}\right)^{\frac{d-1}{4}}  \epsilon^{1/2} +\Omega^2_2 \epsilon \right]\nonumber\\
&\quad\times(t+1)e^{-\frac{C_2}{2}t} +\tilde{c}_3 \epsilon. \nonumber
\end{align}

By Lemma \ref{laplace eigenvalue lower bound}, $\frac{\lambda_K}{C_2} \geq 1$. Hence, if $\epsilon<\frac{C_2^{(d-1)/2}}{K^4 \lambda_K^{(d-1)/2}}= \frac{\mathcal{K}}{K^4 \lambda_K^{(d-1)/2}}$, where $\mathcal{K}$  depends on $d$, the Ricci curvature and the diameter of $M$, $K^2 (\frac{\lambda_K}{C_2})^{(d-1)/2} \epsilon^{\frac{3}{2}}$, $ K (\frac{\lambda_K}{C_2})^{\frac{d-1}{4}}  \epsilon^{1/2}$, and $\epsilon$  are all bounded by $\frac{1}{K}$. In conclusion, we have 
\begin{align}
& |[\mathsf H^{(K)}_{\epsilon,t}]_{ij}-\mathsf{H}(x_i,x_j,t)| \nonumber \\
\leq & \left(\tilde{c}_1 C^{\frac{d}{2}}_2+\frac{2\Omega_1 C^2_1}{e} C_2^{\frac{d-3}{2}} +2C_1 C_2^{\frac{d-1}{4}}\Omega_2  +\Omega^2_2 \right) \frac{1}{K} (t+1)e^{-\frac{C_2}{2}t} +\tilde{c}_3 \epsilon \nonumber \\
= & \frac{\Omega_3}{K}(t+1)e^{-\Omega_4 t} +\Omega_5 \epsilon, \nonumber
\end{align}
where $\Omega_3:=\tilde{c}_1 C^{\frac{d}{2}}_2+\frac{2\Omega_1 C^2_1}{e} C_2^{\frac{d-3}{2}} +2C_1 C_2^{\frac{d-1}{4}}\Omega_2  +\Omega^2_2$. Observe that $\Omega_3$ depends on $d$, $\mathsf p_m$, the $C^2$ norm of $\mathsf p$, the injectivity radius, the diameter, the volume and the curvature of $M$.  $\Omega_4=\frac{C_2}{2}$ depends on $d$, the Ricci curvature and the diameter of $M$. $\Omega_5=\tilde{c}_3$ depends on the volume of $M$, and the $C^2$ norm of $\mathsf p$. 
}

\section{Discussion}

The heat kernel approximation we propose in this paper relies on the series representation of the heat kernel in terms of the eigenpairs of the Laplace-Beltrami operator.  Since the {kernel normalized} GL associated with a point cloud sampled from a manifold can recover the Laplace-Beltrami operator regardless of the distribution of the data points on the manifold, we propose to approximate the heat kernel based on the eigenpairs of the {kernel normalized} GL. To this end, we study the spectral convergence of the eigenvalues and eigenvectors of the {kernel normalized} GL to the eigenvalues and eigenfunctions of the Laplace-Beltrami operator, and provide the convergence rate. With this result, the proposed approximation to the heat kernel is quantified with {a} convergence rate. 

As an alternative to the GL, one can potentially leverage on other algorithms that produce estimates of the Laplace-Beltrami operator in order to obtain different approximations to the heat kernel. As in practice the data can have arbitrary non-uniform distributions over unknown manifolds, it is important for the Laplace-Beltrami operator to be efficiently and consistently estimated regardless of the distribution of the data. It {has been} shown that the well-known unsupervised learning algorithm, Locally Linear Embeddings (LLEs), provides another choice in this regard. Indeed, it has been shown in \cite{wu2018think} that if the regularization is properly chosen, the LLE asymptotically approximates the Laplace-Beltrami operator even if the sampling scheme is non-uniform. In other words, the {KDE} step commonly applied in machine learning to handle the non-uniform sampling is {implicitly} carried out in the LLE.

The methods applied in this work serve for the convergence rate of the eigenvalues and the $L^\infty$ convergence rate of the eigenvectors of the {kernel normalized} GL to the eigenvalues and eigenfunctions of the Laplace-Beltrami operator, {even if} the sampling on the manifold is non-uniform. In order to relate the discrete eigenvectors to the continuous eigenfunctions, an interpolation is necessary (e.g. we introduce  the operator $T_{n,\epsilon}$ in the definition \ref{intro of operators}). Since we do not know any relationship between the position of the sampling and the geometry of {the} manifold, the analysis of the interpolation in Lemma \ref{fix u , T n espilon-T epsilon} is based on the generic case. This step seems to be the main bottleneck {of improving the convergence rate} in our analysis. While we provide a convergence rate, it does not fully match the numerical findings in practice. In \cite{trillos2018error,calder2019improved}, the authors claim a sharper eigenvalue convergence rate and an eigenvector convergence rate in the $L^2$ sense when the $1$-normalization step (kernel density estimation) is not applied. Besides an independent application of a similar eigenvalue control idea in Lemma \ref{ratio of eigenvalues} to improve the estimate, the techniques applied in \cite{calder2019improved} are completely different from ours. Since the asymptotical operator considered in \cite{calder2019improved} is not the Laplace-Beltrami operator when the sampling is non-uniform, and obtaining the optimal convergence rate is out of the scope of this paper, we refer readers with interest to \cite{calder2019improved} for details. We will explore whether their approaches can be developed to obtain a sharper $L^\infty$ convergence rate in our future work.

There are many other interesting directions motivated by this work. The most important direction is its practical application in statistics.  For example, we can develop an algorithm to handle the Gaussian process regression problem on an unknown embedded manifold in a Euclidean space by using the heat kernel of the manifold as the covariance kernel in the Gaussian process. It is not straightforward to define a valid covariance kernel for the Gaussian process regression when the predictors are on a {nonlinear} manifold \cite{castillo2014thomas,lin2018extrinsic}{, which might have non-trivial topology}. Although the heat kernel provides a natural and canonical choice theoretically, it is analytically intractable to directly evaluate, and current Bayesian methods rely on computationally demanding Monte Carlo approximations \cite{niu2019intrinsic}. With an efficient and accurate approximation of the heat kernel relying on the diffusion property of the GL, we may have a diffusion-based Gaussian process modeling framework for the regression purpose. This development has immediate applications to various fields, including biomedical signals and ecological datasets. Next, it is interesting to develop broader classes of covariance kernels on unknown manifolds beyond the heat kernel; for example, including anisotropy and non-stationarity inspired by the background knowledge, so that the regression can be carried out in a more adaptive way. Additional important directions include approximating the heat kernel under when there exist unbounded measurement errors and more intricate structure in the data, such as disconnected manifolds. Indeed, while we provide a preliminary analysis of the error-in-variable result, the analysis might be further improved with more suitable analysis tools. Another natural question to ask is if we could achieve a similar result for the connection Laplacian associated with a principal bundle via the graph connection Laplacian \cite{singer2012vector,singer2016spectral}, so that more topological information could be obtained. 

\section*{Acknowledgement}

Hau-Tieng Wu thanks the valuable discussion with Professor Jeff Calder and Professor Nicolas Garcia Trillos. {The authors would like to thank the anonymous reviewers for their constructive comments that helped improve this paper.}
 
\bibliographystyle{plain}
\bibliography{heatkernel}

\appendix

{
\section{Proof of a technical lemma}
\begin{lemma}\label{L^2 regularity}
Suppose $(\lambda_{k}, \phi_{k})$ is the $k$-th eigenpair of $-\Delta$ with $\|\phi_{k}\|_2=1$ and $(\lambda_{k,\epsilon}, \phi_{k,\epsilon})$ the $k$-th eigenpair of $\frac{I-T_\epsilon}{\epsilon^2}$ with $\|\phi_{k,\epsilon}\|_2=1$, {where we assume $\epsilon>0$} is small enough.  {Take $K\geq k$ so that} $\lambda_k \leq \lambda_K$ and {assume that $\lambda_K$ and $\epsilon$} satisfy (6'') in Table \ref{Table:Relations}. {Then, when $\epsilon>0$ is sufficiently small,} 
$$\left\|\frac{T_\epsilon-I}{\epsilon^2}\phi_{k,\epsilon}-\Delta \phi_{k,\epsilon}\right\|_2  \leq  2 \left(1+C_1(\lambda_K)+(C'_6+2\lambda_K) C_2(\lambda_K)+4\lambda_k^2\right) \epsilon^2\,,$$
{where $C_1(\lambda_K)>0$ and $C_2(\lambda_K)>0$ are constants depending on $\lambda_K$ but independent of $\epsilon$ and $C'_6>0$} depends on $d$ and the volume, the curvature and the diameter of $M$.
\end{lemma}

\begin{proof}
{First of all, we have the $L^2$ convergence $|\lambda_{k}-\lambda_{k,\epsilon}| \leq C_1(\lambda_K) \epsilon^2$ and $\|\phi_{k}- a \phi_{k,\epsilon}\|_2 \leq C_2(\lambda_K) \epsilon^2$, where $a \in \{1,-1\}$, according to \eqref{L^2 convergence of eigenvalues} and \eqref{L^2 convergence of eigenfunctions}.}
Recall that $\mathsf H_{\epsilon}$ is the integral operator associated with the heat kernel on the manifold $M$ with the diffusion time $\epsilon^2$, in other words, we have
\begin{equation}
\mathsf H_{\epsilon}f(x):=\int_M \mathsf{H}(x,y,\epsilon^2)f(y)dV_M(y)\,.
\end{equation} 
{Thus, when $\epsilon>0$ is sufficiently small, a trivial bound leads to}
\begin{align}
&\left\|\frac{T_\epsilon-I}{\epsilon^2}\phi_{k,\epsilon}-\Delta \phi_{k,\epsilon}\right\|_2=\|-\lambda_{k, \epsilon} \phi_{k,\epsilon}-\Delta \phi_{k,\epsilon}\|_2\nonumber \\
=&\,\left\|-\lambda_{k, \epsilon}\phi_{k,\epsilon}-\frac{\Delta \mathsf  H_{\epsilon}\phi_{k,\epsilon}}{1-\epsilon^2\lambda_{k,\epsilon}}+\frac{\Delta \mathsf  H_{\epsilon}\phi_{k,\epsilon}}{1-\epsilon^2\lambda_{k,\epsilon}}-\frac{\Delta T_{\epsilon}\phi_{k,\epsilon}}{1-\epsilon^2\lambda_{k,\epsilon}} \right\|_2 \label{Appendix A eq 1} \\
\leq &\,\left\|-\lambda_{k, \epsilon}\phi_{k,\epsilon}-\frac{\Delta \mathsf  H_{\epsilon}\phi_{k,\epsilon}}{1-\epsilon^2\lambda_{k,\epsilon}}\right\|_2+\left\|\frac{\Delta T_{\epsilon}\phi_{k,\epsilon}}{1-\epsilon^2\lambda_{k,\epsilon}} -\frac{\Delta \mathsf  H_{\epsilon}\phi_{k,\epsilon}}{1-\epsilon^2\lambda_{k,\epsilon}}\right\|_2 \nonumber \\
\leq &\, 2 \left\|-\lambda_{k, \epsilon}(1-\epsilon^2\lambda_{k,\epsilon})\phi_{k,\epsilon}-\Delta \mathsf  H_{\epsilon}\phi_{k,\epsilon}\right\|_2+2\left\|\Delta T_{\epsilon}\phi_{k,\epsilon} -\Delta \mathsf  H_{\epsilon}\phi_{k,\epsilon}\right\|_2 \nonumber
\end{align}
We first bound the term $\|\Delta T_{\epsilon}\phi_{k,\epsilon} -\Delta \mathsf  H_{\epsilon}\phi_{k,\epsilon}\|_2$ in \eqref{Appendix A eq 1}.
{Recall that} $\{\phi_i\}_{i=0}^\infty$ form an orthonormal basis of $L^2(M)$. We thus have 
$$\phi_{k,\epsilon}=\sum_{i=0}^\infty b_i \phi_i=:b_k\phi_k+f\,,
$$ 
{where since} $\|\phi_{k}- \phi_{k,\epsilon}\|_2 \leq C_2(\lambda_K) \epsilon^2$, we have $|b_k| \leq 1$ and$\|f\|_2 \leq C_2(\lambda_K) \epsilon^2$ by {the} triangle inequality. {Therefore,}
\begin{align}\label{Appendix A eq 2}
\|\Delta T_{\epsilon}\phi_{k,\epsilon} -\Delta \mathsf  H_{\epsilon}\phi_{k,\epsilon}\|_2 \leq \|\Delta T_{\epsilon}\phi_{k} -\Delta \mathsf  H_{\epsilon}\phi_{k}\|_2+\|\Delta T_{\epsilon}f -\Delta \mathsf  H_{\epsilon}f\|_2.
\end{align}
By Lemma \ref{property 2 of R epsilon} and integration by parts, 
\begin{align}\label{Appendix A eq 3}
\|\Delta T_{\epsilon}\phi_{k} -\Delta \mathsf  H_{\epsilon}\phi_{k}\|_2= \lambda_k \| T_{\epsilon}\phi_{k} -\mathsf  H_{\epsilon}\phi_{k}\|_2\leq  C_4 \epsilon^4  \lambda_k \big(1+\lambda_k^{d/2+5}\big)\,.
\end{align}
{For} the term $\|\Delta T_{\epsilon}f -\Delta \mathsf  H_{\epsilon}f\|_2$, {we need a further exploration of $T_\epsilon$ and hence a comparison with $\mathsf  H_{\epsilon}$.} 
Recall that $k_{\epsilon}(x,y)=\exp\Big(-\frac{\|\iota(x)-\iota(y)\|^2_{\mathbb{R}^D}}{4\epsilon^2}\Big)$, and
based on the definition of $T_{\epsilon}$, we have 
\begin{align}
T_{\epsilon}f(x)=& \frac{\int_M \frac{k_{\epsilon}(x,y)}{d_{\epsilon}(x)d_{\epsilon}(y) } f(y) \mathsf p(y) dV_M(y)}{\int_M \frac{k_{\epsilon}(x,y)}{d_{\epsilon}(x)d_{\epsilon}(y) } \mathsf p(y) dV_M(y)} = \frac{\int_M \frac{\mathsf p(y)}{d_{\epsilon}(y) }k_{\epsilon}(x,y) f(y)  dV_M(y)}{\int_M \frac{\mathsf p(y) }{d_{\epsilon}(y) }k_{\epsilon}(x,y) dV_M(y)} \\
=& \int_M \mathbb{K}_{\epsilon}(x,y) f(y)  dV_M(y), \nonumber
\end{align}
where
$$
\mathbb{K}_{\epsilon}(x,y):=\frac{ \frac{\mathsf p(y)}{d_{\epsilon}(y) }k_{\epsilon}(x,y)}{\int_M \frac{\mathsf p(y) }{d_{\epsilon}(y) }k_{\epsilon}(x,y) dV_M(y)}\,.
$$
{We now compare $\mathbb{K}_{\epsilon}(x,y)$ with $\mathsf{H}(x,y,t)$ when $t=\epsilon^2$.}
By Lemma \ref{bounds on heat kernel}, we have the following facts about $\mathsf{H}(x,y,t)$ and its space and time derivatives. 
{There is} $\gamma>0$ and $B_\gamma(x)$ is the geodesic ball in $M$ around $x$ {with the radius $\gamma$}, {so that when} $\epsilon>0$ {is} small enough, we have
\begin{align}
&\mathsf{H}(x,y, 2\epsilon^2) \geq C'_1 \epsilon^{-d} e^{-\frac{d(x,y)^2}{8 \epsilon^2}} & \quad  \hbox{if $y \in B_{\gamma}(x)$,} \\
&|\nabla_x \mathsf{H}(x,y,\epsilon^2)| \leq C_H \epsilon^{-d-1} e^{-\frac{d(x,y)^2}{8 \epsilon^2}}  &\quad \hbox{if $y \in B_{\gamma}(x)$,} \\
&|\Delta_x \mathsf{H}(x,y,\epsilon^2)|=|\partial_t \mathsf{H}(x,y,\epsilon^2)| \leq C_H \epsilon^{-d-2} e^{-\frac{d(x,y)^2}{8\epsilon^2}} & \quad  \hbox{if $y \in B_{\gamma}(x)$,} \\
&|\Delta_x \mathsf{H}(x,y,\epsilon^2)|=|\partial_t \mathsf{H}(x,y,\epsilon^2)| \leq C_H \epsilon^2  &\quad \hbox{if $y \in B^c_{\gamma}(x)$\,,} \label{appendix A bound 0}
\end{align} 
{where} $C'_1{>0}$ depends on $d$ and the Ricci curvature of $M$ {and} $C_H{>0}$ depends on $d$, the volume, the Ricci curvature and the diameter of $M$. In other words, 
\begin{align}
&|\nabla_x \mathsf{H}(x,y,\epsilon^2)| \leq \frac{C_H}{C'_1 \epsilon} \mathsf{H}(x,y, 2\epsilon^2)  &\quad \hbox{if $y \in B_{\gamma}(x)$,} \label{appendix A bound 1}\\
&|\Delta_x \mathsf{H}(x,y,\epsilon^2)| \leq \frac{C_H}{C'_1 \epsilon^2} \mathsf{H}(x,y, 2\epsilon^2)& \quad  \hbox{if $y \in B_{\gamma}(x)$.}\label{appendix A bound 2} \\
&|\Delta_x \mathsf{H}(x,y,\epsilon^2)|d(x,y)^2 \leq \frac{8 C_H}{C'_1} \mathsf{H}(x,y, 4\epsilon^2)& \quad  \hbox{if $y \in B_{\gamma}(x)$.}\label{appendix A bound 2.5} 
\end{align} 
By a straightforward expansion of $\mathbb{K}_{\epsilon}(x,y)$ as in Lemma 7 and Lemma 8 in \cite{coifman2006diffusion} and by Lemma \ref{bounds on heat kernel}, we have the following facts. {There is} $\gamma>0$ {so that when} $\epsilon>0$ {is} small enough, we have
\begin{align}
&\mathbb{K}_{\epsilon}(x,y)=\mathsf{H}(x,y,\epsilon^2)(w_1(x,y)+w_2(x,y,\epsilon^2)\epsilon^2)  & \quad  \hbox{if $y \in B_{\gamma}(x)$,} \label{appendix A bound 3}\\
&|\Delta_x \mathbb{K}_{\epsilon}(x,y)|\leq C'_2 \epsilon^2  &\quad \hbox{if $y \in B^c_{\gamma}(x)$\,,} \label{appendix A bound 6}
\end{align}
{where $C'_2>0$ depends on $d$ and the curvature of the manifold and $w$ is a smooth function of $x,y$ satisfying}
\begin{align}
& |w_1(x,y)-1| \leq C'_3 d(x,y)^2  & \quad\hbox{for all $x$, $y \in B_{\gamma}(x)$,}  \label{appendix A bound 4.5} \\
& |w_2(x,y,\epsilon^2)| \leq C'_4 & \quad\hbox{for all $x$, $y \in B_{\gamma}(x)$,} \label{appendix A bound 4}\\
&\|w_2(x,y,\epsilon^2)\|_{C^2(B_{\gamma}(x))} \leq C'_5 &\quad\hbox{for all $x$, $y \in B_{\gamma}(x)$,} \label{appendix A bound 5}
\end{align}
{where $C'_3, C'_4>0$ depend} on $d$ and the curvature of the manifold {and} $C'_5>0$ depends on $d$ and the second order covariant derivatives of the curvature tensor. In order to bound $\|\Delta T_{\epsilon}f -\Delta \mathsf  H_{\epsilon}f\|_2$, we first bound $|\Delta T_{\epsilon}f(x) -\Delta \mathsf  H_{\epsilon}f(x)|$. {By a straightforward bound, }
\begin{align}
& \left|\int_{M}\Delta_x \mathbb{K}_{\epsilon}(x,y)f(y) dV_M(y)- \int_{M} \Delta_x \mathsf{H}(x,y,\epsilon^2)f(y) dV_M(y)\right| \nonumber\\
\leq & \left|\int_{ B_{\gamma}(x)}\big[ \Delta_x \mathbb{K}_{\epsilon}(x,y)-\Delta_x \mathsf{H}(x,y,\epsilon^2) \big]f(y) dV_M(y)\right| +\int_{ B^c_{\gamma}(x)}|\Delta_x \mathbb{K}_{\epsilon}(x,y)| |f(y)| dV_M(y)  \nonumber \\
  &+\left|\int_{ B^c_{\gamma}(x)} \Delta_x \mathsf{H}(x,y,\epsilon^2) f(y) dV_M(y) \right|\,, \nonumber 
  \end{align}
{where 
\begin{align}
&\int_{ B^c_{\gamma}(x)}|\Delta_x \mathbb{K}_{\epsilon}(x,y)| |f(y)| dV_M(y)  +\left|\int_{ B^c_{\gamma}(x)} \Delta_x \mathsf{H}(x,y,\epsilon^2) f(y) dV_M(y) \right|\nonumber\\ 
\leq &\,(C'_2+C_H) \sqrt{\texttt{Vol}(M)} \epsilon^2 \|f\|_2\leq \sqrt{\texttt{Vol}(M)} \|f\|_2\nonumber
\end{align} 
when $\epsilon$ is sufficiently small due to \eqref{appendix A bound 0} and \eqref{appendix A bound 6}. It remains to bound the main term over $B_{\gamma}(x)$. {By \eqref{appendix A bound 3}, we have 
\begin{align}
&\left|\int_{ B_{\gamma}(x)}\big[ \Delta_x \mathbb{K}_{\epsilon}(x,y)-\Delta_x \mathsf{H}(x,y,\epsilon^2) \big]f(y) dV_M(y)\right|  \nonumber \\
=& \left| \int_{ B_{\gamma}(x)} \Delta_x\mathsf{H}(x,y,\epsilon^2)\Big[w_1(x,y)-1\Big]f(y) dV_M(y)\right|+\left| \int_{ B_{\gamma}(x)} \Delta_x\Big[ \mathsf{H}(x,y,\epsilon^2)w_2(x,y,\epsilon^2)\epsilon^2 \Big] f(y) dV_M(y)\right|\,. \nonumber
\end{align}
For the first term on the right hand side,
\begin{align}
& \left| \int_{ B_{\gamma}(x)} \Delta_x\mathsf{H}(x,y,\epsilon^2)\Big[w_1(x,y)-1\Big]f(y) dV_M(y)\right| \nonumber \\
\leq &  \int_{ B_{\gamma}(x)} |\Delta_x\mathsf{H}(x,y,\epsilon^2)| |w_1(x,y)-1| |f(y)| dV_M(y) \leq  \int_{ B_{\gamma}(x)}  C'_3|\Delta_x\mathsf{H}(x,y,\epsilon^2)| d(x,y)^2 |f(y)| dV_M(y) \nonumber \\
\leq & \int_{ B_{\gamma}(x)}  C'_3|\Delta_x\mathsf{H}(x,y,\epsilon^2)| d(x,y)^2 |f(y)| dV_M(y) \leq  \int_{ B_{\gamma}(x)}  \frac{8 C_HC'_3}{C'_1} \mathsf{H}(x,y, 4\epsilon^2) |f(y)| dV_M(y) \nonumber \\
\leq & \int_{M}  \frac{8 C_HC'_3}{C'_1} \mathsf{H}(x,y, 4\epsilon^2) |f(y)| dV_M(y). \nonumber
\end{align}
Note that we use \eqref{appendix A bound 4.5}  and \eqref{appendix A bound 2.5}  in the third and fourth step. }
Again by a direct expansion, $\left| \int_{ B_{\gamma}(x)} \Delta_x\Big[ \mathsf{H}(x,y,\epsilon^2)w_2(x,y,\epsilon^2)\epsilon^2 \Big] f(y) dV_M(y)\right|$ is bounded by}
\begin{align}
 & \quad\epsilon^2 \int_{ B_{\gamma}(x)} |\Delta_x \mathsf{H}(x,y,\epsilon^2)| |w_2(x,y,\epsilon^2)| |f(y)| dV_M(y) \nonumber\\
 &+2 \epsilon^2 \int_{ B_{\gamma}(x)} |\nabla_x \mathsf{H}(x,y,\epsilon^2)| |\nabla_x w_2(x,y,\epsilon^2)| |f(y)| dV_M(y) \nonumber \\
& +\epsilon^2 \int_{ B_{\gamma}(x)} \mathsf{H}(x,y,\epsilon^2) |\Delta_x w_2(x,y,\epsilon^2)| |f(y)| dV_M(y)\nonumber \\
\leq & \frac{C_HC'_4}{C'_1} \int_{ B_{\gamma}(x)}  \mathsf{H}(x,y, 2\epsilon^2) |f(y)| dV_M(y) +2 \epsilon \frac{C_H  C'_5}{C'_1} \int_{ B_{\gamma}(x)} \mathsf{H}(x,y, 2\epsilon^2) |f(y)| dV_M(y)\nonumber \\
& +\epsilon^2  C'_5 \int_{ B_{\gamma}(x)} \mathsf{H}(x,y,\epsilon^2)  |f(y)| dV_M(y)\nonumber \\
\leq & \left(\frac{C_HC'_4}{C'_1}+1\right) \int_{ B_{\gamma}(x)}  \mathsf{H}(x,y, 2\epsilon^2) |f(y)| dV_M(y) + \int_{ B_{\gamma}(x)} \mathsf{H}(x,y,\epsilon^2)  |f(y)| dV_M(y). \nonumber
\end{align}
Note that we use \eqref{appendix A bound 1}, \eqref{appendix A bound 2}, \eqref{appendix A bound 4}  and \eqref{appendix A bound 5} in the {first bound, and the second bound holds when $\epsilon$ is sufficiently small}. 
Denote 
\[
u(x, t)=\int_{M} \mathsf{H}(x,y, t)  |f(y)| dV_M(y)\,. 
\]
{Obviously we have} 
\[
\int_{ B_{\gamma}(x)}  \mathsf{H}(x,y, 2\epsilon^2) |f(y)| dV_M(y) \leq u(x, 2\epsilon^2)\quad\mbox{and}\quad\int_{ B_{\gamma}(x)} \mathsf{H}(x,y,\epsilon^2)  |f(y)| dV_M(y) \leq u(x,\epsilon^2).
\]
{Hence, we have the bound}
\begin{align}
& |\int_{M}\Delta_x \mathbb{K}_{\epsilon}(x,y)f(y) dV_M(y)- \int_{M} \Delta_x \mathsf{H}(x,y,\epsilon^2)f(y) dV_M(y)|\nonumber  \\
\leq &  \frac{8 C_HC'_3}{C'_1}u(x, 4\epsilon^2) +\left(\frac{C_HC'_4}{C'_1}+1\right) u(x, 2\epsilon^2)+u(x,\epsilon^2)+\sqrt{\texttt{Vol}(M)} \|f\|_2. \nonumber
\end{align}
{Finally, since} $u(x, t)$ is the solution to the heat equation with initial condition $|f|$, {we have } $\|u(x,t)\|_2 \leq \|f\|_2$  for any $t \geq 0$. Hence, we {have shown} that
\begin{align}\label{Appendix A eq 4}
\|\Delta T_{\epsilon}f -\Delta \mathsf  H_{\epsilon}f\|_2 \leq \left(\frac{8 C_HC'_3}{C'_1}+\frac{C_HC'_4}{C'_1}+2+\texttt{Vol}(M)\right) \|f\|_2.
\end{align}
Substitute \eqref{Appendix A eq 3} and \eqref{Appendix A eq 4} into \eqref{Appendix A eq 2}, we have
\begin{align}\label{Appendix A eq 5}
\|\Delta T_{\epsilon}\phi_{k,\epsilon} -\Delta \mathsf  H_{\epsilon}\phi_{k,\epsilon}\|_2 \leq &  C_4 \epsilon^4  \lambda_k \left(1+\lambda_k^{d/2+5}\right)+\left(\frac{8 C_HC'_3}{C'_1}+\frac{C_HC'_4}{C'_1}+2+\texttt{Vol}(M)\right)\|f\|_2 \nonumber \\
\leq &\left (1+\left(\frac{8 C_HC'_3}{C'_1}+\frac{C_HC'_4}{C'_1}+2+\texttt{Vol}(M)\right) C_2(\lambda_K)\right) \epsilon^2\,,
\end{align}
where  we use relation (6'') in table \ref{Table:Relations} in the last step.

{Next}, we bound $\|-\lambda_{k,\epsilon}(1-\epsilon^2\lambda_{k,\epsilon})\phi_{k,\epsilon}-\Delta \mathsf  H_{\epsilon}\phi_{k,\epsilon}\|_2$ in \eqref{Appendix A eq 1}. Recall that $\phi_{k,\epsilon}=\sum_{i=0}^\infty b_i \phi_i=b_k\phi_k+f$. Hence,
\begin{align}
& \|-\lambda_{k,\epsilon}(1-\epsilon^2\lambda_{k,\epsilon})\phi_{k,\epsilon}-\Delta \mathsf  H_{\epsilon}\phi_{k,\epsilon}\|_2 \nonumber\\
\leq & \|-\lambda_k b_k\phi_{k}-\Delta \mathsf  H_{\epsilon}\phi_{k,\epsilon}\|_2 +\epsilon^2(\lambda^2_{k,\epsilon}+C_1(\lambda_K))+\lambda_{k,\epsilon}\|f\|_2 \,.\nonumber 
\end{align}
By using the similar method as above, we can show that $\|-\lambda_k b_k\phi_{k}-\Delta \mathsf  H_{\epsilon}\phi_{k,\epsilon}\|_2 \leq 2\epsilon^2\lambda^2_{k}+C_2(\lambda_K) \epsilon^2$. Hence,
\begin{align}\label{Appendix A eq 6}
& \|-\lambda_k(1-\epsilon^2\lambda_{k,\epsilon})\phi_{k,\epsilon}-\Delta \mathsf  H_{\epsilon}\phi_{k,\epsilon}\|_2 \\
\leq &2\epsilon^2\lambda^2_{k}+C_2(\lambda_K) \epsilon^2+\epsilon^2(\lambda^2_{k,\epsilon}+C_1(\lambda_K))+\lambda_{k,\epsilon}\|f\|_2 \nonumber\\
\leq & ( C_2(\lambda_K)+4\lambda_k^2+2\lambda_kC_2(\lambda_K)+C_1(\lambda_K))\epsilon^2. \nonumber
\end{align}
{Finally, if} we substitute \eqref{Appendix A eq 5} and \eqref{Appendix A eq 6} into \eqref{Appendix A eq 1}, we have 
\begin{align}
&\left\|\frac{T_\epsilon-I}{\epsilon^2}\phi_{k,\epsilon}-\Delta \phi_{k,\epsilon}\right\|_2\nonumber \\
\leq & 2 \left(1+C_1(\lambda_K)+\left(\left(\frac{8 C_HC'_3}{C'_1}+\frac{C_HC'_4}{C'_1}+2+\texttt{Vol}(M)\right)+1+2\lambda_k\right) C_2(\lambda_K)+4\lambda_k^2\right) \epsilon^2\,. \nonumber
\end{align}
The statement of the lemma follows.
\end{proof}

\section{Proof of Proposition \ref{GC class}}

Let $\mathcal{F} \subset C(M)$. The covering number $\mathcal{N}(\mathcal{F}, \gamma, \|\cdot\|_N)$ is the smallest number of balls with radius $\gamma$ in the norm $\|\cdot\|_N$ that cover $\mathcal{F}$. In this proof, we consider {three spaces,} $\mathcal K_{\epsilon}$, $u\mathcal{Q}_{\epsilon}$ and $\mathcal{Q}_{\epsilon}\mathcal{Q}_{\epsilon}$.  The norms that we consider are $\|\cdot\|_N=\|\cdot\|_\infty$ and $\|\cdot\|_N=\|\cdot\|_{L^2(P_n)}$. Observe that the cover number satisfies the basic property: if $\gamma_1 \leq \gamma_2$, then 
$$\mathcal{N}(\mathcal{F}, \gamma_1, \|\cdot\|_N) \geq \mathcal{N}(\mathcal{F}, \gamma_2, \|\cdot\|_N).$$

Recall that for the point cloud $\mathcal{X}:=\{x_i\}_{i=1}^n$ i.i.d. sampled from the random vector with the density function $\mathsf p$ supported on the manifold $M$, {for} $f \in C(M)$, {we have}
$$\|f\|_{L^2(P_n)}=\sqrt{\frac{1}{n}\sum_{i=1}^n f(x_i)^2}.$$
{Clearly} we have for $f \in C(M)$, $\|f\|_{L^2(P_n)} \leq \|f\|_{\infty}$. Hence, 
\begin{align}
\mathcal{N}(\mathcal{F}, \gamma, \|\cdot\|_{L^2(P_n)}) \leq  \mathcal{N}(\mathcal{F}, \gamma, \|\cdot\|_{\infty}).
\end{align}

In the next Lemma, we bound the covering number $\mathcal{N}(\mathcal K_{\epsilon}, \gamma, \|\cdot\|_{\infty})$.  Then, we will use $\mathcal{N}(\mathcal K_{\epsilon}, \gamma, \|\cdot\|_{\infty})$ to bound the corresponding covering number for the spaces $\mathcal K_{\epsilon}$, $u\mathcal{Q}_{\epsilon}$ and $\mathcal{Q}_{\epsilon}\mathcal{Q}_{\epsilon}$.

\begin{lemma}\label{covering number K epsilon}
{For $\gamma>0$ and $\epsilon>0$ sufficiently small, we have}
$$
\mathcal{N}(\mathcal K_{\epsilon}, \gamma, \|\cdot\|_{\infty}) \leq \left(\frac{4  \texttt{diam}(M)}{\epsilon^2 \gamma}\right)^{\frac{d(3d+11)}{2}}$$
\end{lemma}
\begin{proof}
{Recall that $M$ is isometrically embedded in $\mathbb{R}^D$ through $\iota$. By the Nash embedding theorem \cite{nash1954c1}, there is an isometric embedding $E:M \rightarrow \mathbb{R}^{s'}$, where $s'\leq s:=\frac{d(3d+11)}{2}\leq D$. We claim that, for any $\epsilon_1>0$,  the covering number $\mathcal{N}(E(M), \epsilon_1, \|\cdot\|_{\mathbb{R}^{s}})$ is bounded by $\left(\frac{4\texttt{diam}(M)}{\epsilon_1}\right)^{s}$. 
Choose any point $p \in E(M)$. Clearly, $E(M)$ is contained in the $s$ dimensional cube in $\mathbb{R}^s$ centered at $p$ with the side length $2\texttt{diam}(M)$. 
The cube can be covered by $(\frac{4 \texttt{diam}(M)}{\epsilon_1})^{s}$ Euclidean balls of radius $\frac{\epsilon_1}{2}$, denoted as $\{B^{\mathbb{R}^{s}}_{\epsilon_1/2}(z_i)\}$, where $z_i \in \mathbb{R}^{s}$. Note that $\{z_i\}$ may not be on $E(M)$. 

Next, we apply the following argument. If $B^{\mathbb{R}^{s}}_{\epsilon_1/2}(z_i)$ intersects $E(M)$, we could take $z'_i \in B^{\mathbb{R}^{s}}_{\epsilon_1/2}(z_i)\cap E(M)$, where $B^{\mathbb{R}^{s}}_{\epsilon_1/2}(z_i) \subset B^{\mathbb{R}^{s}}_{\epsilon_1}(z'_i)$. Otherwise, we let $z'_i=z_i$. Obviously, we have
$$
E(M) \subset \bigcup_{B^{\mathbb{R}^{s}}_{\epsilon_1/2}(z_i) \cap E(M) \neq\emptyset } B^{\mathbb{R}^{s}}_{\epsilon_1/2}(z_i) \subset \bigcup_{B^{\mathbb{R}^{s}}_{\epsilon_1/2}(z_i) \cap E(M) \neq\emptyset } B^{\mathbb{R}^{s}}_{\epsilon_1}(z'_i) \,.
$$
Then, we claim that the number of $z'_i$ in $E(M)$ such that $B_{\epsilon_1/2}(z_i) \cap E(M) \neq\emptyset$ is bounded by the total number of $\{z_i\}$, which is $(\frac{4 \texttt{diam}(M)}{\epsilon_1})^{s}$. Indeed, for any $x,y \in M$ so that $\|E(x)-E(y)\|_{\mathbb{R}^{s}}$ is small enough, it is well known that
$$
\|E(x)-E(y)\|_{\mathbb{R}^s} \leq d(x,y) \leq 2 \|E(x)-E(y)\|_{\mathbb{R}^{s}}.
$$ 
See, for example, \cite[Lemma B.3]{wu2018think}.
Hence, if $\epsilon_1$ is small enough and $\{B^{\mathbb{R}^{s}}_{\epsilon_1}(z'_i)\}$ with $z'_i \in E(M)$ is a cover of $E(M)$, then the above inequality implies that the balls $\{B^{\mathbb{R}^{s}}_{\epsilon_1}(w_i)\}$ with $w_i=\iota(E^{-1}(z'_i)) \in \mathbb{R}^{D}$ cover $\iota(M)$. Hence, when $\epsilon_1$ is small enough, the covering number $\mathcal{N}(\iota(M), \epsilon_1, \|\cdot\|_{\mathbb{R}^D}) \leq \left(\frac{8 \texttt{diam}(M)}{\epsilon_1}\right)^{s}$. 

For $\gamma>0$, fix $\epsilon$ small enough so that $\epsilon_1=2 \epsilon^2 \gamma$ is sufficiently small; in other words, $\gamma=\frac{\epsilon_1}{2 \epsilon^2}$.} For $x, x_i \in M$, if $\|E(x)-E(x_i)\|_{\mathbb{R}^s} \leq \epsilon_1$, by {the} mean value theorem, 
\begin{align}
\max_{y}|k_{\epsilon}(x,y)-k_{\epsilon}(x_i,y)| \leq \frac{\epsilon_1}{2 \epsilon^2}.
\end{align}
Hence, if the Euclidean balls centered at $\{w_i\} \subset E(M)$ of radius  $\epsilon_1=2\epsilon^2\gamma$  cover $E(M)$, then the $L^\infty$ balls centered at $k_{\epsilon}(\iota^{-1}(w_i), \cdot)$ of radius $\gamma$ can cover the space $K_{\epsilon}$.  We substitute $\epsilon_1=2 \epsilon^2 \gamma$ into the bound of $\mathcal{N}(\iota(M), \epsilon_1, \|\cdot\|_{\mathbb{R}^D}) $, then we prove the lemma. 
\end{proof}
 
\begin{lemma} \label{covering number Q epsilon}
Take $u\in C(M)$ so that $\|u\|_\infty \leq 1$. There is a constant $\tilde{C}_1>0$ depending on $\mathsf p_m$ and the $C^0$ norm of $\mathsf p$ such that if $\epsilon$ is small enough, for all $f \in \tilde{C}_1 \epsilon^{2d}u\mathcal{Q}_{\epsilon}$ and $f \in \tilde{C}_1 \epsilon^{4d} \mathcal{Q}_{\epsilon} \mathcal{Q}_{\epsilon}$, we have $\|f\|_\infty \leq 1$. Moreover, 
\begin{align}
& \mathcal{N}(\tilde{C}_1 \epsilon^{2d} u\mathcal{Q}_{\epsilon}, \gamma, \|\cdot\|_{\infty}) \leq \left(\frac{\tilde{C}_2}{\epsilon^{d+2} \gamma}\right)^{\frac{d(3d+11)}{2}}, \nonumber \\
& \mathcal{N}(\tilde{C}_1 \epsilon^{4d} \mathcal{Q}_{\epsilon}\mathcal{Q}_{\epsilon}, \gamma, \|\cdot\|_{\infty}) \leq \left(\frac{\tilde{C}_2}{\epsilon^{d+2} \gamma}\right)^{\frac{d(3d+11)}{2}}, \nonumber 
\end{align}
where $\tilde{C}_2>0$ is a constant depending on $d$, the diameter of $M$, $\mathsf p_m$ and the $C^0$ norm of $\mathsf p$.
\end{lemma}
\begin{proof}
First, it follows from (a) in Lemma \ref{properties of T epsilon} that if $\epsilon$ is small enough, { for any $x,y \in M$, $|Q_{\epsilon}(x,y)| \leq \frac{\tilde{C}}{\epsilon^{2d}}$, where $\tilde{C}$ depends on $\mathsf p_m$ and the $C^0$ norm of $\mathsf p$. Take $\tilde{C}_1=\min(\frac{1}{\tilde{C}}, \frac{1}{\tilde{C}^2})$, then for all $f \in \tilde{C}_1 \epsilon^{2d} u \mathcal{Q}_{\epsilon}$ and $f \in \tilde{C}_1 \epsilon^{4d} \mathcal{Q}_{\epsilon} \mathcal{Q}_{\epsilon}$, we have $\|f\|_\infty \leq 1$.}

\begin{align}
&| \tilde{C}_1 \epsilon^{2d}u(z)Q_\epsilon(x,z)- \tilde{C}_1 \epsilon^{2d} u(z)Q_\epsilon(y,z)| \\
\leq & \tilde{C}_1 \epsilon^{2d} |\frac{u(z)}{d_\epsilon(z)}|\Big(|k_\epsilon(x,z)| \frac{|d_\epsilon(x)-d_\epsilon(y)|}{|d_\epsilon(x)d_\epsilon(y)|}+\frac{1}{d_\epsilon(y)}|k_\epsilon(x,z)-k_\epsilon(y,z)|\Big). \nonumber
\end{align}
By Lemma \ref{bounds on operators}, for any $x \in M$,  $C_5 \epsilon^d \leq d_{\epsilon}(x) \leq C_6 \epsilon^d$, where $C_5$ and $C_6$ depend on $\mathsf p_m$ and the $C^0$ norm of $\mathsf p$. Moreover, 
\begin{align}
|d_\epsilon(x)-d_\epsilon(y)|= |\int_M (k_\epsilon(x,z')-k_\epsilon(y,z')dP(z')| \leq \|k_\epsilon(x,\cdot)-k_\epsilon(y,\cdot)\|_\infty. 
\end{align}
Hence,
\begin{align}
\| \tilde{C}_1 \epsilon^{2d} u(\cdot)Q_\epsilon(x,\cdot)-  \tilde{C}_1 \epsilon^{2d} u(\cdot)Q_\epsilon(y,\cdot)\|_\infty \leq \frac{2 \tilde{C}_1}{C_5^3\epsilon^{d}}\|k_\epsilon(x,\cdot)-k_\epsilon(y,\cdot)\|_\infty.
\end{align}
Since $\mathcal{N}(K_{\epsilon}, \gamma, \|\cdot\|_{\infty}) \leq (\frac{4 diam(M)}{\epsilon^2 \gamma})^{\frac{d(3d+11)}{2}}$, we have $\mathcal{N}(u\mathcal{Q}_{\epsilon}, \gamma, \|\cdot\|_{\infty}) \leq (\frac{8 \tilde{C}_1 diam(M)}{C_5^3\epsilon^{d+2} \gamma})^{\frac{d(3d+11)}{2}}$. Note that $\frac{8 \tilde{C}_1 diam(M)}{C_5^3}$ depends on $d$, the diameter of $M$, $\mathsf p_m$ and the $C^0$ norm of $\mathsf p$.
 Similarly, we can prove the bound for $\mathcal{N}( \tilde{C}_1 \epsilon^{4d}\mathcal{Q}_{\epsilon}\mathcal{Q}_{\epsilon}, \gamma, \|\cdot\|_{\infty})$.
\end{proof}

Now we are ready to prove Proposition \ref{GC class}.

\begin{proof} (Proof of Proposition \ref{GC class})

Let $\mathcal{F} \subset C(M)$ and $\|f\|_\infty \leq 1$ for all $f \in \mathcal{F}$.  Based on the entropy bound in Proposition 19 in \cite{von2008consistency}, there is a constant $c_1>0$ such that for all $n$, with probability greater than $1-\delta$,
\begin{align}
\sup_{f \in \mathcal{F}}|P_nf-Pf| \leq \frac{c_1}{\sqrt{n}} \int_{0}^\infty \sqrt{\log \mathcal{N}(\mathcal{F}, \gamma, \|\cdot\|_{L^2(P_n)})} d\gamma+\sqrt{\frac{1}{2n}\log \frac{2}{\delta}}.
\end{align}
Since $\mathcal{N}(\mathcal{F}, \gamma, \|\cdot\|_{L^2(P_n)}) \leq  \mathcal{N}(\mathcal{F}, \gamma, \|\cdot\|_{\infty})$, 
\begin{align}
\sup_{f \in \mathcal{F}}|P_nf-Pf| \leq & \frac{c_1}{\sqrt{n}} \int_{0}^\infty \sqrt{\log \mathcal{N}(\mathcal{F}, \gamma, \|\cdot\|_{\infty})} d\gamma+\sqrt{\frac{1}{2n}\log \frac{2}{\delta}} \\
{=} & \frac{c_1}{\sqrt{n}} \int_{0}^2 \sqrt{\log \mathcal{N}(\mathcal{F}, \gamma, \|\cdot\|_{\infty})} d\gamma+\sqrt{\frac{1}{2n}\log \frac{2}{\delta}}, \nonumber
\end{align}
where we use the fact that the diameter of the set $\mathcal{F}$ is bounded by $2$ in the last step. Note that by Lemma \ref{covering number Q epsilon}, for any $f\in {\mathcal K_\epsilon \cup \tilde{C}_1 \epsilon^{2d} u\mathcal{Q}_{\epsilon}\cup \tilde{C}_1 \epsilon^{4d} \mathcal{Q}_{\epsilon} \mathcal{Q}_{\epsilon}}$, we have $\|f\|_\infty \leq 1$. 

We first discuss the case $\mathcal{F}=K_\epsilon$. By Lemma \ref{covering number K epsilon}, $\log \mathcal{N}(K_\epsilon, \gamma, \|\cdot\|_{\infty})=\frac{d(3d+11)}{2}\log (\frac{4 \texttt{diam}(M)}{\epsilon^2 \gamma})$. Let $\gamma^*=\frac{\epsilon^2}{4 \texttt{diam}(M)}$. If $\gamma \leq \gamma^*$, then 
\begin{align}
\log \mathcal{N}(\mathcal K_\epsilon, \gamma, \|\cdot\|_{\infty})=& \frac{d(3d+11)}{2}\log \left(\frac{4 \texttt{diam}(M)}{\epsilon^2 \gamma}\right) =  \frac{d(3d+11)}{2} \log \left(\frac{1}{\gamma^* \gamma}\right) \nonumber \\
\leq & d(3d+11) \log\left(\frac{1}{\gamma}\right) \leq d(3d+11) \left(\log\left(\frac{1}{\gamma}\right)\right)^2. \nonumber 
\end{align}
If $\gamma > \gamma^*$, 
{\begin{align}
&\log \mathcal{N}(\mathcal K_\epsilon, \gamma, \|\cdot\|_{\infty})  \leq  \log \mathcal{N}(\mathcal K_\epsilon, \gamma^*, \|\cdot\|_{\infty}) 
=2d(3d+11) \log\left(\frac{2 \sqrt{\texttt{diam}(M)}}{\epsilon}\right)\, .\nonumber
\end{align}
Hence, 
\begin{align}
& \int_{0}^2 \sqrt{\log \mathcal{N}(K_\epsilon, \gamma, \|\cdot\|_{\infty})} d\gamma\nonumber \\
\leq & \int_{0}^{\gamma^*} \sqrt{d(3d+11)} \log\left(\frac{1}{\gamma}\right) d\gamma+(2-\gamma^*) \sqrt{2d(3d+11)} \sqrt{ \log\left(\frac{2 \sqrt{\texttt{diam}(M)}}{\epsilon}\right)} \nonumber \\
\leq &   \sqrt{d(3d+11)} \gamma^* \log\left(\frac{1}{\gamma^*}\right) + 2 \sqrt{2d(3d+11)} \sqrt{ \log\left(\frac{2 \sqrt{\texttt{diam}(M)}}{\epsilon}\right)}\nonumber \\
\leq & \sqrt{d(3d+11)} + 2 \sqrt{2d(3d+11)} \sqrt{ \log\left(2 \sqrt{\texttt{diam}(M)}\right)}+ 2 \sqrt{2d(3d+11)} \sqrt{\log\left(\frac{1}{\epsilon}\right)} \nonumber \\
\leq & c_2 \sqrt{-\log\epsilon}, \nonumber 
\end{align}}
where we use {the trivial bound} $ \gamma^* \log(\frac{1}{\gamma^*}) <1$ {since} $\gamma^*$ {is of order $\epsilon^2$} and the fact that $\epsilon$ is small enough in the last step, and $c_2{>0}$ is a constant depending on $d$ and the diameter of $M$. Therefore, if we choose $\delta=\frac{1}{n^2}$,
\begin{align}
\sup_{f \in \mathcal K_\epsilon}|P_nf-Pf| \leq  \frac{c_1 c_2 \sqrt{-\log\epsilon}}{\sqrt{n}} +\sqrt{\frac{\log 2n^2}{2n}} \leq \frac{c_3}{\sqrt{n}}\left( \sqrt{-\log\epsilon}+\sqrt{\log n}\right)\,,
\end{align}
where $c_3{>0}$ is a constant depending on $d$ and the diameter of $M$. Similarly, by using Lemma \ref{covering number Q epsilon}, we can show that 
\begin{align}
& \sup_{f \in \epsilon^{2d} u\mathcal{Q}_{\epsilon}}|P_nf-Pf| \leq  \frac{c_4}{\sqrt{n}}\left( \sqrt{-\log\epsilon}+\sqrt{\log n}\right), \\
& \sup_{f \in \epsilon^{4d} \mathcal{Q}_{\epsilon} \mathcal{Q}_{\epsilon}}|P_nf-Pf| \leq  \frac{c_5}{\sqrt{n}}\left( \sqrt{-\log\epsilon}+\sqrt{\log n}\right),
\end{align}
where $c_4{>0}$ and $c_5{>0}$ are constants depending on $d$, the diameter of $M$, $\mathsf p_m$, and the $C^0$ norm of $\mathsf p$. Hence, we have 
\begin{align}
& \sup_{f \in u\mathcal{Q}_{\epsilon}}|P_nf-Pf| \leq  \frac{c_4}{\sqrt{n} \epsilon^{2d}}\left( \sqrt{-\log\epsilon}+\sqrt{\log n}\right), \\
& \sup_{f \in \mathcal{Q}_{\epsilon} \mathcal{Q}_{\epsilon}}|P_nf-Pf| \leq  \frac{c_5}{\sqrt{n}\epsilon^{4d}}\left( \sqrt{-\log\epsilon}+\sqrt{\log n}\right)\,.
\end{align}
The conclusion follows by taking $C_{gc}:=\max\{c_3,c_4,c_5\}$.
\end{proof}

\section{proof of Theorem \ref{spectral convergence of un L on closed manifold}}\label{proof of un GL rate}

{In addition to notations used in the proof of Theorem \ref{spectral convergence of L on closed manifold} that are summarized in Section \ref{Section:basic quantities for proof}, we need the following new notations.}
\begin{definition}
For any function $f(x) \in C(M) $, we define two operators from $C(M)$ to $C(M)$ with the $L^\infty$ norm:
$$T^{un}_{n,\epsilon}f(x)= \frac{P_{n}k_{\epsilon}f (x)}{d_{n,\epsilon}(x)},\quad 
T^{un}_{\epsilon}f(x)= \frac{Pk_{\epsilon}f (x)}{d_{\epsilon}(x)}.$$
\end{definition}

Based on the same argument {as that for} Lemma \ref{properties of T epsilon}, we have {the following lemma, which proof is omitted.}
\begin{lemma}\label{operator un T epsilon}
\begin{enumerate}
\item
For any $f \in L^2(M)$, ${T^{un}_{\epsilon}f(x)}=\frac{\int_{M}k_{\epsilon}(x,y)f(y)dV_M}{\int_{M}k_{\epsilon}(x,y)dV_M}$ is a smooth function on $M$. In particular, the eigenfunctions of {$T^{un}_{\epsilon}$} are smooth.
\item
For any function $f \in C^4(M)$, 
\begin{align}
\frac{T^{un}_{\epsilon}f(x)-f(x)}{\epsilon^2}=\Delta f(x)+O(\epsilon^2)\,, \nonumber 
\end{align}
where the implied constant in $O(\epsilon^2)$ depends on the $C^4$ norms of $f$ and {the volume} and the Ricci curvature of the manifold. 
\end{enumerate}
\end{lemma}

Suppose $(\lambda^{un}_{i,n,\epsilon}, \phi^{un}_{i,n,\epsilon})$ is the $i$-th eigenpair of $\frac{I-T^{un}_{n,\epsilon}}{\epsilon^2}$, $(\lambda^{un}_{i,\epsilon}, \phi^{un}_{i,\epsilon})$ is the $i$-th eigenpair of $\frac{I-T^{un}_{\epsilon}}{\epsilon^2}$, and $(\lambda_{i}, \phi_{i})$ is the $i$-th eigenpair of $-\Delta$. Assume $\phi^{un}_{i,n,\epsilon}$, $\phi^{un}_{i,\epsilon}$ and $\phi_{i}$ are normalized in the $L^2$ norm. Hence, we have the following proposition about the spectral convergence rate from $\frac{I-T^{un}_{\epsilon}}{\epsilon^2}$ to $-\Delta$. The proof follows by combining Lemma \ref{operator un T epsilon} and the same method as in Lemma \ref{property 1 of R epsilon}, Lemma \ref{property 2 of R epsilon} and Proposition \ref{T epsilon and Delta}.

\begin{proposition}\label{un T epsilon and Delta  of un GL}
Assume that the eigenvalues of $\Delta$ are simple. Let $\mathsf \Gamma_K:=\min_{1 \leq i \leq K}\textup{dist}(\lambda_i, \sigma(-\Delta)\setminus \{\lambda_i\})$. Suppose $\epsilon$ is small enough. Then, there are $a_i \in \{-1, 1\}$ such that for all $1 \leq i < K$, 
\begin{align}
|\lambda^{un}_{i,\epsilon}-\lambda_{i}|  \leq C_1 \epsilon^{2},\quad 
\|a_i\phi^{un}_{i,\epsilon}-\phi_{i}\|_{\infty}  \leq C_1 \epsilon^2, \,
\end{align}
where $C_1$ is a constant depending on $\mathsf \Gamma_K$, $\lambda_K$,  and the volume, the injectivity radius, the sectional curvature and the second fundamental form of the manifold. 
\end{proposition}

By using the similar argument as Proposition \ref{T epsilon and T n espilon 1}, we can show the convergence rate of $\lambda^{un}_{i, n,\epsilon}$ to $\lambda^{un}_{i,\epsilon}$ and  $\phi^{un}_{i,n,\epsilon}$ to $\phi^{un}_{i,\epsilon}$.

\begin{proposition}
Assume that the eigenvalues of $\Delta$ are simple. Let $\mathsf \Gamma_K:=\min_{1 \leq i \leq K}\textup{dist}(\lambda_i, \sigma(-\Delta)\setminus \{\lambda_i\})$. Fix $\epsilon$ small enough. If  $n$ is large enough so that $\frac{1}{\sqrt{n}\epsilon^d} ( \sqrt{-\log\epsilon}+\sqrt{\log n})< \mathcal{C}_1$, where $\mathcal{C}_1$ depends on $d$ {and the volume and} the diameter of $M$, {then there are $a_i \in \{1,-1\}$} such that with probability greater than $1-n^{-2}$,  for all $1 \leq i < K$, 
we have
{\begin{align}
& |\lambda^{un}_{i,\epsilon}-\lambda^{un}_{i, n,\epsilon}|\leq   \frac{C_{2}}{\sqrt{n}\epsilon^{d+4}}\left( \sqrt{-\log\epsilon}+\sqrt{\log n}\right) , \nonumber  \\
& \|a_i\phi^{un}_{i,n,\epsilon}-\phi^{un}_{i,\epsilon}\|_{\infty} \leq  \frac{C_{2}}{\sqrt{n}\epsilon^{d+2}}\left( \sqrt{-\log\epsilon}+\sqrt{\log n}\right) , \nonumber 
\end{align}}
where $C_{2}>0$ is a constant depending on $\mathsf \Gamma_K$, $\lambda_K$,  $d$ {and the volume and} the diameter of $M$.
\end{proposition}

Next, if we impose the relation between $\epsilon$ and $n$ we have the following corollary.

\begin{corollary}\label{un T epsilon and un T n espilon  of un GL}
Assume that the eigenvalues of $\Delta$ are simple. Fix $\epsilon$ small enough. Let $\mathsf \Gamma_K:=\min_{1 \leq i \leq K}\textup{dist}(\lambda_i, \sigma(-\Delta)\setminus \{\lambda_i\})$. {For a sufficiently small $\epsilon>0$,  if $n$ is sufficiently large so that $\epsilon=\epsilon(n)\geq\big(\frac{\log n}{n}\big)^{\frac{1}{2d+12}}$,  then} with probability greater than $1-n^{-2}$,  for all $1 \leq i < K$, 
we have
\begin{align}
& |\lambda^{un}_{i,\epsilon}-\lambda^{un}_{i, n,\epsilon}|\leq  2C_{2} \epsilon^2. 
\end{align} 
For a sufficiently small $\epsilon>0$,  if $n$ is sufficiently large {so that $\epsilon=\epsilon(n)\geq\big(\frac{\log n}{n}\big)^{\frac{1}{2d+8}}$,  then} there are $a_i \in \{1,-1\}$ such that with probability greater than $1-n^{-2}$,  for all $1 \leq i < K$, 
\begin{align}
\|a_i\phi^{un}_{i,n,\epsilon}-\phi^{un}_{i,\epsilon}\|_{\infty} \leq 2C_{2} \epsilon^2 .\nonumber 
\end{align}
$C_{2}>0$ is a constant depending on  $\mathsf \Gamma_K$, $\lambda_K$,  $d$ {and the volume and} the diameter of $M$.
\end{corollary}

{Finally, we need to control the} error in renormalizing the eigenvector $\tilde{v}_{i,n,\epsilon}$ by $\|\cdot\|_{l^2(1/\hat{\mathsf p})}$ norm in the following proposition. The proof is the same as the proof of Proposition \ref{relation between W and Q} and Proposition \ref{Convergence of normalized eigenvectors}.

\begin{proposition}\label{Convergence of normalized eigenvectors of un GL}
Let $\mathsf \Gamma_K:=\min_{1 \leq i \leq K}\textup{dist}(\lambda_i, \sigma(-\Delta)\setminus \{\lambda_i\})$. For a sufficiently small $\epsilon>0$,  if $n$ is sufficiently large {so that $\epsilon=\epsilon(n)\geq\big(\frac{\log n}{n}\big)^{\frac{1}{2d+8}}$,  then} with probability greater than $1-n^{-2}$,  for all $1 \leq i < K$, 
we have
\begin{align}
\max_{x_j} \Big|\frac{\tilde{v}_{i,n,\epsilon}(j)}{\|\tilde{v}_{i,n,\epsilon}\|_{l^2(1/\hat{\mathsf p})}}-\phi^{un}_{i,n,\epsilon}(x_j)\Big| \leq C_3 \epsilon^2 \,,\nonumber 
\end{align}
where the constant $C_3>0$ is a constant depending on $\mathsf \Gamma_K$, $\lambda_K$, $d$ {and the volume and} the diameter of $M$.
\end{proposition}

Finally, in Theorem \ref{spectral convergence of un L on closed manifold}, the proof of the cases $1 \leq i <K$ follows by combining Proposition \ref{un T epsilon and Delta  of un GL}, Corollary \ref{un T epsilon and un T n espilon  of un GL} and Proposition \ref{Convergence of normalized eigenvectors of un GL}. For the case $i=0$, the proof is the same as the case $i=0$ in Theorem \ref{spectral convergence of L on closed manifold}.
}

\end{document}